\documentclass{amsproc}
\usepackage{amssymb}
\usepackage[matrix, arrow, curve]{xy}
\usepackage{enumerate}
\newtheorem{theorem}{Theorem}[section]
\newtheorem{lemma}[theorem]{Lemma}
\newtheorem{prop}[theorem]{Proposition}
\newtheorem{cor}[theorem]{Corollary}  

\newtheorem{definition}[theorem]{Definition}

\theoremstyle{definition}
\newtheorem{rem}[theorem]{Remark}
\newtheorem{example}[theorem]{Example}

\newcommand{\Prob}{{\mathbb P}}
\newcommand{\E}{\mathbb E}
\newcommand{\pa}{\mathcal P_{\alpha}}
\newcommand{\fa}{\mathcal F_{\alpha}}
\newcommand{\EE}{\mathcal E}
\newcommand{\FF}{\mathcal F}
\newcommand{\HH}{\mathcal H}

\newcommand{\pb}{\mathcal P_{\beta}}

\newcommand{\pw}{\mathcal P_{w}}

\newcommand{\PP}{\mathcal P}

\newcommand{\qb}{\mathcal Q_{\beta}}

\newcommand{\Q}{\mathbb Q}

\newcommand{\R}{\mathbb R}
\newcommand{\N}{\mathbb N}

\newcommand{\fij}{\phi_{\bf i}^{\bf j}}

\theoremstyle{remark}

\numberwithin{equation}{section}

\begin{document}

\title[Continuous crystal and Coxeter groups]{Continuous crystal and Duistermaat-Heckmann measure for  Coxeter
groups.}

\author{Philippe Biane}
\address{CNRS, IGM, Universit\'e Paris-Est,
77454 Marne-la-Vall\'ee Cedex 2
}
\email{Philippe.Biane@univ-mlv.fr}
\author{ Philippe Bougerol}
\address{Laboratoire de Probabilit\'es et mod\`eles al\'eatoires,  
Universit\'e Pierre et Marie Curie, 4, Place Jussieu, 75005 Paris,  
FRANCE}
\email{philippe.bougerol@upmc.fr}
\author{ Neil O'Connell}
\address{Mathematics Institute,
University of Warwick,
Coventry CV4 7AL, UK}
\email{n.m.o-connell@warwick.ac.uk}
\date{\today}
\thanks{ Research of the third author supported in part by Science Foundation  
 Ireland
 Grant No. SFI04/RP1/I512.}

\begin{abstract}We introduce a notion of continuous crystal analogous, for general
Coxeter groups, to the combinatorial   crystals introduced by Kashiwara in
representation theory of Lie algebras.
 We explore their main properties in the case of finite Coxeter
groups, where we use a generalization of the Littelmann path model to show the
existence of  the
crystals. We introduce a remarkable measure, analogous to the
Duistermaat-Heckman measure, which we interpret in terms of Brownian motion.
We also  show that  the Littelmann path operators can be derived from simple
considerations on Sturm-Liouville equations.
\end{abstract}
\maketitle
\section{Introduction}
\subsection{}
The aim of this paper is to introduce a notion of continuous crystals for Coxeter
groups, which are not necessarily Weyl groups. Crystals are
combinatorial objects,
which have been  associated by Kashiwara to Kac-Moody algebras, in order to 
provide a combinatorial model for the representation theory of these algebras, see e.g.
\cite{H-K}, \cite{Joseph-book}, \cite{Joseph-notes},  \cite{kashbook} for an
introduction to this theory. The crystal graphs defined by Kashiwara turn out to be
equivalent to certain other graphs, constructed independently by Littelmann, using his
path model. The approach of Kashiwara to the crystals is through representations of
quantum groups and their ``crystallization'', which is the process of letting the
parameter $q$ in the quantum group go to zero. This requires representation theory and
therefore does not make sense for realizations of arbitrary Coxeter groups.
On the other hand, as it was realized in a previous paper \cite{bbo},
 Littelmann's model can be adapted to fit with non-crystallographic
Coxeter groups, but the price to pay is that, since there is no lattice invariant under
the action of the group, one can only define a continuous version of the path model,
namely of the Littelmann path operators (see however the recent preprint
\cite{Joseph-penta}, which has appeared when this paper was under revision). In this continuous model, instead of  
 the Littelmann path
operators $e_i,f_i$ we have continuous semigroups $e^t_i,f^t_i$ indexed by
nonnegative real numbers $t\geq 0$. In the crystallographic case it is possible to
think of these continuous crystals as ``semi-classical limits" of the combinatorial
crystals, in much the same way as the coadjoint orbits arise  as
 semi-classical limits of the
representations of a compact semi-simple Lie group. 
These continuous path operators, and the
closely related Pitman transforms, were used in \cite{bbo} to investigate symmetry 
properties of Brownian
motion in a space where a finite  Coxeter group acts, with applications in particular
to the motion of eigenvalues of  matrix-valued Brownian motions.
In this paper, which is a sequel to \cite{bbo}, but can for the most part be read
independently,  we define continuous  crystals and start investigating their main
properties.  As for now the theory works well for finite Coxeter groups, but there are
still several difficulties to extend it to infinite groups. This
theory allows us to define objects which are analogues to simplified versions of 
the Schubert varieties (or Demazure-Littelmann modules)
 associated with semi-simple Lie groups. We hope these objects
might help in certain questions concerning Coxeter groups, such as, for example,
the Kazhdan-Lusztig polynomials.
\subsection{}
 This paper is organized as follows.
 The next section contains the main definition, that 
  of a continuous crystal associated with
 a realization of a Coxeter group. We establish the main properties of these objects,
 following closely the exposition of Joseph in \cite{Joseph-notes}. It would have been
 possible to just refer to \cite{Joseph-notes} for the most part of this section,
 however, for the convenience of the reader, and also for convincing ourselves that
 everything from the crystallographic situation goes smoothly to the continuous
 context, we have preferred to write everything down. The main body of the proof
 is relegated to an appendix in order to  ease the reading of the paper.
The main  result of this section is theorem \ref{thmuniq}, a uniqueness result for 
continuous crystals, analogous to the one in \cite{Joseph-notes}.
 In section 3 we introduce the
 path operators and establish their most important properties.
 Our approach to the path model is different from that in Littelmann
 \cite{littel},
 in that we base our exposition on the Pitman transforms,
  which are defined from scratch. These transforms satisfy 
  braid relations, which where proved in  \cite{bbo}, and which play a prominent role.
Using these operators, the set of continuous paths is endowed with a crystal structure and the 
continuous analogues of the Littelmann  modules are introduced as "connected components" of this crystal (see  the discussion following  proposition \ref{prop:exis}, definition \ref{dem-littel} and
 theorem \ref{thm:exis}).
Our definition makes sense for arbitrary Coxeter groups, but we are able to prove significant properties of these only in the case of finite Coxeter groups. It
 remains an interesting and challenging problem to extend these properties to the
 general case.
Continuous  Littelmann modules can be parameterized in several ways by  polytopes, corresponding to different reduced decompositions of an element in the Coxeter group.
In the case of Weyl groups, these are  the  
 Berenstein-Zelevinsky polytopes (see \cite{beze2})
 which contain the Kashiwara coordinates on the crystals. In section 4 
we state some properties of these parametrizations. In theorem \ref{piecewise} we prove that two such parametrizations are related by a piecewise linear transformation,  and in theorem
\ref{theo_poly} we show that the polytopes can be obtained by the intersection of a cone 
depending only on the element of the Coxeter group, and a set of inequalities which depend on the dominant path. Furthermore, we provide explicit equations for the cone in the dihedral case (in proposition \ref{dihedralcone}).
 In theorem \ref{uniq_iso} we prove that the crystal associated with a Littelmann module depends only on the end point of the dominant path, then in theorem \ref{uniq_closed}
we obtain the existence and uniqueness of a family of highest weight normal continuous crystals. We show that the Coxeter group acts on each Littelmann module (theorem \ref{theo:action}). 
 We introduce the Sch\"utzenberger involution in section \ref{schutz} and use it to give a direct combinatorial proof of the commutativity of the tensor product of continuous crystals (theorem \ref{crysiso}).
  We think that
 even in the crystallographic case our treatment sheds some light on these topics.
 In section 5, we introduce an analogue of the Duistermaat-Heckman measure, motivated by a
 result of Alexeev and Brion \cite{ab}. We prove several interesting properties of
 this measure, in particular, in theorem  \ref{p-2}, an analogue of the Harish-Chandra formula. 
The Laplace transform appearing in this formula is a generalized Bessel function. It is    shown in theorem \ref{prod_J} to satisfy a product formula, giving a positive answer to a question of R\"osler.
The Duistermaat-Heckman measure is intimately linked with Brownian motion,
 and in corollary \ref{cor_cone} we give a Brownian proof of the fact that the
 crystal defined by the path model depends only on the final position of the
 path. 
 The final section is of a quite different nature, and somewhat independent of the rest
 of the paper. The Littelmann path operators have been introduced as a generalization,
 for arbitrary root systems, of combinatorial operations on Young tableaux. Here we
 show how, using some simple considerations on Sturm-Liouville equations, the Littelmann
 path operators appear naturally. In particular this gives a concrete geometric basis
 to the theory of geometric lifting which has been introduced by Berenstein and
 Zelevinsky in \cite{beze2} in a purely formal way.

\section{Continuous crystal}
This section is devoted to introducing the main definition and first 
properties of continuous crystals.
  \subsection{Basic definition}
We use the standard references \cite{bo2}, \cite{humphreys} on Coxeter groups and their realizations.
 A Coxeter system $(W, S)$ is a group $W$ generated by a finite set of involutions $S$ such that, if $m(s,s')$ is the order of $ss'$ then  the relations 
  $$(ss')^{m(s,s')}=1$$
  for $m(s,s')$ finite, give a   presentation of $W$. 
  
  A realization of $(W,S)$ is given by a real vector space $V$ with dual
   $V^\vee$, an  action of $W$ on $V$,
    and a subset $\{(\alpha_s,\alpha^\vee_s), s \in S\}$ of
    $V \times V^\vee$ such that  each $s\in S$ acts on $V$ by
     the reflection given by
  $$s(x)=x-\alpha_s^\vee(x)\alpha_s,\;\; x \in V,$$ so $\alpha_s^\vee(\alpha_s)=2$.
One calls $\alpha_s$ the simple root associated with $s\in S$ and  
$\alpha_s^\vee$ its coroot.

 We consider a realization of a Coxeter system $(W,S)$ in a real vector space $V$,
and the associated simple roots $\Sigma=\{\alpha_s, s \in S\}$ in $V$ and coroots $\{\alpha^\vee_s, s \in S\}$ in $V^\vee$.
The closed Weyl chamber is the convex cone $$\overline   
C=\{v \in V; \alpha_s^\vee(v) \geq 0, \mbox{ for all } \alpha\in S\}$$
thus the simple roots are positive on $\overline C$.
There is an order relation on $V$ induced by this cone, namely
$\lambda\leq \mu$ if and only if $\mu-\lambda\in \overline C$.

 We adapt the definition of crystals due to Kashiwara (see, e.g.,  Kashiwara \cite{kash93}, \cite{kashbook}, Joseph \cite{Joseph-book}) to a continuous setting. 
 \begin{definition}
 A continuous crystal is a set $B$ equipped with maps
\begin{eqnarray*}wt&:&B\to V,\\
\varepsilon_\alpha, \varphi_\alpha&:&B \to \R\cup\{-\infty\},\, \alpha \in \Sigma,\\
 e_\alpha^r&:& B\cup\{{\bf 0}\}\to B\cup\{{\bf 0}\}, \, \alpha \in \Sigma, r \in \R,
\end{eqnarray*}
 where $\bf 0$ is a ghost element, such that the following properties 
 hold, for all $\alpha \in \Sigma$, and $b\in B$
\medskip

(C1) $\varphi_\alpha(b)=\varepsilon_\alpha(b)+\alpha^\vee(wt(b)),$
\medskip

(C2) If  $ e_\alpha^r(b)\not={\bf 0}$ then
\begin{eqnarray*}
\varepsilon_\alpha( e_\alpha^rb)&=&\varepsilon_\alpha(b)-r,\\
\varphi_\alpha( e_\alpha^rb)&=&\varphi_\alpha(b)+r,\\
wt( e_\alpha^rb)&=&wt(b)+r\alpha,
\end{eqnarray*}

 (C3)  For all $r \in \R, b\in B$ one has $e^r_\alpha({\bf 0})={\bf 0},e^0_\alpha(b)=b.$
If  $ e_\alpha^r(b)\not={\bf 0}$ then, for all $s\in \R$,
  $$ e_\alpha^{s+r}(b)= e_\alpha^s  (e_\alpha^r(b)),$$

(C4) If $\varphi_\alpha(b)=-\infty$ then $e_\alpha^r(b)={\bf 0}$, for all $r \in \R, r�\neq  0.$
\medskip
\end{definition}

The point is that, in this definition, $r$ takes any real value, and not only discrete ones. Sometimes we write, for $r \geq 0$,
$$f_\alpha^r=e_\alpha^{-r}.$$ 
\begin{example}[The crystal $B_\alpha$] \label{exba}For each $\alpha \in \Sigma$, we define the crystal  $B_\alpha$ as the set $\{b_\alpha(t), t \mbox{  is a nonpositive real number}\}$,  with the maps given by
$$wt(b_\alpha(t))=t\alpha,\quad \varepsilon_\alpha(b_\alpha(t))=-t, \quad \varphi_\alpha(b_\alpha(t))=t,$$ 
$$e_\alpha^r(b_\alpha(t))=b_\alpha(t+r) \mbox{ if } r \leq -t \mbox{ and }e_\alpha^r(b_\alpha(t)) = {\bf 0} \mbox{ otherwise,}$$
and, if $\alpha ' \neq \alpha$,
 $\varepsilon_{\alpha'}(b_\alpha(t))=-\infty, \,\varphi_{\alpha'}(b_\alpha(t))=-\infty,\,e^r_{\alpha'}(b_\alpha(t))={\bf 0},$  when $r \neq 0 $.
\end{example}
\subsection{Morphisms}

\begin{definition}
Let $B_1$ and $B_2$ be continuous crystals. 

1. A  morphism of crystals $\psi: B_1 \to B_2$ is a map $\psi: B_1\cup \{{\bf 0}\} \to B_2 \cup \{\bf 0\}$ such that $\psi({\bf 0})={\bf 0}$ and for all $\alpha \in \Sigma$ and $b \in B_1$,
$$
wt(\psi(b))=wt(b),\,  \varepsilon_\alpha(\psi(b))=\varepsilon_\alpha(b),\,
\varphi_\alpha(\psi(b))=\varphi_\alpha(b)
$$
and 
$e_\alpha^r(\psi(b))=\psi(e_\alpha^r(b))$ when $e_\alpha^r(b)\in B_1$.

2. A strict morphism is a morphism $\psi: B_1\to B_2$ such that $e_\alpha^r(\psi(b))=\psi(e_\alpha^r(b))$ for all $b\in B_1$.

3. A crystal embedding is an injective strict  morphism.
\end{definition}
The morphism
 $\psi$ is called a {\it crystal isomorphism}
if there exists a crystal morphism $\phi: B_2\to B_1$ such that $\phi\circ \psi=id_{B_1\cup\{\bf 0\}}$, and $\psi\circ
\phi=id_{B_2\cup\{\bf 0\}}$. It is then an embedding.

\subsection{Tensor product}\label{subtens}

Consider two continuous crystals $B_1$ and $B_2$ associated with $(W,S,\Sigma)$. We  define the tensor product $B_1\otimes B_2$ as the continuous crystal with set $B= B_1\times B_2$, whose elements are denoted
$b_1\otimes b_2,\text{for}\, b_1\in B_1,b_2\in B_2$. 
Let $\sigma=\varphi_\alpha(b_1)-\varepsilon_\alpha(b_2)$ where $(-\infty)-(-\infty)=0$, let $\sigma^+=\max(0,\sigma)$ and $\sigma^-=\max(0,-\sigma)$, then the maps defining the tensor product are given by the following formulas:
\begin{eqnarray*}wt(b_1\otimes b_2)&=&wt(b_1)+wt(b_2)\\\varepsilon_\alpha(b_1\otimes b_2)&=&\varepsilon_\alpha(b_1)+\sigma^-\\
\phi_\alpha(b_1\otimes b_2)&=&\phi_\alpha(b_2)+\sigma^+\\
e_\alpha^{r}(b_1\otimes b_2)&=&
 e_\alpha^{\max(r,-\sigma)-\sigma^-}b_1\otimes  e_\alpha^{\min(r,-\sigma)+\sigma^+} b_2, 
  \end{eqnarray*}
Here $b_1\otimes {\bf 0}$ and ${\bf 0}\otimes b_2$ are understood to be ${\bf 0}$. Notice that when $\sigma \geq 0$, one has $\varepsilon_\alpha(b_1\otimes b_2)=\varepsilon_\alpha(b_1)$ and
\begin{equation}\label{sigmapos}
e_\alpha^r(b_1\otimes b_2)=
 e_\alpha^{r}b_1\otimes b_2, \mbox{ for all } r \in [-\sigma,+\infty[.
\end{equation}
As in the discrete case, one can check that the tensor product is associative (but not commutative) so we can define without ambiguity the tensor product of several crystals.
\subsection{Highest weight crystal}
A crystal $B$ is called upper normal when, for all $b\in B$,
$$\varepsilon_\alpha(b) = \max\{r \geq  0 ; e_\alpha^r(b)\not={\bf 0}\}$$ and is called lower normal if
$$\varphi_\alpha(b)=\max\{r \geq  0 ; e_\alpha^{-r}(b)\not={\bf 0}\}.$$
We  call it normal (this is sometimes called seminormal by Kashiwara) when it is lower and upper normal. Notice that this implies that $\varepsilon_\alpha(b) \geq 0$ and $ \varphi_\alpha(b) \geq 0$.
\medskip

We introduce  the semigroup $\FF$ generated by the $\{f_\alpha^r, \alpha \mbox { simple root}, r \geq 0\}$:
$$\FF=\{f_{\alpha_1}^{r_1}\cdots f_{\alpha_k}^{r_k} , k \in \N^*, r_1,\cdots,r_k \geq 0, \alpha_1,\cdots , \alpha_k\in \Sigma\},$$
and, if $b$ is an element of a continuous crystal $B$,  the subset 
$\FF(b)=\{f (b), f\in \FF\}$ of $B$.

\begin{definition}
Let $\lambda \in V$, a continuous crystal $B(\lambda)$ is said to be of highest weight $\lambda$ if there exists $b_\lambda\in B(\lambda)$ such that $wt(b_\lambda)=\lambda$,
 $e_\alpha^r(b_\lambda)={\bf 0}$, for all $r > 0$ and $\alpha \in \Sigma$ and such that $B(\lambda) =\FF(b_\lambda).$
\end{definition}

For a continuous crystal with highest weight $\lambda$,
such an element $b_\lambda$ is unique, 
and called the primitive element of $B(\lambda)$. 
If the crystal is normal then $\lambda$ must be 
in the Weyl chamber $\bar C$.  
The vector $\lambda$ is a highest weight in the sense that, 
for all $b \in B(\lambda)$, $wt(b) \leq \lambda$.

\subsection{Uniqueness.}
Following Joseph \cite{Joseph-book}, \cite{Joseph-notes} we introduce the following definition.

\begin{definition} Let $(B(\lambda), \lambda \in \bar C),$ be  a family
 of highest weight continuous crystals.
The family  is closed if, for each $\lambda, \mu \in \bar C$, the subset $\FF(b_\lambda\otimes b_\mu)$ of $B(\lambda)\otimes B(\mu)$ is a crystal isomorphic to 
$B(\lambda+\mu)$.
\end{definition}
Joseph (\cite{Joseph-book}, 6.4.21) has shown in the Weyl group case,
 for discrete crystals, that a closed family of highest weight
  normal crystals is unique.  
 The analogue holds in our situation. 
\begin{theorem}\label{thmuniq} For a realization of a Coxeter system 
$(W,S)$, if a closed family $B(\lambda), \lambda \in \bar C,$
 of highest weight continuous normal
crystals exists, then it is unique.\end{theorem}

The proof of the theorem, which follows closely
Joseph \cite{Joseph-notes}, is in the appendix \ref{prfuniq}.

\section{Pitman transforms and Littelmann path operators
 for Coxeter groups}
 In this section we recall definition and properties of Pitman transforms, introduced in our previous paper
\cite{bbo}. We deduce from these properties the existence of Littelmann operators,  then we define continuous Littelmann modules, prove that they are  continuous crystals, and make a first study of their parametrization.

\subsection{The Pitman transform}

Let $V$ be a real vector space, with dual space $V^{\vee}$.
Let $\alpha\in V$
and $\alpha^{\vee}\in V^{\vee}$ be such that $\alpha^{\vee}(\alpha)=2$.
 The reflection $s_\alpha:V \to V$ associated to $(\alpha,\alpha^\vee)$ is the 
 linear map
defined, for $x \in V$,  by
$$s_\alpha(x)=x- \alpha^{\vee}(x)\alpha.$$
For $T > 0$, let $C_T^0(V)$ be
 the set of continuous path $\eta:[0,T]\to V$ such that 
 $\eta(0)=0$, with the topology of uniform convergence. We have introduced
  and studied in \cite{bbo} the following path transformation, similar to the 
  one defined by Pitman in \cite{pitman}.
\begin{definition}\label{pitman-transform}The Pitman
    transform  $\pa$ associated with $(\alpha,\alpha^\vee)$
    is defined on $C_T^0(V)$ by the formula:
$$\pa \eta(t)=\eta(t)-\inf_{t\geq s\geq 0}\alpha^{\vee}(
\eta(s))\alpha,\qquad
T\geq t\geq 0.$$
\end{definition}
 A path  $\eta \in C_T^0(V)$ is called $\alpha$-dominant
  when $\alpha^\vee(\eta (t))\geq 0$ for all $t \in [0,T]$. 
  The following properties of the Pitman transform are easily established.
    
\begin{prop}\label{pit}

({\it i}) The transformation $\pa:C_T^0(V) \to C_T^0(V)$ is continuous.

({\it ii}) For all $\eta \in C_T^0(V)$, the path
 $\pa \eta $  is $\alpha$-dominant and   
  $\pa \eta =\eta $ if and only if  $\eta $ is $\alpha$-dominant.

({\it iii}) 
  The transformation $\pa$ is an idempotent, i.e. $\pa\pa \eta =\pa \eta $
for all
$\eta \in C_T^0(V)$.
  
({\it iv)})
 Let $\pi\in C_T^0(V)$ be $\alpha$-dominant, and let $ x\in[0,  \alpha^\vee(\pi(T))]$, then 
 there exists a unique path $\eta$ in $C^0_T(V)$ such that
$\pa \eta =\pi$
and $\eta(T)=\pi(T)-x \alpha.$
Moreover for $0\leq t\leq
T,$
$$\eta(t)=\pi(t)-\min[x,\inf_{T\geq s\geq t}\alpha^\vee(\pi(s))]\alpha.$$
\end{prop}
\subsection{ Littelmann path operators }\label{littelpit}
Let $V,V^{\vee},\alpha,\alpha^{\vee}$ be as above.
Using proposition \ref{pit}, as in \cite{bbo}, we can define generalized
Littelmann path
operators (see \cite{littel}).

\begin{definition}\label{littelmanntransform}
Let $\eta \in C_T^0(V)$, and 
$x\in \mathbb R$, then  we define
$\EE_{\alpha}^x\eta$ as the unique path such that $$\pa \EE_{\alpha}^x\eta=\pa  
\eta\quad
\text{and}\quad \EE_{\alpha}^x\eta(T)=\eta(T)+x\alpha 
$$
if $
    -\alpha^{\vee}(\eta(T))+\inf_{0\leq t\leq T}\alpha^{\vee}(
\eta(t))\leq x\leq -\inf_{0\leq t\leq T}\alpha^{\vee}(
\eta(t))$ and $\EE_{\alpha}^x\eta=\bf 0$ otherwise. The following formula holds
$$\EE_{\alpha}^x\eta(t)=
\eta(t)-\min(-x,\inf_{t\leq s\leq T}\alpha^\vee(\eta(s))
-\inf_{0\leq s\leq T}\alpha^\vee(\eta(s)))\alpha$$
if $\  -\alpha^\vee(T)+\inf_{0\leq t\leq T}
\alpha^{\vee}(
\eta(t))\leq x\leq 0$, and 
$$\EE_{\alpha}^x\eta(t)=
\eta(t)-
\min(0,-x-\inf_{0\leq s\leq T}\alpha^\vee(\eta(s))+
\inf_{0\leq s\leq t}\alpha^\vee(\eta(s)))\alpha$$
if $ 0\leq x\leq -\inf_{0\leq t\leq T}
\alpha^{\vee}(
\eta(t))$.
\end{definition}
Here, as in the definition of crystals, $\bf 0$ is a ghost element.
The following result is immediate from the definition of the Littelmann
operators.
\begin{prop}
    $\EE_{\alpha}^0\eta=\eta$ and
    $\EE_{\alpha}^x \EE_{\alpha}^y\eta=\EE_{\alpha}^{x+y}\eta$ as long as
    $\EE_{\alpha}^y\eta\ne \bf 0$.
    \end{prop}
We shall also use the notation $\fa^x=\EE_{\alpha}^{-x}$ for $x\geq 0$, and denote by
$\HH_{\alpha}^x$ the restriction of the operator $\fa^x$ to 
$\alpha$-dominant paths.
Let $\pi$ be 
 an $\alpha$-dominant path in $C_T^0(V)$
  and $0 \leq x \leq \alpha^\vee(T)$, then 
  $\HH^x_\alpha\pi$ is the unique path in $C^0_T(V)$ such that
$$\pa \HH^x_\alpha \pi =\pi$$
and $$\HH^x_\alpha \pi(T)=\pi(T)-x \alpha.$$
Observe that in this equality
$$x=-\inf_{0\leq t\leq T}\alpha^{\vee}(\HH^x_\alpha \pi(t)).$$

\subsection{Product of Pitman transforms}

Let $\alpha,\beta\in V$ and   $\alpha^{\vee},\beta^{\vee}\in V^{\vee}$
be such
that $\alpha^{\vee}(\beta)< 0$ and $\beta^{\vee}(\alpha)< 0$. Replacing if necessary
$(\alpha,\alpha^\vee,\beta,\beta^\vee)$ by 
$(t\alpha,\alpha^\vee/t,\beta/t,t\beta^\vee)$, which does not change $\pa$ and $\pb$,
 we  will assume that
$\alpha^{\vee}(\beta)=\beta^{\vee}(\alpha)$.
We use the notations
$
\rho=-\frac{1}{2}\alpha^{\vee}(\beta)=
-\frac{1}{2}\beta^{\vee}(\alpha)$.
The following result is proved in \cite{bbo}.
\begin{theorem}\label{formula} Let $n$ be a positive integer, then  if
$\rho\geq\cos\frac{\pi}{n}$, 
\begin{eqnarray}\label{forpapb}
(\underbrace{\pa\pb\pa\ldots}_{ \text{$n$ terms}} )
\pi(t)&=&\pi(t)-\inf_{t\geq s_0\geq s_1\geq \ldots\geq s_{n-1}\geq
0}\bigl(\sum_{i=0}^{n-1}T_i(\rho)Z^{(i)}(s_i)\bigr)\alpha\nonumber\\&&
-\inf_{t\geq s_0\geq s_1\geq \ldots\geq s_{n-2}\geq
0}\bigl(\sum_{i=0}^{n-2}T_i(\rho)Z^{(i+1)}(s_i)\bigr)\beta
\end{eqnarray}
where $Z^{(k)}(t)=\alpha^{\vee}( \pi(t))$ if $k$ is even and $Z^{(k)}(t)=\beta^{\vee}( \pi(t))$ if $k$ is odd. The
$T_k(x)$ are the Tchebycheff polynomials defined by
\begin{equation}\label{Tcheb}
T_0(x)=1,\, T_1(x)=2x, \,
2xT_k(x)=T_{k-1}(x)+T_{k+1}(x) \text{ for}\ k\geq 1.
\end{equation}
\end{theorem}
    The Tchebycheff polynomials satisfy
$T_k(\cos\theta)=\frac{\sin(k+1)\theta}{\sin\theta}$ and, in particular,
under the assumptions on $\rho$ and $n$, 
$T_k(\rho)\geq 0$ for all $k\leq n-1$.
An important property of the Pitman transforms is the following corollary
(see \cite{bbo}).
\begin{theorem}\label{braid}({\it Generalized braid relations for the Pitman
transforms.})
Let $\alpha,\beta\in V$ and   $\alpha^{\vee},\beta^{\vee}\in V^{\vee}$
     be  such that
$\alpha^{\vee}(\alpha)=\beta^{\vee}(\beta)=2$, and
$\alpha^{\vee}(\beta)< 0,
\beta^{\vee}(\alpha)< 0$ and
$\alpha^{\vee}(\beta)
\beta^{\vee}(\alpha)=4\cos^2\frac{\pi}{n}$, where $n\geq 2$
    is some  integer.
Then 
    $$\pa\pb\pa \ldots =\pb\pa\pb\ldots$$ where there are $n$
factors in each product.
\end{theorem}
\subsection{Pitman transforms for Coxeter groups}\label{sec_cox}
Let $(W,S)$ be a Coxeter system, with a realization in the space $V$.
For a simple reflection $s$, 
 denote by
$\PP_{\alpha_s}$ or $\PP_{s}$ the Pitman transform associated with  the pair  
$(\alpha_s,\alpha^\vee_s)$.
From theorem \ref{braid} and Matsumoto's lemma [\cite{bo2}, Ch. IV, No. 1.5.
Prop.5] we deduce (\cite{bbo}):
\begin{theorem}\label{braidP} Let $w=s_{1}\cdots s_{r}$ be a reduced decomposition of $w \in W$, with $s_1,\cdots,s_r \in S$.  Then
$$\PP_w:=\PP_{s_{1}}\cdots\PP_{s_{r}}$$
depends only on $w$ and not on the chosen decomposition.  
\end{theorem}

When $W$ is finite, it  has a unique longest element, denoted by $w_0$.
 The transformation $\PP_{w_0}$ plays a fundamental role in the sequel.
 The following result   is proved in \cite{bbo}.

\begin{prop} When $W$ is finite, for any path $\eta \in C_T^0(V)$, the path 
$\PP_{w_0}\eta$ takes values in the closed Weyl chamber $\overline   
C$. Furthermore
    $\PP_{w_0}$ is an idempotent and  
$\PP_w\PP_{w_0}=\PP_{w_0}\PP_w=\PP_{w_0}$ for all $w \in W$.
\end{prop}

\subsection{The continuous cristal $C_T^0(V)$}

For any path $\eta$ in $C^0_T(V)$, 
 let
  $wt(\eta)=\eta(T)$. Let  $e_\alpha^r$ be the generalized Littelmann operator 
  $\EE_{\alpha}^r$ defined in Definition \ref{littelmanntransform}, and
  $$\varepsilon_\alpha(\eta) = \max\{r \geq  0 ; \EE_\alpha^r(\eta)\not=0\}
  =-\inf_{0 \leq t \leq T} \alpha^\vee(\eta(t))$$ 
$$\varphi_\alpha(\eta)=\max\{r \geq  0 ; \EE_\alpha^{-r}(\eta)\not=0\}
=\alpha^\vee(\eta(T))-\inf_{0 \leq t \leq T} \alpha^\vee(\eta(t)).$$
It is clear that
  \begin{prop} \label{prop:exis}
   With the above definitions,
 $C_T^0(V)$ is a  normal continuous crystal.  \end{prop}
   We say that a path is dominant 
  if it takes its values in the closed
   Weyl chamber $\overline   C$.   
\begin{definition} \label{dem-littel} Let $\pi \in C_T^0(V)$ be a dominant path, and $w\in W$. 
  We define $$L_\pi^w=\{\eta  \in C_T^0(V);  \PP_{w}\eta=\pi\}.$$
   \end{definition}
   These sets are  defined for arbitrary  
  Coxeter groups. We shall establish their main properties in the case of finite
  Coxeter groups, where they are analogues of Demazure-Littelmann modules.
   It remains an interesting problem to establish similar
  properties in the general case.
  
  From now on we assume that $W$ is finite, with longest element 
  $w_0$, and we denote
  $L_\pi=L^{w_0}_\pi$, which we call the Littelmann module associated with $\pi$.
The set $L_\pi\cup\{{\bf 0}\}$ is a subset of $C_T^0(V)\cup\{{\bf 0}\}$invariant under  the Littelmann operators, thus:  \begin{theorem} \label{thm:exis}
For any dominant path $\pi$,
  $L_\pi$ is a  normal continuous crystal with highest weight $\pi(T)$.
  \end{theorem}
  \proof 
  This follows from the result of \ref{sec_cox}, except the highest
  weight property, which follows from the fact that, see (\ref{xy}),
  any  $\eta\in L_\pi$ can be written as 
  $$\eta=
  \HH_{s_{q}}^{x_q}\HH_{s_{q-1}}^{x_{q-1}}\cdots\HH_{s_{1}}^{x_1}\pi.\qed$$ 

Two paths $\eta_1$ and $\eta_2$ are said to be connected if there exists simple roots $\alpha_1,\cdots, \alpha_k$ and real numbers $r_1,\cdots, r_k$ such that
$$ \eta_1=\EE_{\alpha_1}^{r_1}\cdots \EE_{\alpha_k}^{r_k}\eta_2.$$
This is equivalent with the relation $\PP_{w_0}\eta_1=\PP_{w_0}\eta_2$. A connected set in $C_T^0(V)$ is a subset in which each two elements are connected.
We see that the sets $\{L_\pi, \pi \mbox { dominant}\}$ are the connected components in 
 $C_T^0(V)$. Moreover 
we will show in theorem \ref{uniq_iso} that the continuous crystals  $L_{\pi_1}$ and $L_{\pi_2}$ are isomorphic if and only if  $\pi_1(T)=\pi_2(T)$.   
\subsection{Braid relations for the $\HH$ operators}

Let $w \in W$ and fix a reduced decomposition $w=s_{ 1}\ldots s_{p}$. 
For any  path $\eta$ in  $C_T^0(V)$,  denote $\eta_p=\eta$ and for 
$k=1,\ldots,p$,
    $$\eta_{k-1}=\PP_{s_{{k}}}\ldots \PP_{s_{p}}\eta.$$ 
    Then $\eta_{k-1}=\PP_{s_{{k}}}\eta_{k}$ is $\alpha_{s_{k}}$-dominant, by
    proposition \ref{pit} ({\it ii}) and 
   $$ \eta_{k}=\FF_{s_{k}}^{x_k}\eta_{k-1}=\HH_{s_{k}}^{x_k}\eta_{k-1}$$
   where
 \begin{equation}\label{kash}x_k=-
    \inf_{0\leq t\leq T}\alpha_{s_{k}}^{\vee}( \eta_k(t)).
    \end{equation}
 Observe that   \begin{equation}\label{inegx}x_k\in
    [0,\alpha_{s_{k}}^{\vee}( \eta_{k-1}(T))] \end{equation} 
    and
   $$ \eta_{k}(T)= \eta_{k-1}(T)-x_k\alpha_{s_{k}};$$
   thus,  
  $$ \eta_{k}(T)=\eta_0(T)-\sum_{i=1}^k x_i\alpha_{s_{i}}.$$
 Furthermore,
    \begin{equation}\label{xy}
    \eta_k=\HH_{s_{k}}^{x_k}\HH_{s_{{k-1}}}^{x_{k-1}}
    \cdots \HH_{s_{1}}^{x_1}\PP_{w}\eta,
    \end{equation}
   and the numbers $(x_1,\ldots, x_k)$ are uniquely determined by this equation.
   
   We consider two reduced decompositions
$$w=s_{1}\cdots s_{p}, w=s'_{1}\cdots s'_{p}$$
    of $w$. Let 
    ${\bf i}=(s_1,\cdots,s_p)$ and ${\bf j}=(s'_1,\cdots,s'_p)$.
    Let $\eta:[0,T]\to V$ be a continuous path such that $\eta(0)=0$, 
    and let $(x_1,\ldots, x_p)$, respectively $(y_1,\ldots, y_p)$, be
    the numbers determined by equation (\ref{xy}) for the two decompositions
    ${\bf i}$ and ${\bf j}$.
The following 
 theorem states that the correspondence between
  the $x_n$'s and the $y_n$'s actually does not depend on the path $\eta$.
In other words, we have the following braid relation for the operators $\HH$.
\begin{equation}
\HH_{s_{p}}^{x_p}\cdots \HH_{s_{2}}^{x_2}\HH_{s_{1}}^{x_1}=\HH_{s'_{p}}^{y_p}\cdots 
\HH_{s'_{2}}^{y_2}\HH_{s'_{1}}^{y_1}.
\end{equation}
    \begin{theorem}\label{piecewise}
      There exists a piecewise linear continuous map $\fij:\R^p\to \R^p$ such that for all 
  paths $\eta \in C_T^0(V)$, 
        $$(y_1,\cdots, y_p)=\fij(x_1,\cdots,x_p).$$
    \end{theorem}
    \proof
    First step: 
    If the roots $\alpha,\beta$ generate a system of type $A_1\times A_1$
    and $w=s_\alpha s_\beta=s_\beta s_\alpha$, then $\pa$ and $\pb$ commute,
    and 
    it is immediate that $x_1=y_2$, $x_2=y_1$.
   Let $\alpha,\alpha^\vee$ and $\beta,\beta^\vee$ be such that
  
$$\alpha^\vee(\alpha)=\beta^\vee(\beta)=2,\quad
\alpha^\vee(\beta)=\beta^\vee(\alpha)=-1,$$
then 
    $\alpha$ and $\beta$ generate a root system of type $A_2$
     and the braid relation is
    $$w_0=s_\alpha s_\beta s_\alpha=s_\beta s_\alpha s_\beta.$$ 
Define 
$$a\wedge b=\min(a,b),\quad a\vee b=\max(a,b).$$
    We prove that the following map
     \begin{equation}\label{transition}
  \begin{array}{ll}
  x_1=(y_2-y_1)\wedge y_3&\qquad y_1=(x_2-x_1)\wedge x_3\\
  x_2=y_1+y_3&\qquad y_2=x_1+x_3\\
  x_3=y_1\vee (y_2-y_3)&\qquad y_3=x_1\vee (x_2-x_3)\\
  \end{array}
  \end{equation}
  satisfies the required properties.
  Assume that, for $\pi=\mathcal P_{w_0}\eta$,  
  $$\eta=\HH^{x_3}_\alpha\HH^{x_2}_\beta\HH^{x_1}_\alpha \pi.$$
  Then define $\eta_2=\pa\eta, \eta_1=\pb\pa\eta,
  \eta_0=\pi=\pa\pb\pa\eta$.
   Using theorem \ref{formula} for computing the paths $\eta_i$ one gets
  the explicit formulas.
   $$
  \begin{array}{rcl}
  x_3&=&-\inf_{0\leq s\leq T}\alpha^{\vee}(\eta(s))\\
  x_2&=&-\inf_{0\leq s_2\leq s_1\leq T}
  \left(\beta^{\vee}(\eta(s_1))+\alpha^{\vee}(\eta(s_2))\right)\\
  x_1&=&-\inf_{0\leq s_2\leq s_1\leq T}
  \left(\alpha^{\vee}(\eta(s_1))+\beta^{\vee}(\eta(s_2))\right)-x_3.
  \end{array}
  $$
  Similar formulas are obtained for the $y_i$ coming from the other reduced
  decomposition, by exchanging the roles of $\alpha$ and $\beta$.  The formula 
  (\ref{transition}) follows by inspection. 
  
  In the context of crystals,    this result is well known and first 
    appeared   in  Lusztig \cite{luszt} and Kashiwara \cite{kash93}.
 We observe that it can also be obtained from
 the considerations of section 6,
see e.g. \ref{redbru}.

 Second step: When the roots generate 
 a root system of type $A_n$,  
 using Matsumoto's lemma, one can pass from one reduced
  decomposition to another by a sequence of 
 braid relations corresponding to the two cases of the first step.
 \medskip
 
 Third step: We consider now the case where the
  roots generate the dihedral group  $I(m)$, and 
  $w=s_{\alpha}s_\beta...=s_{\beta} s_{\alpha}...$ is the longest element in 
  $W$. 
We will use an embedding of the dihedral group $I(m)$
 in the Weyl group of the system $A_{m-1}$, 
 see e.g. Bourbaki \cite{bo2}, ch.\  V, 6, Lemme 2. 
Recall the Tchebicheff polynomials $T_k$
defined in (\ref{Tcheb}).
Let $\lambda=\cos(2\pi/m)$, 
$a_1=a_2=1$ and, for $k \geq 1$, $$a_{2k}=T_{k-1}(\lambda),\quad
 a_{2k+1}=T_{k}(\lambda)+T_{k-1}(\lambda)$$
then, 
\begin{equation}\label{rec_Ceby}a_{2k}+a_{2k+2}=a_{2k+1}, \,\,\,a_{2k+1}a_{2k-1}+a_{2k+1}=(1+a_3)a_{2k},\end{equation}
Moreover $a_k >0$ when $k <m$  and $a_{m}=0$.

In the Euclidean space $V=\R^{m-1}$ we choose
  simple roots  $\alpha_1,\cdots,\alpha_{m-1}$ which satisfy
 $\langle \alpha_i, \alpha_{j} \rangle = a_{ij}$
 where $a_{ij}=2$ if $i=j$, $a_{ij}=-1$ if $|i-j|=1$, $a_{ij}=0$ otherwise.
Let $\alpha_i^\vee=\alpha_i$ and $s_i=s_{\alpha_i}$. These generate a root
system of type $A_{m-1}$.

Let $\Pi$ be the two dimensional plane  
defined as the set of $x \in V$ such that for all $n <m$,
 $$\langle \alpha_n, x \rangle =a_n \langle \alpha_1, x \rangle$$  if  $n$ is odd,  and $$  \langle \alpha_n, x \rangle =a_n \langle \alpha_2, x \rangle$$
if $n$ is even. It follows  from the relation (\ref{rec_Ceby}) 
that the vectors $$\alpha=\sum_{n \mbox{ \small  odd}, n < m}a_n\alpha_n,
 \,\,\, \beta=\sum_{n \mbox{ \small even}, n <m}a_n\alpha_n$$ are in $\Pi$. 
 Let $\alpha^\vee=2\alpha/||\alpha||^2,\,\beta^\vee=2\beta/||\beta||^2$ and
$$\tau_1=s_1s_3s_5\cdots s_{2p-1} ,$$ 
$$ \tau_2=s_2s_4s_6\cdots s_{2r},$$
where $2p=m-1, r=p$ when $m$ is odd and  $2p=m,r=p-1$ when $m$ is even.
Let $w_0$ be the longest element in the Weyl group of $A_{m-1}$. Its length is $q=(m-1)m/2$.  We first consider the case where $m$ is odd, $m=2p+1, q= pm$. Then
$$w_0=(\tau_1\tau_2)^p\tau_1, \mbox{  and  } w_0=\tau_2(\tau_1\tau_2)^p$$
are two reduced decompositions of $w_0$. Since $(\tau_1\tau_2)^m=Id$ the
 angle between $\alpha$ and $-\beta$ is $\pi/m$ and these vectors are 
 the simple roots of the dihedral system $I(m)$.

 Let
 $\gamma$ be a continuous path in $\Pi$, let $\gamma_p=\gamma$ and for  
 $1 < k \leq p$, $\gamma_{k-1}=\PP_{\alpha_{2k-1}}\gamma_{k}$ and 
$$z_k(t)=-
    \inf_{0\leq s\leq t}\alpha_{2k-1}^{\vee}(\gamma_{k}(s)).$$
\begin{lemma}\label{lemzx}
Let $\gamma$ be a continuous path with values in $\Pi$ and let 
$$x(t)= -\inf_{0\leq s\leq t}\alpha^\vee(\gamma(s)).$$ 
    Then, for all $k$, $z_k(t)=a_{2k-1}x(t)$ and 
      $$ {\PP}_{\tau_1}\gamma(t)=
      {\PP}_{\alpha_1}{\PP}_{\alpha_3}{\PP}_{\alpha_5}
      \cdots{\PP}_{\alpha_{2p-1}}\gamma(t)
=\gamma(t)-\inf_{s \leq t}\alpha^{\vee}(\gamma(s))
\alpha={\PP}_\alpha\gamma(t).$$
\end{lemma}
\proof First, notice that $\alpha^\vee(\gamma(t))=\alpha_1^\vee(\gamma(t))$.
Since $\gamma$ is in $\Pi$, one has
$$z_p(t)=-\inf_{0\leq s\leq t}\alpha_{2p-1}^\vee(\gamma(s))=-\inf_{0\leq s\leq t}a_{2p-1}\alpha_{1}^\vee(\gamma(s))=a_{2p-1}x(t)$$
where we use the positivity of $a_{2p-1}$. Therefore
$$\gamma_{p-1}(t)=\PP_{\alpha_{2p-1}}\gamma(t)=\gamma(t)+z_p(t)\alpha_{2p-1}=\gamma(t)+a_{2p-1}x(t)\alpha_{2p-1}.$$
Now, since the $\alpha_{2i+1}$ are orthogonal,
$$z_{p-1}(t)=-\inf_{0\leq s\leq t}\alpha_{2p-3}^\vee(\gamma_{p-1}(s))=-\inf_{0\leq s\leq t}\alpha_{2p-3}^\vee(\gamma(s))=a_{2p-3}x(t),$$
and 
$$\gamma_{p-2}(t)=\PP_{\alpha_{2p-3}}\gamma_{p-1}(t)=\gamma_{p-1}(t)+z_{p-1}(t)\alpha_{2p-3}$$
$$=\gamma(t)+x(t)(a_{2p-3}\alpha_{2p-3}+a_{2p-1}\alpha_{2p-1}).$$
Continuing, we obtain that 
$$z_{k}(t)=a_{2k-1}x(t)$$
$$\gamma_k(t)=\gamma(t)+x(t)(a_{2k-1}\alpha_{2k-1}+\cdots+a_{2p-1}\alpha_{2p-1})$$
Since $\alpha=\alpha_1+a_3\alpha_3+a_5\alpha_5+\cdots+a_{2p-1}\alpha_{2p-1}$ we obtain the lemma. $\square$
\medskip

We  have similarly, if $\gamma$ is a path in  $\Pi$, 
 $$ {\PP}_{\tau_2}\gamma(t)={\PP}_{\alpha_2}
 {\PP}_{\alpha_4}{\PP}_{\alpha_6}\cdots{\PP}_{\alpha_{2r}}
 \gamma(t)=\gamma(t)-\inf_{s \leq t}\beta^{\vee}(\gamma(s))\beta=
 {\PP}_\beta\gamma(t).$$

Let ${\bf i}=(s_{i_1},\cdots,s_{i_q})=({\bf i_1},{\bf i_2},\cdots,{\bf i_{m}})$ and ${\bf j}=(s_{j_1},\cdots,s_{j_q})=({\bf j_1},{\bf j_2},\cdots,{\bf j_{m}})$ where
${\bf i_k}={\bf j_{k+1}}=(s_1,s_3,\cdots,s_{2p-1})$ when $k$ is odd and ${\bf i_k}={\bf j_{k+1}}=(s_2,s_4,\cdots,s_{2p})$ when $k$ is even.
We write explicitly $$w_0=(\tau_1\tau_2)^p\tau_1=s_{i_1}\cdots s_{i_q}, w_0=\tau_2(\tau_1\tau_2)^p=s_{j_1}\cdots s_{j_q}.$$ Let us denote by $\phi_{\bf i}^{\bf j}:\R^q\to \R^q$ the mapping given by the second step corrresponding to these two reduced decompositions of $w_0$ in the Weyl group of $A_{m-1}$.

 Let $\gamma$ be a path with values in $\Pi$. 
 If we consider it as a path in $V$ we can set 
$\eta_q=\tilde \eta_q=\gamma$ and,  for $n=1,2,\ldots ,q$,
$$\eta_{n-1}=\PP_{\alpha_{i_{n}}}\eta_{n},\quad z_n=-
    \inf_{0\leq t\leq T}\alpha_{i_n}^{\vee}( \eta_n(t))$$
$$\tilde\eta_{n-1}=\PP_{\alpha_{j_{n}}}\tilde\eta_{n}, \quad\tilde z_n=-
    \inf_{0\leq t\leq T}\alpha_{j_n}^{\vee}( \tilde\eta_n(t)).$$
Then, by definition,
$$(\tilde z_1,\cdots, \tilde z_q)= \phi_{\bf i}^{\bf j}(z_1,\cdots,z_q).$$
We now consider $\gamma$ as a path in $\Pi$. We let
$$(u_1,u_2,\cdots,u_m)=(\alpha,\beta,\alpha,\beta,\cdots,\alpha)$$
and 
$$(v_1,v_2,\cdots,v_m)=(\beta,\alpha,\beta,\alpha,\cdots,\beta).$$
In $I(m)$ the two reduced decompositions of the longest element are 
$$s_{u_1}\cdots s_{u_m}=s_{v_1}\cdots s_{v_m}.$$We introduce
$\gamma_m=\tilde \gamma_m=\gamma$, and,  for $n=1,2,\ldots ,m,$
$$\gamma_{n-1}=\PP_{u_{n}}\ldots \PP_{u_{m}}\gamma_m, \tilde\gamma_{n-1}=\PP_{v_{n}}\ldots \PP_{u_{m}}\tilde\gamma_m$$ 
 $$x_n=-
    \inf_{0\leq t\leq T}u_n^{\vee}( \gamma_n(t)), \tilde x_n=-
    \inf_{0\leq t\leq T}v_n^{\vee}(\tilde \gamma_n(t)).$$
It follows from  lemma \ref{lemzx} and from its analogue with $\alpha$ replaced by $\beta$
 that 
$$z_1=a_1x_1,z_2=a_3x_1,\cdots,z_p=a_{2p-1}x_1$$ $$z_{p+1}=a_2x_2,z_{p+2}=a_4x_2,\cdots,z_{2p}=a_{2p}x_2$$
and more generally, for $k=0,\cdots $
$$a_1^{-1}z_{2kp+1}=a_3^{-1}z_{2kp+2}=\cdots=a_{2p-1}^{-1}z_{2kp+p}=x_{k+1}$$ 
$$a_2^{-1}z_{(2k+1)p+1}=a_4^{-1}z_{(2k+1)p+2}=\cdots=a_{2p}^{-1}z_{(2k+2)p}=x_{k+2}.$$ This defines a linear map $$(x_1,\cdots,x_m)=g(z_1,z_2,\cdots,z_q).$$
Analogously exchanging the role of $\alpha$ and $\beta$
  we define a similar map $$ (\tilde x_1,\cdots,\tilde x_m)
  =\tilde g(\tilde z_1,\tilde z_2,\cdots,\tilde z_q)$$
(for instance $\tilde z_1=a_2\tilde x_1,\tilde z_2=
a_3 \tilde x_1,\cdots$). Then we see that
$$(x_1,\cdots,x_m)=\phi (\tilde x_1,\cdots,\tilde x_m)$$
where 
$\phi=\tilde g \circ \phi_{\bf i}^{\bf j} \circ g^{-1}.$
The proof when $m$ is even is similar (when $m=2p$, 
$w_0=(\tau_1\tau_2)^p$ and $w_0=(\tau_2\tau_1)^p$
are two reduced decompositions of $w_0$).
This proves the theorem in the dihedral case.
\medskip

Fourth step. We use Matsumoto's lemma to reduce the general case to
 the dihedral case. 
 
 This ends the proof of theorem    \ref{piecewise}.
 $\square$

\medskip

\begin{rem} Although the given proof is constructive, 
it gives a complicated expression for $\phi_{\bf i}^{\bf j}$
 which can sometimes be simplified. In the dihedral case $I(m)$, for the Weyl
  group case, i.e. $m=3,4,6$, 
these expressions are given in Littelmann \cite{littel2}.
 For $m=5$ it can be shown by a tedious verification that it is given when $\alpha,\beta$ have the same length, by a similar formula. Thus 
 for $m=2,3,4,5,6$  let $c_0=1,c_1=2\cos(\pi/m),  c_{n+1}+c_{n-1}=c_1c_n$ for $n \geq 0,$
 and
 $$u=\max(c_kx_{k+1}-c_{k-1}x_{k+2}, 0 \leq k \leq m-3),$$ $$ v=\min(c_kx_{k+2}-c_{k+1}x_{k+1}, 1 \leq k \leq m-2).$$
 Then the expressions are given by
 \begin{eqnarray*}
y_m&=&\max(x_{m-1}-c_1x_m,u)\\
y_{m-1}&=&x_m+\max(x_{m-2}-c_2x_m,c_1u)\\
y_2&=&x_1+\min(x_3-c_2x_1,c_1v)\\
y_1&=&\min(x_2-c_1x_1,v)
\end{eqnarray*}
and
 $$y_1+y_3+\cdots= x_2+x_4+\cdots$$
 $$y_2+y_4+\cdots=x_1+x_3+\cdots$$
 This determines completely $(y_1,\cdots,y_m)$ as a function of $(x_1,\cdots,x_m)$ when $m \leq 6$. For $m=7$ we think
(and made a computer check) that we have to add that
  \begin{eqnarray*}y_7+y_5&=&x_6+\max(c_2x_1,x_4-c_3x_7,w)\\
w&=&\min(c_2u,x_4-c_2v,\max(x_6-c_1x_5+x_4+c_2u,c_1x_3-x_2-c_2v).
\end{eqnarray*}
\end{rem}
We do not know of similar formulas for $m\geq 8$.
\begin{rem} The map given by theorem \ref{piecewise} is unique on the set of
all possible coordinates of paths. We will see in the next section that this set is a
convex cone. Since the value of the map $\fij$ is irrelevant outside this cone, we may say
that there exists a unique such  map for each pair of reduced decompositions
 ${\bf i},{\bf j}$.
 \end{rem} 
 \section{Parametrization of the continuous Littelmann module}
In this section we make a more in-depth study of the parametrization of the Littelmann modules,
and we prove the analogue of the independence theorem of Littelmann (the crystal structure depends only on the endpoint of the dominant path), then we study the concatenation of paths, using it to prove existence and uniqueness of families of crystals. Finally we define the action of the Coxeter group on the crystal, and the 
Sch\"utzenberger involution.
  \subsection{String parametrization of $C_T^0(V)$}
 Let $(W,S,V,V^\vee)$ be a realization of the Coxeter system
   $(W,S)$.  
  From now on we assume that $W$ is finite, with longest element 
  $w_0$.
  For notational convenience, we
 sometimes write $\alpha^\vee \eta$ instead of $\alpha^\vee (\eta)$.

Let $\eta \in L_\pi$, where $\pi$ is dominant and $w_0=s_1\ldots s_q$ be a reduced decomposition, then we have seen that
 $$\eta=\HH_{s_{q}}^{x_q}\HH_{s_{q-1}}^{x_{q-1}}\cdots
\HH_{s_{1}}^{x_1}\pi
$$
for a unique sequence $$\varrho_{\bf i}(\eta)=(x_1,\ldots, x_q).$$
Following Berenstein and Zelevinsky \cite{beze2}, we call
$\varrho_{\bf i}(\eta)$ the {\bf i}-string parametrization of $\eta$, 
or the string parametrization if no confusion is possible.

  We let $$C_{\bf i}^\pi=\varrho_{\bf i}(L_\pi),$$
  this is the set of all the $(x_1,\cdots,x_q)\in \R^q$ 
  which occur in the string parametrizations of the elements of $L_\pi$.
  \begin{prop}\label{homeo}
 The set $L_\pi$  is compact and the map $\varrho_{\bf i}$  is a bicontinuous bijection  from $L_\pi$ onto its image $C_{\bf i}^\pi$.
\end{prop}
\proof
The map $\varrho_{\bf i}$ has an  inverse
$$\varrho_{\bf i}^{-1}(x_1,\cdots,x_q)=\HH_{s_{q}}^{x_q}
\HH_{s_{q-1}}^{x_{q-1}}\cdots\HH_{s_{1}}^{x_1}\pi,
$$
hence it is bijective.
It is clear that $\varrho_{\bf i}$ and $\varrho_{\bf i}^{-1}$ are continuous. Since $\PP_{w_0}$ is continuous, $L_\pi=\{\eta; \PP_{w_0}(\eta)=\pi\}$ is closed. Using $\varrho_{\bf i}^{-1}$ we easily see  that $L_\pi$ is equicontinuous, it is thus compact by Ascoli's theorem. \hfill$\square$

We will study $C_{\bf i}^\pi$ in detail in the following sections.
\medskip

       \subsection{The crystallographic case}
In this subsection we  consider the case of a Weyl group $W$ with a
crystallographic root system. 
When $\alpha$ is a root and $\alpha^{\vee}$ 
its coroot,
    then  $\EE_{\alpha}^1$ and $\EE_{\alpha}^{-1}$ from definition \ref{littelmanntransform}
coincide with the Littelmann operators $e_{\alpha}$ and $f_{\alpha}$,
    defined in
\cite{littel}. 
Recall that a path $\eta$ is called integral in \cite{littel} if its endpoint $\eta(T)$
 is in the weight
lattice and if, for each simple root $\alpha$,
 the
minimum of the function $\alpha^{\vee}(\eta(t))$ over $[0,T]$ is an integer.
 The class of 
integral paths
is invariant under the Littelmann operators. 

 Let $\pi$ be a dominant integral path.  
 The discrete Littelmann module $D_\pi$ is defined
as the orbit of $\pi$ under the semigroup generated by all the transformations
 $e_\alpha, f_\alpha$, for all simple roots $ \alpha $, so it is the set of integral paths in $L_\pi$.

Let ${\bf i}=(s_1,\cdots,s_q)$ where $w_0=s_1\cdots s_q$ 
is a reduced decomposition, then it follows from Littelmann's theory that 
$$D_\pi=\{\eta \in L_\pi; x_1,\cdots, x_q \in \N\}=\varrho_{\bf i}^{-1}(\{(x_1,\cdots,x_q)
\in C_{\bf i}^\pi;  x_1\in \N,\cdots, x_q \in \N\}).$$
Furthermore, the set  $D_\pi$  
has a crystal structure isomorphic to the Kashiwara 
crystal associated with the highest weight $\pi(T)$.
On $ D_\pi$ the coordinates $(x_1,\cdots,x_q)$ are called the string or the Kashiwara parametrization of the dual canonical basis. They are 
described in  Littelmann \cite{littel2} and Berenstein and Zelevinsky \cite{beze2}.

When restricted to $D_\pi$, the Pitman operator $\pa$ coincides with $e_\alpha^{max}$, i.e.
 the operator sending $\eta$ to $e_\alpha^n\eta$, where $n=\max(k,e_\alpha^k\eta\ne{\bf 0})$.

For any path $\eta:[0,T]\to V$ 
and $\lambda>0$ let $\lambda\eta$ be the path defined by $(\lambda\eta)(t)=\lambda\eta(t)$ for $0 \leq t \leq T$. 
The following results are immediate.
\begin{prop}[Scaling property] \label{scal}

\begin{enumerate}[(i)]

\item For any $\lambda >0$, $\lambda L_\pi=L_{\lambda\pi}$.

\item Let $\eta\in C^0_T(V)$, $r\in\mathbb R,u>0$, then $\EE_\alpha^{ru}(u\eta)=u\EE_\alpha^{r}(\eta).$

\item Let $\pi$ be a dominant path and $a>0$ then  
$C_{\bf i}^{a\pi}=aC_{\bf i}^\pi$.
\end{enumerate}
\end{prop}
\begin{prop} \label{dense}If $\pi$ is a dominant integral path, then 
 the set $$D_\pi(\Q)=\cup_{n \in \N}\frac{1}{n}D_{n\pi}$$
is dense in $L_\pi$.
\end{prop}
Actually a good interpretation of $L_\pi$ in the Weyl group case is as the "limit" of $ \frac{1}{n}B_{n\pi}$ when $n \to \infty$. In the general Coxeter case only the limiting object is defined.
\subsection{Polyhedral nature of the continuous crystal for a Weyl group}
 Let $W$ be a finite Weyl group,  associated to a crystallographic
 root system.  
 Let $D_\pi$ be the discrete Littelmann module associated 
 with an  integral dominant path $\pi$.  We fix a reduced decomposition
   $w_0=s_{1}\cdots s_{q}$ of the longest element and let 
 ${\bf i}=(s_1,\cdots,s_q)$. We have seen that if $\rho_{\bf i}:L_\pi\to C_{\bf i}^\pi$   is the string parametrization of the continuous module $L_\pi$,  then
 $$D_\pi=\{\eta \in L_\pi; x_1,\cdots, x_q \in \N\}=
\varrho_{\bf i}^{-1}(\{(x_1,\cdots,x_q)\in C_{\bf i}^\pi;  x_1\in \N,\cdots, x_q \in \N\}).$$
Therefore the set 
 $$\tilde C_{\bf i}^\pi=C_{\bf i}^\pi\cap \N^q$$
 is the image of the discrete Littelmann module $D_\pi$, or equivalently, 
 the image of the Kashiwara crystal with highest 
 weight $\pi(T)$, under the string parametrization 
 of Littelmann \cite{littel2} and Berenstein and Zelevinsky \cite{beze2}.
Let 
$$K_\pi=\{(x_1,\cdots,x_q)\in \R^q;0 \leq x_r \leq  \alpha_{i_r}^{\vee}(\pi(T)-\sum_{n=1}^{r-1}x_n\alpha_{i_n}), r=1,\cdots q\}.$$
 It is shown in Littelmann  \cite{littel2} that there exists a 
  convex rational polyhedral cone $C_{\bf i}$ in $\R^q$, depending only on 
  $\bf i$ such that, for all dominant integral paths $\pi$,  
 $$ \tilde C_{\bf i}^\pi= C_{\bf i}\cap \N^q\cap K_\pi.$$
 This cone is described explicitly in Berenstein and Zelevinsky \cite{beze2}. 
 Recall that $C_{\bf i}^\pi=\varrho_{\bf i}(L_\pi)$. 
 Using propositions \ref{scal}, \ref{dense} it is easy
  to see that the following holds.
 \begin{prop} For all dominant paths $\pi$,
  $ C_{\bf i}^\pi= C_{\bf i}\cap K_\pi.$ 
 \end{prop}
 
    \subsection{The cone in the general case}
We now consider a general  Coxeter system $(W,S)$, with $W$ finite, realized in
$V$. 
 \begin{theorem}\label{theo_poly}Let 
 $\bf i$ be a reduced decomposition of $w_0$, then  there exists a unique
 polyhedral
  cone $C_{\bf i}$ in $\R^q$ such that for any dominant path $\pi$ 
  $$ C_{\bf i}^\pi= C_{\bf i}\cap K_\pi.$$
 In particular $C_{\bf i}^\pi$ depends only on $\lambda=\pi(T)$.
 \end{theorem}
\proof It remains to consider the non crystallographic Coxeter systems. It is clearly
enough to consider reduced systems.
We use their classification:  $W$ is either a dihedral group $I(m)$ or 
$H_3$ or $H_4$ (see Humphreys \cite{humphreys}), and the same trick as the one used 
in the proof of theorem \ref{piecewise}.

We first consider the case $I(m)$ where $m=2p+1$ and we use the notation of the proof of 
theorem \ref{piecewise}. Let 
 ${\bf i}=(i_1,\cdots,i_q)$  be as in that proof, and 
 write  $$w_0=(\tau_1\tau_2)^p\tau_1=s_{i_1}\cdots s_{i_q}$$ for the longest word in $A_{m-1}$.
 Let $\gamma$ be a path with values in the plane $\Pi$.  
 If we consider $\gamma$ as a path in $V=\R^{m-1}$ we can set, for $q=(m-1)m/2$, 
$\eta_q=\gamma$   and,  for $n=1,2,\ldots ,q,$
$$\eta_{n-1}=\PP_{\alpha_{i_{n}}}\eta_{n},\quad z_n=-
    \inf_{0\leq t\leq T}\alpha_{i_n}^{\vee}( \eta_n(t)).$$
We can also consider $\gamma$ as a path in $\Pi$, with the realization of $I(m)$. Let
$${\bf u}=(u_1,u_2,\cdots,u_m)=(\alpha,\beta,\alpha,\beta,\cdots,\alpha).$$
Let
$\tilde\eta_m=\gamma$ and,  for $n=1,2,\ldots ,m,$
$$\tilde\eta_{n-1}=\PP_{u_{n}}\ldots \PP_{u_{m}}\eta_m,\quad x_n=-
    \inf_{0\leq t\leq T}u_n^{\vee}( \eta_n(t)).$$
We have seen that the  map $$(x_1,\cdots,x_m)=g(z_1,z_2,\cdots,z_q),$$ is linear. 
Let $C_{\bf i}$ be the cone associated with ${\bf i}$ in $A_{m-1}$,
then
$C_{\bf u}=g( C_{\bf i})$
 is the cone in $\R^m$ associated with the reduced decomposition   $ \alpha\beta\cdots\alpha$ of the longest word in $I(m)$. 
Furthermore,  for any dominant path $\pi$ in $\Pi$,
  $ C_{\bf u}^\pi= C_{\bf u}\cap K_\pi.$
  
  The proof when $m$ is even is similar. 

\medskip

In order to deal with the cases $H_3$ and $H_4$ it is enough, using an analogous proof to embed these systems in some Weyl groups.

Let us first consider the case of $H_4$. We use the embedding of $H_4$ in $E_8$ (see \cite{moopat}). 
Consider the following indexation of the simple roots of the system $E_8$:
  $${\tt    \setlength{\unitlength}{0.60pt}
\begin{picture}(320,150)
\thinlines    \put(120,10){ System $E_8$}           
 
            \put(310,45){{$8$}}
             \put(260,45){{$7$}}
              \put(210,45){{$6$}}
              \put(160,45){{$4$}}
                 \put(110,45){{$3$}}
                 
                 \put(220,100){{$5$}}
                    \put(211,66){\line(0,1){33}}
              \put(60,45){{$2$}}
               \put(10,45){{$1$}}
               \put(65,63){\line(1,0){40}}
                 \put(115,63){\line(1,0){40}}
                   \put(165,63){\line(1,0){40}}
                \put(265,63){\line(1,0){40}}
              \put(215,63){\line(1,0){40}}
                      \put(211,103){\circle{8}}
              \put(211,63){\circle{8}}
             \put(161,63){\circle{8}}
               \put(261,63){\circle{8}}
                 \put(311,63){\circle{8}}
               \put(11,63){\circle{8}}
              \put(111,63){\circle{8}}
             \put(61,63){\circle{8}}
              \put(16,63){\line(1,0){40}}
              \end{picture}}
$$
In the euclidean space $V=\R^8$ the roots $\alpha_1,...,\alpha_8$,
satisfy  $\langle \alpha_i, \alpha_j \rangle= -1\mbox{ or } 0$ depending whether they are linked or not. Let $ \phi=(1+\sqrt{5})/2$. We consider the 4-dimensional subspace $\Pi$ of $V$ defined as the set of  $x \in V$ orthogonal to $ \alpha_8-\phi\alpha_1, \alpha_7-\phi\alpha_2, \alpha_6-\phi\alpha_3$ and $ \phi\alpha_5- \alpha_4$.
Let $s_i$ be the reflection which corresponds to $\alpha_i$  and $$\tau_1=s_1s_8,\,\, \tau_2=s_2s_7, \,\, \tau_3= s_3s_6,\,\,  \tau_4= s_4s_5.$$
 One checks easily that   $\tau_1,\tau_2,\tau_3,\tau_4$ generate $H_4$ and that the vectors$$\tilde \alpha_1=\alpha_1+\phi\alpha_8,\tilde \alpha_2=\alpha_2+\phi \alpha_7,\tilde \alpha_3=\alpha_3+\phi \alpha_6,\tilde \alpha_4=\alpha_4+\phi^{-1}\alpha_5$$ are in $\Pi$. If $\pi$ is a continuous path in $\Pi$, then,   for $i=1,\cdots,4$, if $\tilde \alpha_i^\vee=\tilde \alpha_i/(2||\tilde \alpha_i||^2)$,
 $$ {\PP}_{\tau_i}\pi(t)=\pi(t)-\inf_{0 \leq s \leq t}\tilde\alpha_i^{\vee}(\pi(s))\tilde\alpha_i.$$

  The case of $H_3$ is similar by using $D_6$:
  $${\tt    \setlength{\unitlength}{0.60pt}
\begin{picture}(263,150)
\thinlines    \put(100,0){ System $D_6$}           
 \put(209,5){{$6$}}
              \put(110,45){{$3$}}
              \put(209,110){{$5$}}
              \put(155,45){{$4$}}
              \put(59,45){{$2$}}
  \put(65,63){\line(1,0){40}}
              \put(9,45){{$1$}}
                      \put(211,103){\circle{8}}
              \put(211,23){\circle{8}}
        \put(166,63){\line(1,-1){40}}
              \put(166,63){\line(1,1){40}}
              \put(161,63){\circle{8}}
              \put(116,63){\line(1,0){40}}
              \put(111,63){\circle{8}}
      \put(61,63){\circle{8}}
              \put(16,63){\line(1,0){40}}
              \put(11,63){\circle{8}}
\end{picture}}
$$
In $V=\R^6$ we choose the roots $\alpha_1,...,\alpha_6$ with $\langle \alpha_i, \alpha_{j} \rangle =-1$ if they are linked. We define a 3-dimensional subspace $\Pi$  defined as the set of  $x \in V$ orthogonal to $ \alpha_5-\phi\alpha_1, \alpha_4-\phi\alpha_2$ and $ \phi\alpha_6- \alpha_3$.
 Then the reflections
\begin{equation}\label{tau_sig}\tau_1=s_1s_5,\,\, \tau_2=s_2s_4, \,\, \tau_3= s_3s_6,
\end{equation}
  generate $H_3$ and
   $$\tilde \alpha_1=\alpha_1+a\alpha_5,\tilde \alpha_2=\alpha_2+a\alpha_4,\tilde \alpha_3=\alpha_3+b\alpha_6$$
 are in $\Pi$. $\square$

    We will  prove in corollary \ref{cor_cone}
that the cones $C_{\bf i}$ have the following description: 
for any simple root $\alpha$, let ${\bf j}(\alpha)$ be a reduced decomposition of $w_0$ which begins by $s_\alpha$. 
    Then $$C_{\bf i}=\{x \in \R^q; \phi_{\bf i}^{{\bf j}(\alpha)}
    (x)_1 \geq 0, \mbox{ for all simple roots } \alpha\}.$$    
    \subsection{The cone in the dihedral case}
  In this section we provide explicit equations for the cone,  in the dihedral case,
 following the approach of Littelmann \cite{littel2} in the Weyl group case.
\begin{lemma}\label{lemme_di} Let $\alpha, \beta\in V$, 
$\alpha^\vee, \beta^\vee \in V^\vee$ and
$c=-\beta^\vee(\alpha)$. Consider a continuous path 
$\eta \in C^0_T(V)$  and
$\pi=\PP_{\alpha}\eta$. Let 
\begin{eqnarray*}
U&=&\min_{T \geq t \geq�0}[a \beta^\vee (\eta(t)) +b \min_{t\geq s\geq 0}
\alpha^\vee(\eta(s))]�,\\
V&=&\min_{T\geq t\geq 0}[a \min_{t\geq s\geq 0}\beta^\vee(\pi(s)+
(ac-b) \alpha^\vee (\pi(t))],\\ W&=&a
\min_{T\geq t\geq 0}\beta^\vee (\pi(t))-(ac-b)\min_{T\geq t\geq 0} \alpha^\vee(\eta(t)),
\end{eqnarray*}
where $a,b$ are real numbers such that $a \geq 0, ac-b \geq�0$. Then $U=\min(V,W)$.
\end{lemma}

\proof Since
$\pi=\PP_{\alpha}\eta$, 
$$\beta^\vee(\eta(t))=\beta^\vee(\pi(t))-c\min_{t\geq s\geq 0} \alpha^\vee(\eta(s)),$$ thus 
\begin{eqnarray*}
U
&=&\min_{T \geq t \geq 0}[a
\beta^\vee (\pi(t))+(b-ac)\min_{t\geq s\geq 0}\alpha^\vee(\eta(s))]\\
&=&\min_{T \geq t \geq 0}[
\min_{t\geq s\geq 0}a\beta^\vee (\pi(s))+(b-ac)\min_{t\geq s\geq 0} \alpha^\vee(\eta(s))].
\end{eqnarray*}
where we have used  the  fact that, 
if  $f,g:[0,T]\to {\mathbb R}$ are two continuous functions, and if $g$ is
non decreasing, then 
$$\min_{T\geq t\geq 0}[f(t)+g(t)]=\min_{T\geq t\geq 0}[\min_{t\geq s\geq 0}f(s)+g(t)].$$

 Since
 $\alpha^\vee(\pi(t))
\geq -\min_{t\geq s\geq 0}
\alpha^\vee(\eta(s)),$ 
$$
\min_{t\geq s\geq 0}a\beta^\vee (\pi(s))+(ac-b) \alpha^\vee(\pi(t))\geq
\min_{t\geq s\geq 0}a\beta^\vee (\pi(s))-(ac-b)\min_{t\geq s\geq 0}
\alpha^\vee(\eta(s)).$$ Let $t_0$ be the largest $t \leq�T$ where the minimum of the right hand side
is achieved. Suppose that $t_0 < T$. 
If $\alpha^\vee(\pi(t_0)) > -\min_{t_0\geq s\geq
0}
\alpha^\vee(\eta(s))$ then $\min_{t\geq s\geq 0}
\alpha^\vee(\eta(s))$ is locally constant on the right of $t_0$. Since
$\min_{t\geq s\geq 0}a\beta^\vee (\pi(s))$ is non increasing, it follows that
$t_0$ is not maximal. Therefore, when $t_0<T,$  
$$\alpha^\vee(\pi(t_0)) = -\min_{t_0\geq s\geq
0}
\alpha^\vee(\eta(s))$$ and 
$$U=\min_{T \geq t \geq 0}[
\min_{t\geq s\geq 0}a\beta^\vee
 (\pi(s))-(ac-b)\inf_{t\geq s\geq 0}\alpha^\vee(\eta(s))]=V\leq W.$$ When
$t_0=T,$ then $U=W\leq V$. Thus $U=\min(V,W)$.
$\square$
\medskip

We consider a realization of  the dihedral system $I(m)$ with two simple roots $\alpha, \beta$ and
 $c:=-\alpha ^{\vee}(\beta) = -\beta^{\vee} (\alpha) =2\cos\frac{\pi}{m}.$
  Let  $$a_n=\frac{\sin(n\pi/m)}{
\sin (\pi/m)}.$$ Then  $a_0=0, a_1=1,$ and $a_{n+1}+a_{n-1}=ca_n$, $a_n > 0$ if $1 \leq n \leq�m-1$ and
$a_m=0$. 
 Let
$w_0=s_{1}\ldots s_{m}$ be  a reduced decomposition of the longest element $w_0\in W$, ${\bf i} =(s_1,\cdots, s_m)$ and $\alpha_1,\cdots,\alpha_m$ be the simple
roots associated with $s_1,\cdots,s_m$. This sequence is either $(\alpha,\beta,\alpha,\cdots)$ or
$(\beta,\alpha,\beta,\cdots)$. Clearly the two roots play a symmetric role, and the cones associated with these two decompositions are the same. We define $\alpha_0$ as the simple root not equal to $\alpha_1$.
As before, when $\eta \in C^0_T(V)$, we define $\eta_m=\eta$ and  for $k=0,\cdots,m-1$,
    $\eta_{k}=\PP_{s_{{k+1}}}\ldots \PP_{s_{m}}\eta,$
  and
 $$x_k=-
    \min_{0\leq t\leq T}\alpha_{{k}}^{\vee}( \eta_k(t))\quad\text{for }k=1,\ldots,m.$$
\begin{prop} \label{dihedralcone}The cone for the dihedral system $I(m)$ is given by
$$C_{\bf i}=\{(x_1,\cdots,x_m)\in {\mathbb R}_+^m; 
\frac{x_{m-1}}{a_{m-1}}\geq \frac{x_{m-2}}{a_{m-2}} \geq \cdots \geq
\frac{x_{1}}{a_{1}}
\}.$$
\end{prop}
    \proof For any $p,k $ such that $0 \leq p \leq m, 0 \leq k \leq p$, let 
    \begin{eqnarray*}V_k&=&\min_{T \geq t
\geq 0}[a_{k+1} \alpha_{p+1-k}^\vee (\eta_{p-k}(t))+a_{k}
\min_{t\geq s\geq 0}
\alpha_{{p-k}}^\vee(\eta_{p-k}(s))],\\W_k&=&a_{k} \min_{T\geq t\geq 0}\alpha_{p-k}^\vee (\eta_{p-k}(t))-a_{k+1} \min_{T\geq t\geq 0}\alpha_{p+1-k}^\vee (\eta_{p+1-k}(t))
.\end{eqnarray*}
 Since $a_{k-1}+a_{k+1}=ca_k$,   the lemma above gives that $V_k=\min(W_{k+1},V_{k+1})$. Therefore
 $$V_0=\min(W_1,W_2, \cdots,W_p,V_p).$$
Notice that 
    $$V_p=\min_{T \geq t
\geq 0}[a_{p+1} \alpha_{1}^\vee (\eta_{0}(t))+a_{p}
\min_{t\geq s\geq 0}
\alpha_{{0}}^\vee(\eta_{0}(s))]=0$$ and $W_p=a_{p+1}x_1$
since $\eta_0=\PP_{w_0}\eta$ is dominant. Furthermore  
$$V_0=\min_{0\leq t\leq T}\alpha_{{p+1}}^{\vee}( \eta_p(t))$$ 
since $a_0=0$ and $a_1=1$. Hence,
 \begin{equation}\label{eq_min}
  \min_{0\leq t\leq T}\alpha_{{p+1}}^{\vee}( \eta_p(t))=\min(a_2x_{p}-a_1x_{p-1},\cdots,a_px_{2}-a_{p-1}x_1,a_{p+1}x_{1},0 ).
  \end{equation}
The path $\eta_{ m-1}=\PP_{\alpha_m}\eta$ is $\alpha_m$-dominant, 
therefore $\alpha_{{m}}^{\vee}( \eta_{m-1}(t))\geq 0$
 and it follows from $(\ref{eq_min})$ applied with $p=m-1$ that for $k=1,\cdots,m-2$
$$a_{m-k}x_{k+1}-a_{m-k-1}x_{k} \geq 0,$$
which is equivalent,  since $a_{m-k}=a_k$ to 
$$\frac{x_{m-1}}{a_{m-1}}\geq \frac{x_{m-2}}{a_{m-2}}
 \geq \cdots \geq
\frac{x_{1}}{a_{1}}\geq 0.$$
Conversely, we suppose that these inequalities hold, i.e. that for $k=1, \cdots,m-2$
\begin{equation}\label{eqnumer}a_{k+1}x_{m-k}-a_kx_{m-k-1} \geq 0,\end{equation}
$$a_{m-k}x_{k+1}-a_{m-k-1}x_{k} \geq 0,$$
 and that $(x_1,\cdots, x_m)\in K_\pi$ for some dominant path $\pi$.
Let us show that 
$$\eta=\HH_{\alpha_m}^{x_m}\cdots \HH_{\alpha_1}^{x_1}\pi$$
is well defined. Since the string parametrization of $\eta$ is $x$ this will prove the proposition. It is enough to show, by induction on $p=0,\cdots,m$ that 
$$\eta_p:=\HH_{\alpha_p}^{x_p}\HH_{\alpha_{p-1}}^{x_{p-1}}\cdots \HH_{\alpha_1}^{x_1}\pi$$
is $\alpha_{p+1}$-dominant. This is clear for $p=0$ since $\eta_0=\pi$ is dominant. If we suppose that this is true until $p-1$ can apply (\ref{eq_min}) and write that
$$  \min_{0\leq t\leq T}\alpha_{{p+1}}^{\vee}( \eta_p(t))=\min(a_2x_{p}-a_1x_{p-1},\cdots,a_px_{2}-a_{p-1}x_1,a_{p+1}x_{1},0 )$$
Since $c \leq 2$, it is easy to see that 
$$\frac{a_{n-1}}{a_n} \geq \frac{a_{n-2}}{a_{n-1}}$$
 for $n \leq�m-1$. Therefore,   $$\frac{x_{k+1}}{x_{k}} \geq \frac{a_{m-k-1}}{a_{m-k}}\geq \frac{a_{p-k}}{a_{p-k+1}}$$ and  $ \alpha_{{p+1}}^{\vee}( \eta_p(t) \geq�0$ for all $0 \leq t \leq T$. $\square$
 \medskip
 
 In the definition of $V_k$ and $W_k$ in the proof above,
 replace the sequence $(a_k)$ by the sequence $(a_{k+1})$.
 We obtain the following formula. \begin{prop}
 If $y_m= -\min_{T \geq t \geq 0} \alpha_{m-1}^\vee
(\eta_m(t))$, then
$$y_m=\max\{0,a_{m-1}x_{m-1}-a_{m-2}x_m,a_{m-2}x_{m-2}-a_{m-3}x_{m-1},\cdots,a_{2}x_{2}-a_{1}x_{3},a_{1}x_1\}$$
\end{prop}    
\subsection{Remark on Gelfand Tsetlin cones}

In the Weyl group case, the continuous cone $C_{\bf i}$ appears in the description of toric degenerations   (see Caldero \cite{cald}, Alexeev and Brion \cite{ab}). The polytopes $C_{\bf i}^{\pi}$ are called the string polytopes in Alexeev and Brion \cite{ab}.
Notice that they have shown that the classical Duistermaat-Heckman 
measure coincides with the one given below in Definition \ref{defDH}.
Explicit inequalities for the string cone $C_{\bf i}$ (and therefore for
the string polytopes) in the Weyl group case are given in full generality in 
Berenstein and Zelevinsky in \cite[Thm.3.12]{beze2}. Before, 
Littelmann \cite[Thm.4.2]{littel2} has described it 
for the so called "nice decompositions" of $w_0$. As explained in that paper they were introduced to generalize the Gelfand Tsetlin cones.

\medskip

For the convenience of the reader let us reproduce the description $C_{\bf i}$ in the $A_n$ case, considered explicitly in Alexeev and Brion \cite{ab}, for 
the standard reduced decomposition of the longest element in the symmetric group
$W=S_{n+1}$. This decomposition $ {\bf i}$ is 
\begin{displaymath}
w_0=(s_1)(s_2s_1)(s_3s_2s_1)\dots(s_ns_{n-1}\dots
s_1),
\end{displaymath}
where $s_i$ denotes the transposition exchanging $i$ with $i+1$.
Let us use on $V$ the coordinates $x_{i,j}$ with $i,j\ge 1$,
$i+j\le n+1$. The string cone
is defined by
$$ 
x_{n,1}\ge0; \quad x_{n-1,2}\ge x_{n-1,1}\ge 0; \quad \dots \quad
x_{1,n} \ge \dots \ge x_{1,1} \ge 0,
$$
and to define the polyhedron $C_{\bf i}^\pi$ one has to add the  
inequalities
\begin{displaymath}
x_{i,j} \le \alpha_j^{\vee}(\lambda)  - x_{i,j-1} +
\sum_{k=1}^{i-1}  (-x_{k,j-1} +2x_{k,j} - x_{k,j+1}).
\end{displaymath}
 where $\lambda=\pi(T)$.  A more familiar description of this cone
is in terms of Gelfand-Tsetlin patterns:
$$
g_{i,j} \ge g_{i+1,j} \ge g_{i,j+1}
$$
where $g_{0,j}=\lambda_j$ and  $g_{i,j}=  \lambda_j + \sum_{k=1}^i  
(x_{k,j-1} - x_{k,j} )$ for $i,j\ge 1$, $i+j\le n+1$.

\subsection{Crystal structure of the Littelmann module}
We now return to the general case of a finite Coxeter group.
Let $\pi$ be a dominant path in $C^0_T(V)$.   
   The geometry of the crystal $L_\pi$ is easy to describe, using the sets $C_{\bf i}^\pi$ which parametrize $L_\pi$. We have seen 
   (theorem \ref{theo_poly}) that $C_{\bf i}^\pi$ depend on the path $\pi$ only through $\pi(T)$. We put on $C_{\bf i}^\pi$ a continuous crystal structure in the following way.
Let ${\bf i}=(s_1,\cdots,s_q)$ where $w_0=s_{1}\cdots s_{q}$ is a reduced decomposition. If $x=(x_1,\cdots,x_q) \in C_{\bf i}^{\pi}$ we set
$$wt(x)=\pi(T)-\sum_{k=1}^q x_k \alpha_{s_k}.$$
If the simple root $\alpha$ is  $\alpha_{s_1}$ then first define $e_{\alpha,{\bf i}}^r$ for $r \in \R$ by
  $$e_{\alpha,{\bf i}}^r(x_1,x_2,\cdots,x_q)=(x_1+r,x_2,\cdots,x_q) \mbox { or } {\bf 0}$$
depending whether  $(x_1+r,\cdots,x_q)$ is in $C_{\bf i}^{\pi}$ or not. We let, for $b\in C_{\bf i}^{\pi}$,  $$\varepsilon_\alpha(b) = \max\{r \geq  0 ; e_{\alpha,{\bf i}}^r(b)\not={\bf 0}\} $$ and
$$\varphi_\alpha(b)=\max\{r \geq  0 ; e_{\alpha,{\bf i}}^{-r}(b)\not={\bf 0}\}.$$
We now consider the case where $\alpha$ is not $\alpha_1$. 
We choose a  reduced decomposition 
$w_0=s'_1s'_2 \cdots s'_q$ with $\alpha_{s'_1}=\alpha$ and let ${\bf j}=(s'_1,s'_2,\cdots, s'_q)$. 
We can define $e_{\alpha,{\bf j}}^r$ on
 $C_{\bf j}^{\pi}, \varepsilon_\alpha,  \phi_\alpha$ as 
 above and transport this action on $C_{\bf i}^{\pi}$ by 
 the piecewise linear map $\phi_{\bf i}^{\bf j}$ 
 introduced in theorem \ref{piecewise}. In other words 
$$e_{\alpha,{\bf i}}^r=\phi_{\bf i}^{\bf j}\circ e_{\alpha,{\bf j}}^r \circ \phi_{\bf j}^{\bf i}.$$
Finally we let we define the crystal operators by $e_{\alpha}^r=e_{\alpha,{\bf i}}^r$. Then  $\rho_{\bf i}:L_\pi \to C_{\bf i}^{\pi}$ 
 is an isomorphism of crystal. This first shows that our construction does not depend on the chosen decompositions $w_0=s'_1s'_2 \cdots s'_q$ and 
 then that the crystal structure on $L_\pi$ depends only on the extremity $\pi(T)$ of the path $\pi$:
  \begin{theorem}\label{uniq_iso}
  If $\pi$ and $\bar \pi$ are 
  two dominant paths such that $\pi(T)=\bar \pi(T)$ 
  then the crystals on $L_\pi$ and  $L_{\bar\pi}$ are  isomorphic. 
  \end{theorem}
This is the analogue of Littelmann independence theorem (see \cite{littel}). 
\begin{definition}
When $W$ is finite, for $\lambda \in \bar C$, we denote $B(\lambda)$ the class of the  continuous crystals isomorphic to $L_{\pi}$ where $\pi$ is a dominant path such that $\pi(T)=\lambda$.
\end{definition}

\subsection{Concatenation and closed crystals}

The concatenation $\pi\star\eta$ of two paths $\pi: [0,T] \to V$, $\eta:  
[0,T] \to V$ is defined in Littelmann \cite{littel} as the
path
$\pi\star\eta:[0,T]\to V$ given by
$(\pi\star\eta)(t)=\pi(2t)$, and $(\pi\star\eta)(t+T/2)=\pi(T)+\eta(2t)$ when
$0
\leq t
\leq T/2$. The following theorem is instrumental to prove uniqueness.

\begin{theorem}\label{closed}
The map $$\Theta:C^0_T(V)\otimes C^0_T(V)\to C^0_T(V)$$ defined by 
$\Theta(\eta_1\otimes \eta_2)=\eta_1\star \eta_2$ is a crystal isomorphism.\end{theorem}

\proof We have to show that, for simple roots $\alpha$, for
$\eta_1\in L_{\pi_1}, 
\eta_2\in L_{\pi_2}$, for all $s \in \R$,
$$\Theta[e_{\alpha}^s(\eta_1\otimes \eta_2)]=
\EE_{\alpha}^s(\eta_1\star \eta_2).$$
This is a purely one-dimensional statement, which uses only
one root, hence it follows from the similar fact for Littelmann
and Kashiwara crystals. For the convenience of the reader we
provide a proof.
For any $x \geq0$, 
let $$\PP_\alpha^x\eta(t)=\eta(t)-\min(0,x+ \inf_{0\leq s\leq t}
\alpha^\vee\eta(s))\alpha.$$
Thus, for $y=(-\inf_{0\leq s\leq T}
\alpha^\vee\eta(s)-x)\vee 0$,
\begin{equation}\label{eq:px}
\PP_\alpha^x\eta=\EE_\alpha^{y}\eta.
\end{equation}
\begin{lemma}\label{lem:Px} Let $\eta_1,\eta_2\in C^0_T(V)$,
then
\begin{enumerate}[(i)]
\item
$\PP_\alpha(\eta_1 \star \eta_2)=\PP_\alpha\eta_1\star \PP_\alpha^x\eta_2$
where $x=\alpha^\vee\eta_1(T)-\inf_{0\leq t\leq T}
\alpha^\vee\eta_1(t)$;
\item if $x\geq 0$, $\PP_\alpha\PP_\alpha^x=\PP_\alpha;$
\item if $x\geq 0$,
$y\in[0,\alpha^\vee\pi(T)]$, and
$\pi$ be an $\alpha$-dominant path, $\PP_\alpha^x\HH_\alpha^y\pi=\HH_\alpha^{x\wedge y}\pi$.
\end{enumerate}
\end{lemma}
\proof For  all $t\in [0,  T/2]$,
$\PP_\alpha(\eta_1 \star \eta_2)(t)=\PP_\alpha\eta_1(t).$
Furthermore,
$$\begin{array}{l}
\PP_\alpha(\eta_1 \star \eta_2)((T+t)/{2})
\\ =(\eta_1 \star \eta_2)((T+t)/{2})-\min[\inf_{0\leq s\leq T}
\alpha^\vee\eta_1(s),\alpha^\vee\eta_1(T)+\inf_{0\leq s \leq t}
\alpha^\vee \eta_2(s)]\alpha\\
=\eta_1(T) -\inf_{0\leq s\leq T}
\alpha^\vee\eta_1(s)\alpha +\\ \qquad\eta_2(t)-
\min[0,\inf_{0\leq s \leq t} \alpha^\vee\eta_2(s)+
\alpha^\vee\eta_1(T)-\inf_{0\leq s\leq T}
\alpha^\vee\eta_1(s)]\alpha\\
=\PP_\alpha\eta_1(T)+\PP_\alpha^x \eta_2(t).
\end{array}
$$
This proves $(i)$, and 
 $(ii)$ follows from (\ref{eq:px}). 
 Furthermore, $\inf_{0\leq s\leq T}\alpha^{\vee}(\HH_\alpha^y\pi(s))=-y$,
 therefore $(iii)$ follows also from (\ref{eq:px}). 
 $\square$

\begin{prop}\label{prop_FF_star}
  Let $\pi_1, \pi_2$ be $\alpha$-dominant paths, 
  $ x\in[0, \alpha^\vee\pi_1(T)]$,
  $ y \in[0, \alpha^\vee\pi_2(T)]$,
 $z=\min(y,\alpha^\vee\pi_1(T)-x)$ and $r=x+y-z$, then
$$\HH_{\alpha}^{x} \pi_1 \star \HH_{\alpha}^{y}\pi_2=\HH_{\alpha}^{r} 
(\pi_1 \star \HH_{\alpha}^{z}  \pi_2),$$
  \end{prop}
\proof Let $s= \alpha^\vee(\HH_{\alpha}^{x} \pi_1(T))-
\inf_{0\leq t\leq T}
 \alpha^\vee (\HH_{\alpha}^{x} \pi_1)(t)$.
  By lemma \ref{lem:Px}:
 $$\PP_\alpha(\HH_{\alpha}^{x} \pi_1 \star \HH_{\alpha}^{y}\pi_2)=
 \PP_\alpha(\HH_{\alpha}^{x} \pi_1) \star \PP_\alpha^s
 ( \HH_{\alpha}^{y}\pi_2)$$
and  $\PP_\alpha^s\HH_\alpha^y\pi_2=
\HH_\alpha^{s\wedge y}\pi_2$. Since $ \PP_\alpha
\HH_{\alpha}^{x}\pi_1=\pi_1 $ one has
$$\PP_\alpha(\HH_{\alpha}^{x} 
\pi_1 \star \HH_{\alpha}^{y}\pi_2)=\pi_1 \star 
\HH_{\alpha}^{s\wedge y}\pi_2.$$ 
Notice that $s=  \alpha^\vee( \pi_1(T))-x$. 
On the other hand,
$$(\HH_{\alpha}^{x} \pi_1 \star \HH_{\alpha}^{y}\pi_2)(T)=
\HH_{\alpha}^{x} \pi_1(T)+ \HH_{\alpha}^{y}\pi_2(T)=
\pi_1(T)+\pi_2(T)-(x+y)\alpha$$
$$(\pi_1\star \HH_{\alpha}^{s\wedge y}\pi_2)(T)=\pi_1(T)+\pi_2(T)-(s\wedge y)\alpha$$
and we know that $\eta=\HH_\alpha^r\pi$ is characterized by the properties $\PP_\alpha\eta=\pi$ and $\eta(T)=\pi(T)-r\alpha$. Therefore
the proposition holds for $r+s\wedge y=x+y$. $\square$
\medskip

We now prove  that, for $\alpha\in \Sigma$, 
$\eta_1\in L_{\pi_1}, 
\eta_2\in L_{\pi_2}$, for all $s \in \R$,
$$\Theta[e_{\alpha}^s(\eta_1\otimes \eta_2)]=
\EE_{\alpha}^s(\eta_1\star \eta_2).$$
Since $e_{\alpha}^se_{\alpha}^t=e_{\alpha}^{s+t}$ and
 $\EE_{\alpha}^s\EE_{\alpha}^t=\EE_{\alpha}^{s+t}$ it is sufficient to check
 this for $s$ near $0$.
We write $\eta_1=\HH_\alpha^x \pi_1$ and $\eta_2=\HH_\alpha^y \pi_2$ 
where $\pi_1=\PP_\alpha(\eta_1), \pi_2=\PP_\alpha(\eta_2)$ are 
$\alpha$-dominant. By proposition \ref{prop_FF_star}, if 
 $z=\min(y,\alpha^\vee\pi_1(T)-x)$ and $r=x+y-z$, then
$$\EE_{\alpha}^s(\eta_1\star \eta_2)=\EE_{\alpha}^s(\HH_{\alpha}^{x} 
\pi_1 \star \HH_{\alpha}^{y}\pi_2)=\EE_{\alpha}^s\HH_{\alpha}^{r} 
(\pi_1 \star \HH_{\alpha}^{z}  \pi_2).$$
We first show that if
\begin{equation}\label{eq:bf0}\EE_{\alpha}^s(\eta_1\star \eta_2)=
{\bf 0}\end{equation}
then $e_{\alpha}^s(\eta_1\otimes \eta_2)={\bf 0}$. For $|s|$ 
small enough (\ref{eq:bf0}) holds only when $r=0$ and $s>0$ or when  
$s<0$ and\begin{equation}\label{eq:rr}r=\alpha^\vee((\pi_1 \star 
\HH_{\alpha}^{z}  \pi_2)(T))=\alpha^\vee\pi_1(T)+ \alpha^\vee\pi_2(T)-2z.
\end{equation}
If $r=0$, then $z=\min(y,\alpha^\vee\pi_1(T)-x)=x+y$ hence $x=0$ and 
$y \leq \alpha^\vee\pi_1(T)$. 
But 
$$\varepsilon_\alpha(\eta_1\otimes \eta_2)=
\varepsilon_\alpha(\eta_1)-\min(\varphi_\alpha(\eta_1)-
\varepsilon_\alpha(\eta_2),0)=\max(2x+y-\alpha^\vee\pi_1(T), x).$$
(notice that, in general, when  $\pi$ is $\alpha$-dominant, 
$\varepsilon_\alpha(\HH_\alpha^x\pi)=x$ and $\varphi_\alpha
(\HH_\alpha^x\pi)=\alpha^\vee\pi(T)-x$). 
Therefore $\varepsilon_\alpha(\eta_1\otimes \eta_2)=0
$ and $e_{\alpha}^s(\eta_1\otimes \eta_2)={\bf 0}$.
Now, if $r$ is given by (\ref{eq:rr}), then
$$z=\alpha^\vee\pi_1(T)-x+ \alpha^\vee\pi_2(T)-y$$
since $r=x+y-z$. We know that   $\alpha^\vee\pi_2(T)-y\geq 0$,
 hence $z=\min(y,\alpha^\vee\pi_1(T)-x)$ only if
$$z=\alpha^\vee\pi_1(T)-x, \alpha^\vee\pi_2(T)=y, y\geq \alpha^\vee\pi_1(T)-x.$$
Then 
$$\varepsilon_\alpha(\eta_1\otimes \eta_2)=2x+y-\alpha^\vee\pi_1(T).$$
On the other hand, 
$$wt(\eta_1\otimes \eta_2)=wt(\eta_1)+wt( \eta_2)=\pi_1(T)-x\alpha+\pi_2(T)-y\alpha,$$
thus, using $y=\alpha^\vee\pi_2(T)$,
$$\varphi_\alpha(\eta_1\otimes \eta_2)=\varepsilon_\alpha(\eta_1\otimes \eta_2)+\alpha^\vee(wt(\eta_1\otimes \eta_2))=0$$
and $e_{\alpha}^s(\eta_1\otimes \eta_2)={\bf 0}$ when $s<0$.

We now consider the case where (\ref{eq:bf0}) does not hold. Then 
for $s$ small enough,
$$\EE_{\alpha}^s(\eta_1\star \eta_2)=\EE_{\alpha}^s\HH_{\alpha}^{r} 
(\pi_1 \star \HH_{\alpha}^{z}  \pi_2)
=\HH_{\alpha}^{r-s} (\pi_1 \star \HH_{\alpha}^{z}  \pi_2).$$
Using   proposition \ref{prop_FF_star},
  if $s$ is small enough, and $y>\alpha^\vee\pi_1(T)-x,$
then 
$$\HH_{\alpha}^{r-s} (\pi_1 \star \HH_{\alpha}^{z}  
\pi_2)=\HH_{\alpha}^{x-s} \pi_1 \star \HH_{\alpha}^{y}\pi_2=
\Theta(e_\alpha^s (\HH_{\alpha}^{x} \pi_1 \otimes  
\HH_{\alpha}^{y}\pi_2))$$
and if $y<\alpha^\vee\pi_1(T)-x$, then
$$\HH_{\alpha}^{r-s} (\pi_1 \star \HH_{\alpha}^{z}  
\pi_2)=\HH_{\alpha}^{x} \pi_1 \star \HH_{\alpha}^{y-s}\pi_2=
\Theta(e_\alpha^s (\HH_{\alpha}^{x} \pi_1 \otimes  \HH_{\alpha}^{y}\pi_2)).$$
The end of the proof is straightforward. $\square$
  \medskip
  
 By theorem \ref{uniq_iso}, this proves that the family of crystals
   $B(\lambda), \lambda\in \bar C$ is closed. 
  From theorem \ref{thm:exis} and theorem \ref{thmuniq}, we get
  \begin{theorem}\label{uniq_closed}
  When $W$ is a finite Coxeter group, 
  there exists one and only one  closed family of highest 
  weight normal continuous crystals $B(\lambda), \lambda \in \bar C$.
  \end{theorem}

\subsection{Action of  $W$ on the Littelmann crystal}

Following Kashiwara \cite{kash93}, \cite{kashbook} and Littelmann \cite{littel}, we show that we can define an action of the Coxeter group on each crystal $L_\pi$. 
We first notice that for each simple root $\alpha$, 
we can define an involution $S_\alpha$ on the set of
 paths by
$$S_\alpha \eta= \EE_\alpha^x \eta\quad\text{for}
\quad x = -\alpha^\vee(\eta(T)). $$
 In particular,
\begin{equation}\label{eq:Ss}
S_{\alpha}\eta(T)=s_\alpha(\eta(T)).
\end{equation}
\begin{lemma}
Let $\eta\in C^0_T(V)$ and $ \alpha \in \Sigma$ 
such that $\alpha^\vee(\eta(T)) <0$. 
For each $\gamma \in C^0_T(V)$ there exists
 $m \in \N$ such that, for all $ n\geq 0$,
$$\PP_\alpha(\gamma\star \eta^{\star (m+n)})
=\PP_\alpha(\gamma\star \eta^{\star m})\star
 S_\alpha(\eta)^{\star n}.$$
\end{lemma}
\proof By lemma \ref{lem:Px}, 
$$\PP_\alpha(\gamma\star \eta^{\star {(n+1)}})
=\PP_\alpha(\gamma\star \eta^{\star n})\star \PP_\alpha^x(\eta)$$
where $$x=\alpha^\vee(\gamma\star \eta^{\star n})(T)-\min_{0 \leq s \leq T}\alpha^\vee(\gamma\star \eta^{\star n})(s).$$
Let $\gamma_{\min}=\min_{0 \leq s \leq T}
\alpha^\vee\gamma(s)$ and $\eta_{\min}=\min_{0 \leq s \leq T}\alpha^\vee\eta(s)$. Since $\alpha^\vee\gamma(T)<0$, there exists $m > 0$ such that for $n \geq m$  one has, 
\begin{eqnarray*}
\min_{0 \leq s \leq T}\alpha^\vee(\gamma\star \eta^{\star n})(s)&=&\min(\gamma_{\min}, \alpha^\vee(\gamma(T)+k\eta(T))+\eta_{\min}; 0 \leq k \leq n-1)\\
&=& \alpha^\vee(\gamma(T)+(n-1)\eta(T))+\eta_{\min}.
\end{eqnarray*}
Using that  $(\gamma\star \eta^{\star m})(T)=\gamma(T)+m \eta(T)$ we have
$x =\alpha^\vee\eta(T)-\eta_{\min}$.
In this case, 
$\PP_\alpha^x(\eta)=S_\alpha(\eta)$,
which proves the lemma by induction on $n \geq m$. $\square$

\begin{theorem}\label{theo:action}There is an action $ \{S_w, w \in W\} $ of the Coxeter group $W$ on each $L_\pi$ such that $S_{s_\alpha}=S_\alpha$  when $\alpha$ is a simple root.
\end{theorem}
\proof By Matsumoto's lemma, 
it suffices to prove that the transformations 
$S_\alpha$ satisfy to the braid relations. Therefore we can assume that $W$ is a dihedral group $I(q)$. Consider
two
roots $\alpha,\beta$ generating $W$.
Let $\eta$ be a path, 
there exists a sequence 
$(\alpha_i)=\alpha,\beta,\alpha,\ldots $ or
$\beta,\alpha,\beta,\ldots$ such that
$s_{\alpha_1}s_{\alpha_2}\ldots s_{\alpha_r}
\eta(T)\in -\bar C$.
Let $\tilde\eta=S_{\alpha_1}S_{\alpha_2}\ldots S_{\alpha_r}
\eta$.
 Let $s_{\alpha_q}\cdots
  s_{\alpha_1}$ be a reduced decomposition. We show by
  induction on $k\leq q$ that there exists
   $m_k \geq0$ and a path $\gamma_k$ such that
  \begin{equation}\label{seta}
  \PP_{\alpha_k}\cdots 
\PP_{\alpha_1}(\tilde\eta^{\star (m_k+ n)})=
\gamma_k \star (S_{\alpha_k}\cdots S_{\alpha_1}\tilde\eta)^{\star n}
\end{equation}
For $k=1$, this is the preceding 
lemma. Suppose that this holds for some $k$. Then
$$\alpha_{k+1}^\vee(S_{\alpha_k}\cdots
S_{\alpha_1}\tilde \eta(T))\leq0$$
(cf. Bourbaki, \cite{bo2}, ch.5, no.4, Thm.\ 1).
Thus, by the lemma,  there exists
 $m$ such that, for $n \geq 0$,
$$\PP_{\alpha_{k+1}}(\gamma_k \star 
(S_{\alpha_k}\cdots S_{\alpha_1}\tilde\eta)^{\star (m+n)})=
\PP_{\alpha_{k+1}}(\gamma_k\star (S_{\alpha_k}\cdots
 S_{\alpha_1}\tilde\eta)^{\star m})\star (S_{\alpha_{k+1}}
 S_{\alpha_k}\cdots S_{\alpha_1}\tilde\eta)^{\star n}$$
Hence, by the induction hypothesis,
 if $\gamma_{k+1}=\PP_{\alpha_{k+1}}
 (\gamma_k\star (S_{\alpha_k}\cdots S_{\alpha_1}\tilde\eta)^{\star m})$, then
$$\PP_{\alpha_{k+1}}\PP_{\alpha_k}\cdots \PP_{\alpha_1}
((\tilde\eta^{\star (m_k+ m+n)})=\gamma_{k+1}\star
 (S_{\alpha_{k+1}}S_{\alpha_k}\cdots S_{\alpha_1}\tilde 
 \eta)^{\star n}$$
We apply (\ref{seta}) with $k=q$, then there exists two reduced
decompositions, and we see that 
$S_{\alpha_{q}}S_{\alpha_{q-1}}\cdots S_{\alpha_1}\tilde 
 \eta$ does not depend on the reduced decomposition because
 the left hand side does not, 
 by the braid relations for the $\PP_\alpha$. This implies easily
that $S_{\alpha_{q}}S_{\alpha_{q-1}}
 \cdots S_{\alpha_1}
 \eta$ also does not depend on the reduced decomposition.
 $\square$.
\medskip

Using the crystal isomorphism between 
$L_\pi$ and the crystal $B(\pi(T))$ we see that 
\begin{cor}
The Coxeter group $W$ acts on each crystal 
$B(\lambda)$, where $\lambda\in \bar C$, in such a
 way that, for $s=s_\alpha$ in $S$, and $b \in B(\lambda)$,
$$S_\alpha(b)=e_\alpha^x(b), 
\mbox { where } x=-\alpha^\vee (wt(b)).$$
\end{cor}
Notice that these $S_\alpha$ are not crystal morphisms. 

 \subsection{Sch\"utzenberger involution} \label{schutz}
 The classical
   Sch\"utzenberger involution
     associates to a Young tableau $T$ another Young tableau
     $\hat T$ of the same shape. If $(P,Q)$ is the pair associated 
     by Robinson-Schensted-Knuth (RSK) algorithm to the word $u_1\cdots u_n$ in the letters $1,\cdots,k$, 
     then $(\hat P, \hat Q)$ is the pair associated with 
     $u_n^*\cdots u_1^*$ where $i^*=k+1-i$, see e.g.\ Fulton \cite{Fulton}.
      It is remarkable that $\hat P$ depends only on $P$, 
      and that $\hat Q$ depends only on $Q$. We will establish 
      an analogous property for the analogue of the Sch\"utzenberger
       involution defined in \cite{bbo}
      for finite Coxeter groups. 
      The crystallographic case has been recently investigated
       by Henriques and Kamnitzer \cite{hen-kamn1}, 
        \cite{hen-kamn2},  and Morier-Genoud \cite{morier}.

   For any path $\eta \in C^0_T(V)$, let
   $\kappa \eta(t)=\eta(T-t)-\eta(T), \;\; 0 \leq t \leq T,$
 and
   $$S\eta=-w_0\kappa \eta.$$
 Since $w_0^2=id$, $S$ is  an involution of $C^0_T(V)$.  The following is proved in \cite{bbo}.
   \begin{prop}\label{SQ}
   For any $\eta\in C^0_T(V),$
   $\PP_{w_0}S\eta(T)=\PP_{w_0}\eta(T).$
   \end{prop} As remarked in \cite{bbo}, this implies that the transformation on dominant paths   $$\pi \mapsto I\pi=\PP_{w_0} S\pi$$
    gives the analogue of the Sch\"utzenberger involution on the $Q'$s. We will consider the action on the crystal itself, i.e.\  the analogue of the Sch\"utzenberger involution on the $P's$.  
    For each dominant path $\pi\in C^0_T(V)$ the crystals $L_\pi$ and $L_{I\pi}$ are isomorphic, since $\pi(T)=I\pi(T)$. Therefore there is a unique isomorphism $J_\pi:L_\pi\to L_{I\pi}$, it satisfies $J_\pi(\pi)=I\pi$. For each path $\eta \in C^0_T(V)$, let $J(\eta)=J_{\pi}(\eta)$, where $\pi=\PP_{w_0}\eta$. This defines  an involutive isomorphism of crystal  $J:C^0_T(V)\to C^0_T(V).$
    We will see that $$\tilde S=J\circ S$$
    is the analogue of the Sch\"utzenberger involution on crystals. Although $\tilde S$ is not a crystal isomorphism, and contrary to $S$,  it conserves the cristal connected components since  $\tilde S(L_\pi)=L_\pi$, for each dominant path $\pi$, this is the main reason for introducing it.
    
If $\alpha$ is a simple root, then $\tilde \alpha=-w_0\alpha$ is also a simple root and $\tilde \alpha^\vee=-\alpha^\vee w_0$. The following property is straightforward.  In the $A_n$ case, it was shown  by Lascoux, Leclerc and Thibon \cite{llt}  and  Henriques and Kamnitzer \cite{hen-kamn1}  that it characterizes the Sch\"utzenberger involution.
   \begin{lemma}\label{lem_Schu}
   For any path $\eta$ in $C^0_T(V)$, any $r\in\mathbb R$, and any simple root $\alpha$, one has
  $$ \begin{array}{c}
   \EE_\alpha^r\tilde S \eta= \tilde S\EE_{\tilde \alpha}^{-r} \eta\\
    \varepsilon_{\tilde \alpha}(\tilde S\eta))=\varphi_\alpha(\eta),\,
    \varphi_{\tilde \alpha}(\tilde S\eta))=\varepsilon_\alpha(\eta)\\
\tilde S\eta(T)=w_0\eta(T).
\end{array}$$
   \end{lemma}    
An important consequence of this lemma is that $\tilde S:L_\pi\to L_\pi$ depends only on the crystal structure of $L_\pi$. Indeed, if $\eta =  \EE_{\alpha_1}^{r_1}\cdots  \EE_{\alpha_k}^{r_k}\pi$ then $\tilde S(\eta)=\EE_{\tilde \alpha_1}^{-r_1}\cdots  \EE_{\tilde \alpha_k}^{-r_k}\tilde S(\pi)$ and $\tilde S(\pi)$ is the unique element of $L_\pi$ which has the lowest weight $w_0\pi(T)$, namely $S_{w_0}\pi$, where  $S_{w_0}$ is given by theorem \ref{theo:action}. 
In particular, using the isomorphism between $L_\pi$ and $B(\lambda)$ where $\lambda=\pi(T)$, we can transport the action of $\tilde S$ on each $B(\lambda), \lambda \in \bar C$.

Notice that $S \circ J$ also satisfies to this lemma. Therefore, by uniqueness, 
$$S \circ J=J \circ S$$
thus $\tilde S$ is an involution. Following Henriques and Kamnitzer 
        \cite{hen-kamn2}, let us show:
                         \begin{theorem}\label{crysiso}
                 The map  $\tau:C^0_T(V)\to C^0_T(V)$ defined by 
                 $$\tau(\eta_1\star \eta_2)=\tilde S (\tilde S\eta_2\star \tilde S\eta_1)$$
                  is an involutive crystal isomorphism.
                  \end{theorem}
           \proof  Remark first that any path can be written uniquely as the concatenation of two paths, hence $\tau$ is well defined, furthermore 
          $S(\eta_1\star \eta_2)=S(\eta_2)\star S(\eta_1)$, therefore, since $\tilde S=SJ=JS$, and $S$ is involutive,
          $$\tau(\eta_1\star \eta_2)=J S ( SJ\eta_2\star SJ\eta_1)=J S^2 ( J\eta_1\star J\eta_2)=J ( J\eta_1\star J\eta_2).
$$
Consider the map 
$J^{(2)}: C^0_T(V)\to C^0_T(V)$ defined by
$J^{(2)}(\eta_1\star \eta_2)=J\eta_1\star J\eta_2$.
Remark that 
$J^{(2)}=\Theta\circ (J\otimes J)\circ \Theta^{-1}$
where $\Theta:C^0_T(V) \otimes C^0_T(V)\to C^0_T(V)$ is the crystal isomorphism defined in theorem \ref{closed} and $(J\otimes J)(\eta_1\otimes \eta_2)=J(\eta_1)\otimes J(\eta_2)$. Since $J$ is an isomorphism, this implies that $J^{(2)}$ is an isomorphism, thus $\tau =J\circ J^{(2)}$ is an isomorphism.

Let $\tilde S^{(2)}$ be defined by $\tilde S^{(2)}(\eta_1\star \eta_2)=\tilde S(\eta_2)\star \tilde S(\eta_1)$. Then $\tau= \tilde S \circ \tilde S^{(2)}$, and, since $\tilde S$ is an involution, the inverse of $\tau $ is  $ \tilde S^{(2)}\circ \tilde S$. So to prove that $\tau$ is an involution we have to show that  $ \tilde S \circ \tilde S^{(2)}=\tilde S^{(2)}\circ \tilde S$. 
Both these maps are crystal isomorphisms, so it is enough to check that for any $\eta\in C_T^0(V)$, the two paths  $(\tilde S \circ \tilde S^{(2)})(\eta)$ and $(\tilde S^{(2)}\circ \tilde S)(\eta)$ are in the same connected crystal component. Since $\tilde S$ conserves each connected component, $\eta$ and $\tilde S(\eta)$ on the one hand, and $\tilde S^{(2)}(\eta)$ and $\tilde S ( \tilde S^{(2)}(\eta))$ on the other hand, are in the same component. Therefore is it sufficient to show that if $\eta$ and $\mu$ are in  the same component then $\tilde S^{(2)}(\eta)$ and $\tilde S^{(2)}(\mu)$ are in the same component. Let us write $\eta=\eta_1\star \eta_2$. Then if $\mu=\EE^r_\alpha(\eta)$, $\sigma=\varphi_\alpha(\eta_1)-\varepsilon_\alpha(\eta_2)$ and $\tilde \sigma=-\sigma$,
           $$ \tilde S(\EE^{\min(r,-\sigma)+\sigma^+}_\alpha \eta_2) = \EE^{-\min(r,-\sigma)-\sigma^+}_{\tilde \alpha}  \tilde S\eta_2=\EE^{\max(-r,\tilde\sigma)-\tilde\sigma^-}_{\tilde \alpha} \tilde S\eta_2$$
          and  $$
      \tilde S(\EE^{\max(r,-\sigma)-\sigma^-}_\alpha\eta_1)= \EE^{-\max(r,-\sigma)+\sigma^-}_{\tilde \alpha}\tilde \eta_1=
 \EE^{\min(-r,-\tilde\sigma)+\tilde\sigma^+}_{\tilde \alpha}\tilde S\eta_1$$
therefore
\begin{eqnarray*}
\tilde S^{(2)}(\mu)&=& \tilde S^{(2)}(\EE^{\max(r,-\sigma)-\sigma^-}_\alpha\eta_1\star\EE^{\min(r,-\sigma)+\sigma^+}_\alpha \eta_2)\\
&=&\tilde S(\EE^{\min(r,-\sigma)+\sigma^+}_\alpha \eta_2) \star \tilde S(\EE^{\max(r,-\sigma)-\sigma^-}_\alpha\eta_1))\\
&=& \EE^{\max(-r,\tilde\sigma)-\tilde\sigma^-}_{\tilde \alpha} \tilde S\eta_2 \star  \EE^{\min(-r,-\tilde\sigma)+\tilde\sigma^+}_{\tilde \alpha}\tilde S\eta_1\\
&=&\EE^{-r}_{\tilde \alpha}(\tilde S\eta_2 \star \tilde S\eta_1)\\
&=&\EE^{-r}_{\tilde \alpha}\tilde S^{(2)}(\eta).
\end{eqnarray*} 
So in this case $\tilde S^{(2)}(\mu)$ and $\tilde S^{(2)}(\eta)$ are in the same component. One concludes easily by induction. $\square$
\medskip

   We can now define an involution $\tilde S_\lambda$ on each continuous crystal of the family $\{B(\lambda), \lambda \in \bar C\}$   by transporting the action of $\tilde S$ on $C^0_T(V)$. Let  $\lambda, \mu \in \bar C$. For $b_1\in B(\lambda)$ and $b_2\in B(\mu)$ let 
$$\tau_{\lambda,\mu}(b_1\otimes b_2)=\tilde S_\gamma (\tilde S_\mu b_2\otimes \tilde S_\lambda b_1)$$
where $\gamma\in \bar C$ is such that $\tilde S_\mu b_2\otimes \tilde S_\lambda b_1 
\in B(\gamma).$
 
\begin{theorem}
For $\lambda, \mu \in \bar C$, the map $$\tau_{\lambda,\mu}:B(\lambda)\otimes B(\mu)\to B(\mu)\otimes B(\lambda)$$
is a crystal isomorphism.
\end{theorem}
  This follows  from theorem \ref{crysiso}.                
 As in the construction of Henriques and Kamnitzer  \cite{hen-kamn1},  \cite{hen-kamn2} these isomorphisms do not obey the axioms for a braided monoidal category, but instead we have that:\begin{enumerate} 
\item  $ \tau_{\mu,\lambda} \circ \tau_{\lambda,\mu} = 1; $
\item The following diagram commutes:
\begin{equation*}   
\xymatrix{
B(\lambda) \otimes B(\mu) \otimes B(\sigma) 
\ar[d]_{\tau_{(\lambda,\mu)} \otimes 1} 
\ar[r]^{1 \otimes \tau_{ (\mu, \sigma)}} & 
B(\lambda) \otimes B(\sigma) \otimes B(\mu) 
\ar[d]^{\tau_{(\lambda, (\sigma,\mu))}}\\
B(\mu) \otimes B(\lambda) \otimes B(\sigma) 
\ar[r]_{\tau_{((\mu,\lambda), \sigma)}} & 
B(\sigma) \otimes B(\mu) \otimes B(\lambda) \\
}
\end{equation*}
\end{enumerate}
which makes of $B(\lambda), \lambda \in \bar C$, a coboundary category.

\section{The Duistermaat-Heckman measure and Brownian motion}
\subsection{} In this section, we  consider a finite Coxeter group, with a
realization in some Euclidean space $V$ identified with its dual
so that, for each root $\alpha$, 
$\alpha^{\vee}=\frac{2\alpha}{\Vert \alpha\Vert^2}$.
We will introduce an analogue, for continuous crystals, of the Duistermaat-Heckman measure, compute its Laplace transform (the analogue of the Harish-Chandra formula), and study its connections with Brownian motion.
\subsection{Brownian motion and the Pitman transform}
Fix a reduced decomposition of the longest word
$$w_0=s_1 s_2 \cdots s_q$$
and let ${\bf i}=(s_1,\cdots,s_q)$.
 Recall that for any $\eta \in C^0_T(V)$, 
 its string parameters 
 $x=(x_1,\cdots,x_q)=\varrho_{\bf i}(\eta)$ satisfy 
\begin{equation}\label{p-ineq1}
0\le x_i\le \alpha_{s_i}^{\vee}(\lambda-\sum_{j=1}^{i-1} x_j\alpha_{s_j}), \qquad i\le q;
\end{equation}
where $\lambda=\PP_{w_0}\eta(T)$.
For each simple root $\alpha$ choose a reduced decomposition
${\bf i_\alpha}=(s_1^\alpha,\cdots,s_q^\alpha)$ such that $s_1^\alpha=s_\alpha$ and denote the corresponding string parameters $\varrho_{\bf i_\alpha}(\eta)$  by 
$(x_1^\alpha,\cdots,x_q^\alpha)$.
Using the map ${\phi_{\bf i}^{\bf i_\alpha}}$  
given by  theorem \ref{piecewise} we obtain
 a continuous piecewise linear  function $\Psi^{\bf i}_\alpha:\R^q \to \R$ such that 
\begin{equation}\label{x1}x_1^\alpha=\Psi^{\bf i}_\alpha(x).\end{equation}
Of course 
\begin{equation}\label{p-ineq3}
\Psi^{\bf i}_\alpha(x) \ge 0, \qquad \text{for all} \  \alpha\in\Sigma .
\end{equation}
Denote by $M_{\bf i}$ the set of $(x,\lambda)\in\R_+^q\times C$ which satisfy the inequalities (\ref{p-ineq1})
and (\ref{p-ineq3}),  and set 
\begin{equation}
\label{milambda}M_{\bf i}^\lambda=\{x\in\R_+^q:\ (x,\lambda)\in M_{\bf i}\}.
\end{equation}
Let $\Prob$ be a probability measure on 
$C^0_T(V)$ under which 
$\eta$ is a standard Brownian motion in $V$. 
We recall the following theorem from \cite{bbo}.
\begin{theorem}\label{P-brown}
The stochastic process $\PP_{w_0}\eta$ is a Brownian motion in $V$
conditioned, in Doob's sense, to stay in the Weyl chamber 
$\bar C$.
\end{theorem}
This means that  $\PP_{w_0}\eta$ is the $h$-process of
 the standard Brownian motion in $V$ killed when it exits $\bar C$, for the
 harmonic function
$$h(\lambda)=\prod_{ \alpha \in R_+} \alpha^\vee (\lambda) ,$$
 for $\lambda \in V$, where $R_+$ is the set of all positive roots. Let
 $c_t=t^{q/2}\int_V e^{-\|\lambda\|^2/2t}\ d\lambda$ and $$k=c_1^{-1}\int_C h(\lambda)^2 e^{-\|\lambda\|^2/2} \ d\lambda.$$
\begin{theorem}\label{p-unif}
For $(\sigma,\lambda)\in M_{\bf i}$,
\begin{equation}\label{uniform}
\Prob(\varrho_{\bf i}(\eta) \in d\sigma,
 \mathcal P_{w_0}\eta(T) \in d\lambda) 
= c_T^{-1}  h(\lambda) e^{- \| \lambda \|^2 /2T} \ d\sigma
\ d\lambda . 
\end{equation}
The conditional law of $\varrho_{\bf i}(\eta)$, given $(\mathcal P_{w_0} \eta(s), s\le T)$ and 
$\mathcal P_{w_0} \eta(T)=\lambda$, is the normalized
Lebesgue measure 
 on $M_{\bf i}^\lambda$, and the volume 
 of $M_{\bf i}^\lambda$ 
is $k^{-1}h(\lambda)$.
\end{theorem}

This theorem has the following interesting corollary, 
which gives a new proof of the fact that the set $C_{\bf i}^\pi$
 depends only on $\pi(T)$, and is polyhedral.
\begin{cor}\label{cor_cone} For any dominant path $\pi$, let $\lambda=\pi(T)$, then
$C_{\bf i}^\pi=M_{\bf i}^\lambda,$ and
$$C_{\bf i}=\{x \in \R_+^q; \Psi^{\bf i}_\alpha(x) \ge 0, \mbox{ for all } \alpha\in\Sigma\}.$$
\end{cor}
\proof It is clear that $C_{\bf i}^\pi$ is contained in $M_{\bf i}^\lambda$ and the theorem 
implies that  $C_{\bf i}^\pi$, equal by definition to the set of 
$\varrho_{\bf i}(\eta)$  when $\PP_{w_0}\eta=\pi$,  contains 
 $M_{\bf i}^\lambda$. The description of $C_{\bf i}$ follows, since  $C_{\bf i}=\cup\{C_{\bf i}^\pi, \pi \text{ dominant path}\}. $ $\qed$

Theorem \ref{p-unif} is proved in section \ref{pf-punif}.
\subsection{The Duistermaat-Heckman measure}
Let $G$ be a compact semi-simple Lie group with maximal torus $T$. If
${\mathcal O}_{\lambda}$ is a coadjoint orbit of $G$, corresponding to a dominant
regular weight, endowed with its canonical symplectic structure $\omega$, 
then this maximal torus
acts  on the symplectic manifold $({\mathcal O}_{\lambda},\omega)$, and the image of the
Liouville measure on ${\mathcal O}_{\lambda}$ by the moment map, which takes values in 
the dual of the Lie algebra of $T$, is called the Duistermaat-Heckman measure.
It is proved in \cite{ab} that this measure is the image of the Lebesgue measure on the
Berenstein-Zelevinsky polytope by an affine map. In analogy with this case, we define
for a realization of a  finite Coxeter group,
 the Duistermaat-Heckman measure, and prove some properties
which generalize the case of crystallographic groups. 
\begin{definition}\label{defDH}
For any $\lambda \in  C$, the Duistermaat-Heckman measure 
$m^\lambda_{DH}$ on $V$ is the image 
of the Lebesgue measure on $M_{\bf i}^\lambda$ (defined by (\ref{milambda})) by the map
 \begin{equation}
 \label{mapp}x=(x_1,\cdots,x_q)\in M_{\bf i}^\lambda 
 \mapsto \lambda-\sum_{j=1}^qx_j\alpha_j\in V.
\end{equation}
\end{definition}
In the following, $V^*$ denotes the complexification of $V$.
\begin{theorem}\label{p-2}
The Laplace transform of the Duistermaat-Heckman measure is given,  
 for $z \in V^*$, by
\begin{equation}\label{harish}
\int_V e^{\langle z, v \rangle} m_{DH}^\lambda(dv) =\frac{\sum_{w\in W} \varepsilon(w)e^{\langle z, w\lambda \rangle}}{h(z)},
\end{equation}
 where $\varepsilon(w)$ is the signature of $w \in W$.

With the notations of theorem \ref{p-unif},
the conditional law of $\eta(T)$, 
given $(\mathcal P_{w_0} \eta(s), 0\leq s\le T)$ and 
$\mathcal P_{w_0} \eta(T)=\lambda$, is the probability measure $\mu_{DH}^\lambda={k \,m^\lambda_{DH}/h(\lambda)}$.
\end{theorem}
Formula \ref{harish} is the analogue, in our setting of the famous formula of Harish-Chandra \cite{harish}.
Theorem \ref{p-2} is proved in section \ref{pf_p-2}.
\begin{prop} The Duistermaat-Heckman measure $m^\lambda_{DH}$ has 
a continuous piecewise polynomial density, invariant under 
$W$ and with support equal to  the convex hull $co(W\lambda)$ of $W\lambda$.
\end{prop}
\proof The measure $m^\lambda_{DH}$ is the image by an affine map of the Lebesgue measure on the convex polytope $C_{\bf i}^\pi$ when $\pi(T)=\lambda$. Therefore it has a piecewise polynomial density and a convex support.  Its Laplace transform is invariant under $W$ so  $m^\lambda_{DH}$ itself  is invariant under $W$.  The support $S(\lambda)$ of $m^\lambda_{DH}/h(\lambda)$  is equal to $\{\eta(T); \eta \in L_{\pi}\}$. Notice that if $\eta$ is in  $L_\pi$, then when $x=\alpha^\vee(\eta(T))$, $\EE_{\alpha}^x\eta$ is in $L_\pi$ and $\EE_{\alpha}^x\eta(T)=s_\alpha \eta(T)$. Starting from $\pi(T)=\lambda$ we thus see that $W\lambda$ is contained in $S(\lambda)$.
So $co(W\lambda)$ is contained in $S(\lambda)$. 
The components of $x \in M_{\bf i}^\pi$ are non negative, 
therefore $co(W\lambda)$ contains $S(\lambda) \cap \bar C$ and,  
by $W$-invariance it contains $S(\lambda)$ itself. 
 $\square$

 \subsection{Proof of theorem \ref{p-unif}}\label{pf-punif}
 
 First we recall some further path transformations which were  
 introduced in~\cite{bbo}.
 For any  positive root $\beta\in R_+$ (not necessarily simple), define $\qb= 
 \pb s_\beta$. Then, for $\psi\in C^0_T(V)$,
 $$\qb\psi(t)=\psi(t)-\inf_{t\ge s\ge 0}\beta^\vee(\psi(t)-\psi(s)) 
 \beta ,\qquad T\ge t\ge 0.$$
 Let  $w_0=s_1s_2\cdots s_q$ be a reduced
 decomposition, and let $\alpha_i=\alpha_{s_i}$. Since
 $ s_{\alpha}\pb=\PP_{s_{\alpha} \beta}s_\alpha$, for roots $\alpha\ne 
 \beta$,
 the following holds
 $$\mathcal Q_{w_0}:=\mathcal P_{w_0} w_0
 =\mathcal Q_{\beta_{1}}\, \ldots \, \mathcal Q_{\beta_{q}},$$
 where
 $\beta_1=\alpha_1,\ \ \beta_i=s_1\ldots s_{i-1}\alpha_i$, when  $i\le  
 q$.
 Set $\psi_q=\psi$ and, for $i\le q$,
 \begin{equation}\label{zetadef}
 \psi_{i-1}=\mathcal Q_{\beta_{i}}\ldots \mathcal Q_{\beta_{q}}\psi  
 \qquad
 y_i=-\inf_{T\ge t\ge 0}\beta_i^{\vee}(\psi_i(T)-\psi_i(t)).
 \end{equation}
 Then $\psi_0=\mathcal Q_{w_0}\psi$ and, for each $i\le q$,
 $$\mathcal Q_{w_0}\psi(T)=\psi_i(T)+\sum_{j=1}^i y_j \beta_j .$$
 Define $\varsigma_{\bf i}(\psi):=(y_1,y_2,\ldots,y_q)$.
 Now let $\eta=w_0\psi$, so that $\mathcal Q_{w_0}\psi=\mathcal P_ 
 {w_0} \eta$.
 Set $\eta_q=\eta$ and, for $i\le q$,
 \begin{equation}\label{xdef}
 \eta_{i-1}=\mathcal P_{\alpha_{i}}\ldots \mathcal P_{\alpha_{q}}\eta  
 \qquad
 x_i=-\inf_{T\ge t\ge 0}\alpha_i^{\vee}(\eta_i(t)).
 \end{equation}
 Then $\eta_0=\mathcal P_{w_0}\eta$ and, for each $i\le q$,
 $$\mathcal P_{w_0}\eta(T)=\eta_i(T)+\sum_{j=1}^i x_j \alpha_j .$$
 The parameters $\varrho_{\bf i}(\eta)=(x_1,\ldots,x_q)$ are related
 to $\varsigma_{\bf i}(\psi)=(y_1,y_2,\ldots,y_q)$ as follows.
 \begin{lemma}\label{x-y}
 For each $i\le q$, we have:
 \begin{enumerate}[(i)]
 \item $\eta_i=s_i\ldots s_1 \psi_i ,$
 \item $x_i=y_i+\beta_i^\vee(\psi_i(T))
 =\beta_i^\vee (\mathcal Q_{w_0} \psi(T)-\sum_{j=1}^{i-1}y_j\beta_j)- 
 y_i ,$
 \item $
 y_i=x_i+\alpha_i^\vee(\eta_i(T))
 =\alpha_i^\vee
 (\mathcal P_{w_0}\eta(T)-\sum_{j=1}^{i-1}x_j\alpha_j) -x_i.$
 \end{enumerate}
 \end{lemma}
 \begin{proof}
 We prove {\it (i)} by induction on $i\le q$.  For $i=q$ it holds because
 $\eta_q=\eta=w_0\psi=w_0\psi_q$ and $s_q\ldots s_1=w_0$.
 Note that, for each $i\le q$, we can write
 $$\mathcal Q_{\beta_i} = \mathcal P_{\beta_i} s_{\beta_i}
 = s_1\ldots s_{i-1} \mathcal P_{\alpha_i} s_i\ldots s_1 .$$
 Therefore, assuming the induction hypothesis $\eta_i=s_i\ldots s_1  
 \psi_i$,
 \begin{eqnarray*}
 \eta_{i-1} &=& \mathcal P_{\alpha_i} \eta_i = \mathcal P_{\alpha_i}  
 s_i\ldots s_1 \psi_i \\
 &=& s_{i-1}\ldots s_1 \mathcal Q_{\beta_i} \psi_i \\
 &=& s_{i-1}\ldots s_1 \psi_{i-1} ,
 \end{eqnarray*}
 as required.
 This implies {\it (ii)}, using $\eta_{i-1}(T)=\eta_i(T)+x_i\alpha_i$
 and $\psi_{i-1}(T)=\psi_i(T)+y_i\beta_i$:
 \begin{eqnarray*}
 2 x_i &=& \alpha_i^\vee(\eta_{i-1}(T)-\eta_i(T)) \\
 &=& \alpha_i^\vee( s_{i-1}\ldots s_1 \psi_{i-1}(T) - s_i\ldots s_1  
 \psi_i(T) ) \\
 &=& \alpha_i^\vee( s_{i-1}\ldots s_1 (\psi_i(T)+y_i\beta_i) - s_i 
 \ldots s_1 \psi_i(T) ) \\
 &=& 2y_i + \alpha_i^\vee( \alpha_i^\vee( s_{i-1}\ldots s_1 \psi_i(T))  
 \alpha_i) \\
 &=& 2y_i + 2\beta_i^\vee(\psi_i(T)) .
 \end{eqnarray*}
 Finally, {\it (iii)} follows immediately from {\it (ii)} and {\it (i)}.
 \end{proof}
 
This lemma shows that, when $W$ is a Weyl group, then $(y_1,\cdots,y_q)$ are the Lusztig coordinates with respect to the decomposition ${\bf i^*}$ of the image of  the path $\eta$ with string coordinates $(x_1,\cdots, x_q)$ with respect to the decomposition ${\bf i}$ under the Schutzenberger involution, where ${\bf i^*}$ is obtained from ${\bf i}$ by the map $\tilde \alpha= -w_0 \alpha$ (see Morier-Genoud \cite{morier}, Cor.\ 2.17). By $(iii)$ of the preceding lemma,
 we can define a mapping $F:M_{\bf i} \to \R_+^q\times C$ such that
 $$(\varsigma_{\bf i}(\psi),\mathcal Q_{w_0} \psi(T))=
 F(\varrho_{\bf i}(\eta),\mathcal P_{w_0}\eta(T)).$$
 Let $L_{\bf i}=F(M_{\bf i})$.
 It follows from $(ii)$ that $F^{-1}(y,\lambda)=(G(y,\lambda),\lambda) 
 $, where
 $$G(y,\lambda)=\beta_i^\vee (\lambda-\sum_{j=1}^{i-1}y_j\beta_j)-y_i .$$
 Thus, $L_{\bf i}$ is the set of $(y,\lambda)\in\R_+^q\times C$ which  
 satisfy
 \begin{equation}\label{ineq1}
 0\le y_i\le \beta_i^{\vee}(\lambda-\sum_{j=1}^{i-1} y_j\beta_j)  
 \qquad (i\le q)
 \end{equation}
 and
 \begin{equation}\label{ineq3}
 \Psi^{\bf i}_\alpha( G(y,\lambda) )\ge 0 \qquad \alpha\in\Sigma .
 \end{equation}
 The analogue of theorem~\ref{piecewise} also holds for the parameters
 $\varsigma_{\bf i}(\psi)=(y_1,y_2,\ldots,y_q)$, and can be proved  
 similarly.
 More precisely, for any two
 reduced decompositions ${\bf i}$ and ${\bf j}$, there is a piecewise  
 linear map
 $\theta_{\bf i}^{\bf j}:\R^q\to\R^q$ such that
 $\varsigma_{\bf j}(\psi) = \theta_{\bf i}^{\bf j}(\varsigma_{\bf i} 
 (\psi))$.
 In particular, for each simple root $\alpha$, we can define a  
 piecewise linear map
 $\Theta^{\bf i}_\alpha:\R^q\to\R$ such that, if ${\bf i}_\alpha=(s_1^ 
 \alpha,\ldots,s_q^\alpha)$
 is a reduced decomposition with $s_1^\alpha=s_\alpha$, and
 $\varsigma_{{\bf i}_\alpha}(\psi)=(y^\alpha_1,y^\alpha_2,\ldots,y^ 
 \alpha_q)$,
 then $y^\alpha_1=\Theta^{\bf i}_\alpha(y)$ where  $\varsigma_{\bf i}(\psi)=(y_1,y_2,\ldots,y_q)$.
 By lemma \ref{x-y}, we have
 \begin{equation}
 \Theta^{\bf i}_\alpha(y) = \alpha^{\vee} (\lambda) - \Psi^{\bf i}_ 
 \alpha(G(y,\lambda) ),
 \end{equation}
 and the inequalities (\ref{ineq3}) can be written as
 \begin{equation}\label{ineq3a}
 \alpha^{\vee} (\lambda) - \Theta^{\bf i}_\alpha( y )\ge 0 \qquad  
 \alpha\in\Sigma .
 \end{equation}
 As in~\cite{bbo}, we extend the definition of $\qb$ to two-sided paths.
 Denote by $C^0_{\mathbb R}(V)$ the set of continuous paths
 $\pi:{\mathbb R}\to V$ such that $\pi(0)=0$ and $\alpha^\vee(\pi(t)) 
 \to\pm\infty$
 as $t\to\pm\infty$ for all simple $\alpha$.
 For $\pi\in C^0_{\mathbb R}(V)$ and $\beta$ a positive root, define
 $\qb\pi$ by
 $$ \qb \pi(t) = \pi(t) + [\omega(t)-\omega(0)]\beta ,$$ where
 $$\omega(t)=-\inf_{t\ge s>-\infty}\beta^{\vee}(\pi(t)-\pi(s)). $$
 It is easy to see that $\qb\pi\in C^0_{\mathbb R}(V)$.
 Thus, we can set $\pi_q=\pi$ and, for $i\le q$,
 $$\pi_{i-1}=\mathcal Q_{\beta_{i}}\ldots \mathcal Q_{\beta_{q}}\pi  
 \qquad
 \omega_i(t)=-\inf_{s\le t}\beta_i^{\vee}(\pi_i(t)-\pi_i(s)).$$
 Then $$\pi_0=\mathcal Q_{w_0}\pi:=
 \mathcal Q_{\beta_{1}}\, \ldots \, \mathcal Q_{\beta_{q}}\pi$$ and,  
 for each $i\le q$,
 $$\mathcal Q_{w_0}\pi(t)=\pi_{i}(t)+\sum_{j=1}^i[\omega_j(t)-\omega_j 
 (0)]\beta_j .$$
 For each $t\in\R$, write $\omega(t)=(\omega_1(t),\ldots,\omega_q(t))$.
 \begin{lemma}\label{future}
 If $\mathcal Q_{w_0}\pi(t)= \lambda$ and $\omega(t)=y$ then
 $$\inf_{u\ge t} \alpha^{\vee}(\mathcal Q_{w_0}\pi(u)) =
 \alpha^{\vee}(\lambda) - \Theta^{\bf i}_\alpha(y).$$
 \end{lemma}
 \begin{proof}
 It is straightforward to see that
 $$\inf_{u\ge t} \beta_1^{\vee}(\mathcal Q_{w_0}\pi(u)-\mathcal Q_{w_0} 
 \pi(t)) =
 \omega_1(t).$$
 In particular, if ${\bf i}_\alpha=(s_1^\alpha,\ldots,s_q^\alpha)$
 is a reduced decomposition with $s_1^\alpha=s_\alpha$ and we denote the
 corresponding $\omega(\cdot)$ (defined as above) by $\omega^\alpha(\cdot)$, then
 $$\inf_{u\ge t} \alpha^{\vee}(\mathcal Q_{w_0}\pi(u)-\mathcal Q_{w_0} 
 \pi(t)) =  \omega_1^\alpha(t).$$
 Now let $\tau_0=\tau^\alpha_0=t$ and, for $0< i \le q$,
 $$\tau_i=\max\{ s\le \tau_{i-1}:\ \omega_i(s)=0\},\qquad \tau^\alpha_i=\max\{ s\le \tau^\alpha_{i-1}:\ \omega^\alpha_i(s)=0\}.$$
 Set $ \tau=\min\{\tau_q,\tau^\alpha_q\}$.
It is not hard to see that the path $\gamma\in  C^0_{t-\tau}(V)$, defined by
 $$\gamma(s)=\pi(\tau+s)-\pi(\tau),\qquad t-\tau \ge s\ge 0,$$ satisfies
 $ \varsigma_{\bf i} (\gamma) = \omega(t) =  y$ and
 $\varsigma_{{\bf i}_\alpha} (\gamma) = \omega^\alpha(t) $.
Thus, $\omega_1^\alpha(t)=\Theta^{\bf i}_\alpha(y)$, as required.
 \end{proof}
 
 Introduce a probability measure ${\mathbb P}_\mu$ under which $\pi$  
 is a two-sided
 Brownian motion in $V$ with drift $\mu\in C$.  Set $\psi=(\pi(t), t  
 \ge 0)$.
   \begin{prop} \label{bmfacts}
 Under ${\mathbb P}_\mu$, the following statements hold:
 \begin{enumerate}
 \item $\mathcal Q_{w_0}\pi$ has the same law as $\pi$.
 \item For each $t\in\R$, the random variables $\omega_1(t),\ldots, 
 \omega_q(t)$
 are mutually independent and exponentially distributed with
 parameters $2\beta^\vee_1(\mu),\ldots,2\beta^\vee_q(\mu)$.
 \item For each $t\in\R$, $\omega(t)$
 is independent of $(\mathcal Q_{w_0}\pi(s), -\infty < s\le t)$.
 \item The random variables $\inf_{u\ge 0} \alpha^{\vee}(\mathcal Q_
 {w_0}\pi(u))$,
 $\alpha$ a simple root, are independent of the $\sigma$-algebra
 generated by $(\pi(t),t\ge 0)$.
 \end{enumerate}
 \end{prop}
 \proof We see by backward induction on $k=q,\cdots,1$ that
 $\mathcal  Q_{\beta_{k}}\cdots \mathcal Q_{\beta_q}\pi(s), s \leq t$  
 has the
 same distribution as $\mathcal Q_{\beta_{k-1}}\cdots \mathcal Q_
 {\beta_q}\pi(s), s \leq t$, is independent of  $\omega_k(t)$, and
 that $\omega_k(t)$ has an exponential distribution with parameter $2
 \beta^\vee_k(\mu)$. At each step, this is a one dimensional statement
 which can be checked directly or seen as a consequence of the classical
 output theorem for the $M/M/1$ queue (see, for example, \cite{Neil}).
 This implies that (1),
 (2), and (3) hold. Moreover
 $$\inf_{t\geq 0}\beta_1^\vee(\mathcal Q_{w_0}\pi(t))=-\inf_{s \leq 0}
 \beta_1^\vee(\mathcal Q_{\beta_{2}}\cdots \mathcal Q_{\beta_{q}}\pi
 (s))$$
 is independent of $\pi(t), t\geq0$. Since $\beta_1$ can be chosen as
 any simple root $\alpha$, this proves (4).
 \qed
 
 \medskip
 
 Let $T>0$. For $\xi \in C$, denote by $E_\xi$ the event that
 $\mathcal Q_{w_0}\pi(s)\in C-\xi$ for all $s\ge 0$ and by
 $E_{\xi,T}$ the event that
 $\mathcal Q_{w_0}\pi(s)\in C-\xi$ for all $T\ge s\ge 0$.
 By proposition~\ref{bmfacts}, $E_\xi$ is independent of $\psi$.
 
 For $r>0$, define $$B(\lambda,r)=\{\zeta\in V:\ \|\zeta-\lambda\|<r\}$$
 and $$R(z,r)=(z_1-r,z_1+r)\times\cdots\times (z_q-r,z_q+r).$$
 Fix $(z,\lambda)$ in the interior of $L_{\bf i}$ and choose $ 
 \epsilon>0$ sufficiently small
 so that $R(z,\epsilon)$ is contained in $ L_{\bf i} \times B(\lambda, 
 \epsilon)$
 and
 \begin{equation}\label{hyp}
 \inf_{\lambda'\in B(\lambda,\epsilon), z'\in R(z,\epsilon) }
 \alpha^{\vee}(\lambda') - \Theta^{\bf i}_\alpha(z') \ge 0.
 \end{equation}
 \begin{lemma}\label{exit}
 \begin{eqnarray*}
 \lefteqn{
 {\mathbb P}_\mu(\mathcal Q_{w_0} \psi(T) \in B(\lambda,\epsilon),
 \ \varsigma_{\bf i}(\psi)\in R(z,\epsilon))}\\
 &=&  \lim_{C\ni \xi\to 0} {\mathbb P}_\mu(E_\xi)^{-1}
 {\mathbb P}_\mu(\mathcal Q_{w_0} \pi(T) \in B(\lambda,\epsilon),
 \ \omega(T)\in R(z,\epsilon), E_{\xi,T}).
 \end{eqnarray*}
 \end{lemma}
 \begin{proof} An elementary induction argument on the recursive  
 construction of $\mathcal Q_{w_0}$
 shows that, on the event $E_\xi$, there is a constant $C$ for which
 $$\max_{i\le q} \|y_i-\omega_i(T)\|\vee \|\mathcal Q_{w_0}\psi(T)- 
 \mathcal Q_{w_0}\pi(T)\|
 \le C \|\xi\| .$$
 Hence, for $\xi$ sufficiently small,
 \begin{eqnarray*}
 \lefteqn{
 {\mathbb P}_\mu(\mathcal Q_{w_0}\psi(T)\in B(\lambda,\epsilon-C\|\xi 
 \|),\ \varsigma_{\bf i}(\psi) \in
 R(z,\epsilon-C\|\xi\|),\ E_\xi) }\\
 &\leq&
 {\mathbb P}_\mu(\mathcal Q_{w_0}\pi(T)\in B(\lambda,\epsilon),\ \omega 
 (T)\in R(z,\epsilon),\ E_\xi) \\
 &\leq& {\mathbb P}_\mu(\mathcal Q_{w_0}\psi(T)\in B(\lambda,\epsilon+C 
 \|\xi\|),\ \varsigma_{\bf i}(\psi)   \in
 R(z,\epsilon+C\|\xi\|),\ E_\xi) .
 \end{eqnarray*}
 Now $E_\xi$ is independent of $\psi$, and so
 \begin{eqnarray*}
 \lefteqn{
 {\mathbb P}_\mu(\mathcal Q_{w_0}\psi(T)\in B(\lambda,\epsilon-C\|\xi\|),
 \ \varsigma_{\bf i}(\psi) \in R(z,\epsilon-C\|\xi\|))}\\
 &\leq& {\mathbb P}_\mu(E_\xi)^{-1}
 {\mathbb P}_\mu(\mathcal Q_{w_0}\pi(T)\in B(\lambda,\epsilon),\ \omega 
 (T)\in R(z,\epsilon),\  E_\xi)\\
 &\leq& {\mathbb P}_\mu(\mathcal Q_{w_0}\psi(T)\in B(\lambda,\epsilon+C 
 \|\xi\|),\ \varsigma_{\bf i}(\psi) \in
 R(z,\epsilon+C\|\xi\|)) .
 \end{eqnarray*}
 Letting $\xi\to 0$, we obtain that
 \begin{equation}\label{Exi}
 \begin{array}{l}
 \lefteqn{
 {\mathbb P}_\mu(\mathcal Q_{w_0} \psi(T) \in B(\lambda,\epsilon),
 \ \varsigma_{\bf i}(\psi)\in R(z,\epsilon))}\\
 \ \quad =  \lim_{C\ni \xi\to 0} {\mathbb P}_\mu(E_\xi)^{-1}
 {\mathbb P}_\mu(\mathcal Q_{w_0} \pi(T) \in B(\lambda,\epsilon),
 \ \omega(T)\in R(z,\epsilon), E_\xi).
 \end{array}
 \end{equation}
 Finally observe that, on the event
 $$\{ \mathcal Q_{w_0} \pi(T) \in B(\lambda,\epsilon),
 \ \omega(T)\in R(z,\epsilon) \},$$ we have,
 by Lemma \ref{future} and (\ref{hyp}),
 \begin{eqnarray*}
 \inf_{u\ge T} \alpha^{\vee}(\mathcal Q_{w_0}\pi(u))
 &=&
 \alpha^{\vee}(Q_{w_0} \pi(T)) - \Theta^{\bf i}_\alpha\left(\omega(T) 
 \right) \\
 &\ge& \inf_{\lambda'\in B(\lambda,\epsilon), z'\in R(z,\epsilon) }
 \alpha^{\vee}(\lambda') - \Theta^{\bf i}_\alpha(z') \ge 0.
 \end{eqnarray*}
 Thus, we can replace $E_\xi$ by $E_{\xi,T}$ on the right hand side of  
 (\ref{Exi}),
 and this concludes the proof of the lemma.
 \end{proof}
 For $a,b\in C$, define
 $\phi(a,b)=\sum_{w\in W}\varepsilon(w) e^{\langle wa,b\rangle}$.
 \begin{lemma}\label{tech}
 Fix $\mu\in C$.  The functions $f(a,b)=\phi(a,b)/[h(a)h(b)]$
 and $g_\mu(a,b)=\phi(a,b)/\phi(a,\mu)$ have unique analytic  
 extensions to
 $V\times V$.  Moreover, $f(0,b)=k^{-1}$ and $g_\mu(0,b)=h(b)/h(\mu)$.
 \end{lemma}
 \begin{proof} It is clear that the function $\phi$ is analytic in $ 
 (a,b)$, futhermore
 it vanishes on the hyperplanes $\langle\beta,a\rangle=0, \langle 
 \beta,b\rangle=0$,
 for all roots $\beta$.
 The first claim follows from an elementary analytic functions argument.
 In the expansion of  $\phi$ as an entire function,
 the term of homogeneous
 degree $d$ is a polynomial in $a,b$ which is antisymmetric under $W$,  
 therefore
   a multiple of $h(a)h(b)$. In particular the term of lowest degree is a
   constant multiple of $h(a)h(b)$. This constant is nonzero, as can  
 be seen by taking
   derivatives in the definition of $\phi$.
 By l'H\^opital's rule, $\lim_{a\to 0} g_\mu(a,b)=h(b)/h(\mu)$.
 It follows that $\lim_{a\to 0} f(a,b)$ is a constant. To evaluate
 this constant, note that, since $h$ is harmonic and vanishes at the  
 boundary
 of $C$,
 $$\int_C h(\lambda)^2 e^{-\|\lambda\|^2/2} f(a,\lambda) d\lambda =
 e^{|a|^2/2} \int_V e^{-\|\lambda\|^2/2} d\lambda .$$
 Letting $a\to 0$, we deduce that $f(0,\lambda)=k^{-1}$, as required.
 \end{proof}
 
 Denote by $F_\xi$ the event that  $\psi(s)\in C-\xi$ for
 all $s\ge 0$ and by $F_{\xi,T}$ the event that
 $\psi(s)\in C-\xi$ for all $T\ge s\ge 0$.
 \begin{lemma} \label{meander}
 For $B\subset C$, bounded and measurable,
 \begin{eqnarray*}
 \lefteqn{
 \lim_{C\ni \xi\to 0} {\mathbb P}_\mu(F_\xi)^{-1}
 {\mathbb P}_\mu(\psi(T) \in B ,\ F_{\xi,T}) }\\ &=&
 c_T^{-1} h(\mu)^{-1} \int_{B} e^{\langle\mu,\lambda\rangle-\|\mu\|^2  
 T/2 }
 e^{- \|\lambda\|^2/2T } h(\lambda)\  d\lambda .\end{eqnarray*}
 \end{lemma}
 \begin{proof}
 Set $z_T=\int_V e^{-\|\lambda\|^2/2T}\ d\lambda$.
 By the reflection principle,
 $$ {\mathbb P}_\mu(\psi(T) \in d\lambda ,\ F_{\xi,T}) =
 e^{\langle\mu,\lambda\rangle-\|\mu\|^2 T/2}
 \sum_{w\in W}\varepsilon(w)p_T(w\xi,\xi+ \lambda)d\lambda,$$
 where $p_t(a,b)=z_t^{-1}e^{-\|b-a\|^2/2t}$ is the transition density of
 a standard Brownian motion in $V$.
 Integrating over $\lambda$ and letting $T\to\infty$, we obtain (see~ 
 \cite{bbo})
 $${\mathbb P}_\mu(F_\xi)= \sum_{w\in W}\varepsilon(w)
 e^{\langle w\xi-\xi,\mu\rangle}.$$
 Thus, using lemma~\ref{tech} and the bounded convergence theorem,
 \begin{eqnarray*}
 \lefteqn{
 \lim_{C\ni \xi\to 0} {\mathbb P}_\mu(F_\xi)^{-1}
 {\mathbb P}_\mu(\psi(T) \in B ,\ F_{\xi,T}) }\\
 &=&
 z_T^{-1}  \lim_{C\ni \xi\to 0}
 \int_{B}
 e^{\langle\mu,\lambda\rangle-\|\mu\|^2 T/2}
 e^{- (|\xi|^2+|\xi+\lambda|^2)/2T }
 \phi(\xi,\mu)^{-1} \phi\left(\xi,\frac{\xi+\lambda}{T}\right) d 
 \lambda \\
 &=&
 z_T^{-1}  \lim_{C\ni \xi\to 0}
 \int_{B}
 e^{\langle\mu,\lambda\rangle-\|\mu\|^2 T/2}
 e^{- (\|\xi\|^2+\|\xi+\lambda\|^2)/2T }
 g_\mu\left(\xi,\frac{\xi+\lambda}{T}\right) d\lambda \\
 &=& z_T^{-1} h(\mu)^{-1} \int_{B} e^{\langle\mu,\lambda\rangle-\|\mu 
 \|^2 T/2 }
 e^{- |\lambda|^2/2T } h(\lambda/T)\   d\lambda \\
 &=& c_T^{-1} h(\mu)^{-1} \int_{B} e^{\langle\mu,\lambda\rangle-\|\mu 
 \|^2 T/2 }
 e^{- \|\lambda\|^2/2T } h(\lambda) \  d\lambda ,
 \end{eqnarray*}
 as required.
 \end{proof}
 
 Applying lemmas \ref{exit}, \ref{meander} and proposition~\ref 
 {bmfacts}, we obtain
 \begin{eqnarray*}
 \lefteqn{
 {\mathbb P}_\mu(\mathcal Q_{w_0} \psi(T) \in B(\lambda,\epsilon),
 \ \varsigma_{\bf i}(\psi)\in R(z,\epsilon))}\\
 &=&  \lim_{C\ni \xi\to 0} {\mathbb P}_\mu(E_\xi)^{-1}
 {\mathbb P}_\mu(\mathcal Q_{w_0} \pi(T) \in B(\lambda,\epsilon),
 \ \omega(T)\in R(z,\epsilon), E_{\xi,T})\quad \text{(lemma \ref{hyp})}\\
 &=&  \lim_{C\ni \xi\to 0} {\mathbb P}_\mu(E_\xi)^{-1}
 {\mathbb P}_\mu( \omega(T)\in R(z,\epsilon) )
 {\mathbb P}_\mu(\mathcal Q_{w_0} \pi(T) \in B(\lambda,\epsilon),\ E_ 
 {\xi,T})
 \quad\text{(lemma \ref{bmfacts}(3))}\\
 &=&  \lim_{C\ni \xi\to 0} {\mathbb P}_\mu(E_\xi)^{-1}
 {\mathbb P}_\mu( \omega(T)\in R(z,\epsilon) )
 {\mathbb P}_\mu(\psi(T) \in B(\lambda,\epsilon),\ F_{\xi,T})\\
 &=& \prod_{i=1}^q e^{-\beta_i^\vee(\mu)z_i}
 [e^{\epsilon\beta_i^\vee(\mu)}-e^{-\epsilon\beta_i^\vee(\mu)}]
 \lim_{C\ni \xi\to 0} {\mathbb P}_\mu(E_\xi)^{-1}
 {\mathbb P}_\mu(\psi(T) \in B(\lambda,\epsilon),\ F_{\xi,T})\\
 &&\qquad\qquad \text{(lemma \ref{bmfacts} (2))}\\
 &=& \prod_{i=1}^q e^{-\beta_i^\vee(\mu)z_i}
 [e^{\epsilon\beta_i^\vee(\mu)}-e^{-\epsilon\beta_i^\vee(\mu)}]  \\
 & & \qquad \times
 c_T^{-1} h(\mu)^{-1} \int_{B_V(\lambda,\epsilon)} e^{\mu(\lambda')-\| 
 \mu\|^2 T/2 }
 e^{- \|\lambda'\|^2/2T } h(\lambda') \  d\lambda' .\quad\text{(lemma \ref 
 {meander})}
 \end{eqnarray*}
 Now divide by $\|B(y,\epsilon)\| (2\epsilon)^{q}$ and let $\epsilon$  
 tend to zero
 to obtain
 \begin{eqnarray*}
 \lefteqn{ {\mathbb P}_\mu(\mathcal Q_{w_0} \psi(T) \in d\lambda,\  
 \varsigma_{\bf i}(\psi)\in dz)}\\
 &=&  \prod_{i=1}^q e^{-\beta_i^\vee(\mu)z_i} e^{\langle\mu,\lambda 
 \rangle-\|\mu\|^2 T/2}
 c_T^{-1} h(\lambda) e^{-\|\lambda\|^2/2T} \ d\lambda\ dz .
 \end{eqnarray*}
 Letting $\mu\to 0$ this becomes, writing $\mathbb P=\mathbb P_0$,
 \begin{equation}
 \mathbb P(\mathcal Q_{w_0} \psi(T) \in d\lambda,\ \varsigma_{\bf i} 
 (\psi)\in dz)
 =  c_T^{-1} h(\lambda) e^{-\|\lambda\|^2/2T} \ d\lambda\ dz .
 \end{equation}
 Using lemma \ref{x-y},
 it follows that, for $(w,\lambda)$ in the interior of $M_{\bf i}$,
 \begin{equation}\label{p-uniform}
 \mathbb P(\varrho_{\bf i}(\eta)\in dw,\ \mathcal P_{w_0} \eta(T) \in d 
 \lambda)
 =  c_T^{-1} h(\lambda) e^{-\|\lambda\|^2/2T} \ dw\ d\lambda .
 \end{equation}
 
 Under the probability measure $\mathbb P$, $\eta$ is a standard  
 Brownian motion
 in $V$ with transition density given by $p_t(a,b)=z_t^{-1}e^{-\|b-a\| 
 ^2/2t}$.
 By theorem \ref{P-brown}  under $\mathbb P$, $\mathcal P_{w_0} \eta$  
 is a
 Brownian motion in $C$. Its transition density is given, for $\xi, 
 \lambda\in C$,
 by $$q_t(\xi,\lambda)
 =\frac{h(\lambda)}{h(\xi)}\sum_{w\in W}\varepsilon(w)p_t(w\xi, 
 \lambda).$$
 As remarked in \cite{bbo}, this transition density can be extended by  
 continuity
 to the boundary of $C$.  From lemma~\ref{tech} we see that
 $q_T(0,\lambda)=k^{-1}h(\lambda)^2e^{-\|\lambda\|^2/2T}$. Thus,
 \begin{equation}\label{gue}
 \mathbb P(\mathcal P_{w_0} \eta(T)\in d\lambda)=k^{-1}h(\lambda)^2e^{- 
 \|\lambda\|^2/2T}d\lambda.
 \end{equation}
 To complete the proof of the theorem, first note that since $ 
 \varsigma_{\bf i}(\psi)$ is measurable
 with respect to the $\sigma$-algebra
 generated by $(\mathcal Q_{w_0} \psi(u),\ u\ge T)$, $\varrho_{\bf i} 
 (\eta)$
 is measurable with respect to the $\sigma$-algebra
 generated by $(\mathcal P_{w_0} \eta(u),\ u\ge T)$.
 Thus, by the Markov property of $\mathcal P_{w_0} \eta$, the conditional
 distribution of $\varrho_{\bf i}(\eta)$, given $(\mathcal P_{w_0} \eta 
 (s), s\le T)$,
 is measurable with respect to the $\sigma$-algebra generated by $ 
 \mathcal P_{w_0} \eta(T)$.
 Combining this with (\ref{p-uniform}) and (\ref{gue}), we conclude  
 that the conditional law of
 $\varrho_{\bf i}(\eta)$, given $(\mathcal P_{w_0} \eta(s), s\le T)$ and
 $\mathcal P_{w_0} \eta(t)=\lambda$, is almost surely uniform on $M^ 
 \lambda_{\bf i}$,
 and that the Euclidean volume of $M_{\bf i}^\lambda$ is $k^{-1}h 
 (\lambda)$,
 as required.

\subsection{Proof of theorem \ref{p-2}}\label{pf_p-2}

Let $\psi=w_0\eta$ and $\mathcal Q_{w_0}=\mathcal P_{w_0} w_0$.
Denote by $P_t$ (respectively $Q_t$) the semigroup of Brownian motion in $V$
(respectively $C$). Under $\mathbb P$, by \cite[Theorem 5.6]{bbo}, 
$\mathcal Q_{w_0} \psi$ is a Brownian motion in $C$. 
Let $\delta\in C$.  The function $e_\delta(v)=e^{\langle\delta,v\rangle}$ 
is an eigenfunction of $P_t$ and the $e_\delta$-transform of $P_t$ is a Brownian motion 
with drift $\delta$. Setting $\phi_\delta(v)=\sum_{w\in W}\varepsilon(w)e^{\langle w\delta,v\rangle},$
the function $\phi_\delta/h$ is an eigenfunction of $Q_t$ and the 
$(\phi_\delta/h)$-transform of $Q_t$ is a Brownian motion with drift $\delta$ conditioned never 
to exit $C$ (see \cite[Section 5.2]{bbo} for a definition of this process).  
By theorem \ref{p-unif},
 the conditional law of $\eta(T)$, given $(\mathcal P_{w_0} \eta(s),\ s\le T)$
and $\mathcal P_{w_0} \eta(T)=\lambda$, 
is almost surely given by $\mu_{DH}^\lambda$.
It follows that the conditional law of $\psi(T)$, given $(\mathcal Q_{w_0} \psi(s),\ s\le T)$
and $\mathcal Q_{w_0} \psi(T)=\lambda$, is almost surely given by $\mu_{DH}^\lambda$.
Denote the corresponding Markov operator by $K(\lambda,\cdot)=\mu_{DH}^\lambda(\cdot)$.
By \cite[Theorem 5.6]{bbo} we automatically have the intertwining $K P_t= Q_tK $. 
Note that $Ke_\delta$ is an eigenfunction of $Q_t$.  By construction, the $Ke_\delta$-transform 
of $Q_t$, started from the origin, has the same law as $\mathcal Q_{w_0} \psi^{(\delta)}$, 
where $\psi^{(\delta)}$ is a Brownian motion in $V$ with drift $\delta$.  Recalling the proof of 
\cite[Theorem 5.6]{bbo} we note that $\mathcal Q_{w_0} \psi^{(\delta)}$ 
has the same law as a Brownian motion with drift $\delta$ conditioned never to exit $C$.
It follows that $Ke_\delta=\phi_\delta/(c(\delta)h)$, for some $c(\delta)\ne 0$.
Now observe (using lemma \ref{tech} for example) that $\lim_{\xi\to 0} Ke_\delta(\xi)=1$.
Thus, by lemma~\ref{tech}, $c(\delta)=\lim_{\xi\to 0} \phi_\delta(\xi)/h(\xi)=k^{-1} h(\delta)$.
We conclude that
$$\int_V e^{\langle\delta,v\rangle} \mu_{DH}^\lambda(dv) = k 
\frac{\sum_{w\in W}\varepsilon(w)e^{\langle w\delta,\lambda\rangle}}{h(\delta)h(\lambda)}.$$
This formula extends to $\delta\in V^*$ by analytic continuation
 (see lemma~\ref{tech} again),
and the proof is complete.

\subsection{A Littlewood-Richardson property} 
In usual Littelmann path theory, the concatenation of paths is used to describe tensor products of representations, and give a combinatorial formula for the Littlewood-Richardson coefficients.
In our setting of continuous crystals, the representation theory does not exist in general, and the analogue of the Littlewood-Richardson coefficients is a certain conditional distribution  
of the Brownian path. In this section we describe this distribution in theorem \ref{the_cond}.

Let ${\bf i}=(s_1,\ldots, s_q)$ where $w_0=s_1\ldots s_q$ is a reduced
decomposition. For $\eta\in C_T^0(V)$, let $x=\rho_{\bf i}(\eta)$.

For each simple root $\alpha$ choose now 
${\bf j_\alpha}=(s_1^\alpha,\cdots,s_q^\alpha)$, a reduced decomposition of
$w_0$, such that $s_q^\alpha=s_\alpha$, 
and denote the corresponding string parameters of the path $\eta$ by 
$(\tilde x_1^\alpha,\cdots,\tilde x_q^\alpha)=\varrho_{\bf j_\alpha}(\eta)$. As in
(\ref{x1}),  there is a continuous function $\Psi'_\alpha:\R^q \to \R$ such that 
$\tilde x_q^\alpha=\Psi'_\alpha(x).$ 
 Fix   $\lambda,\mu\in C$ and suppose
that    $\lambda+\eta(s) \in C$   for $ 0 \leq  s  \leq T$. Then
$\tilde x^\alpha_q=-\inf_{s\le T}\alpha^{\vee}(\eta(s))  \leq  \alpha ^{\vee}(\lambda).$ In other words,
\begin{equation}\label{InegLR} \Psi'_\alpha(x)    \leq \alpha^\vee(\lambda), \qquad \alpha \in \Sigma.\end{equation}
Let $M_{\bf i}^{\lambda,\mu}$  denote the set of $x \in M_{\bf i}^\mu$ which satisfy the additional constraints
(\ref{InegLR}).
This is a compact convex polytope. Let $\nu^{\lambda,\mu}$
be the uniform probability distribution on $M_{\bf i}^{\lambda,\mu}$ and let
 $\nu_{\lambda,\mu}$ be its image on $V$ by the map 
 $$x=(x_1,\cdots,x_q)\in M_{\bf i}^{\lambda,\mu} \mapsto \lambda+
 \mu-\sum_{j=1}^qx_j\alpha_j\in V.$$
  Let $\eta$ be the Brownian motion in $V$ starting from $0$.
   Observe that, by theorem \ref{piecewise},
    the event $\{{\eta (s) \in C-\lambda  , 0 \leq  s \leq  T} \}$ is
measurable with respect to the  $\sigma$-algebra generated by  $\rho_{\bf i}(\eta)$. Combining
this with theorem \ref{p-unif} we obtain:
\begin{cor}\label{corcondi} The conditional law of $ \rho_{\bf i}(\eta)$, given $\PP_{w_0}\eta(s) , s \leq  T, \PP_{w_0} \eta(T) = \mu$ and $\lambda+\eta (s) \in C$ for $0   \leq s \leq  T$, is $\nu^{\lambda,\mu}$ and  the conditional law  of $\lambda+\eta(T)$ is $\nu_{\lambda,\mu}$.
\end{cor} 
For $s, t \geq 0$ let
 $$(\tau_s \eta)(t)=\eta(s+t)-\eta(s), (\tau_s\PP_{w_0}\eta )(t) = \PP_{w_0}\eta( s+t )- \PP_{w_0}\eta(s).$$
\begin{lemma}
For all $s \geq 0$,
$$\PP_{w_0} (\tau_s\PP_{w_0}\eta) = \PP_{w_0}\tau_s\eta.$$
 \end{lemma}
\proof If $\pi_1,\pi_2 : \R^+\to V$ are continuous path starting at $0$, let $ \pi_1 \star_s \pi_2$
be the path defined by $ \pi_1 \star_s \pi_2(r) = \pi_1(r)$ when $0 \leq r \leq s$ and $ \pi_1 \star_s \pi_2(r) = \pi_1(s)+\pi_2(r-s)$
when $s\leq r$. By lemma \ref{lem:Px}, $\PP_{w_0} (\pi_1 \star_s \pi_2) = \PP_{w_0} (\pi_1) \star_s \tilde \pi_2 $ where $\tilde \pi_2$ is a path such that 
$ \PP_{w_0}(\tilde \pi_2) =\PP_{w_0}(  \pi_2)$. Since $\tau_s(\pi_1 \star_s \pi_2)=\pi_2$, 
this gives the lemma. $\square$
\bigskip

 Let $\gamma_{\lambda,\mu}$ be the 
measure on $C$ given by
$$\gamma_{\lambda,\mu}(dx) = \frac{h(x)}{
h(\lambda)} \nu_{\lambda,\mu}(dx).$$ 
It will follow from theorem \ref{the_cond} that this is a probability measure.
Consider 
 the following $\sigma$-algebra
 $${\mathcal  G}_{s,t}=\sigma(\PP_{w_0} \eta(a), a\leq s, \PP_{w_0} \tau_s\eta(r), r\leq t).$$
The following result is a continuous analogue of the Littelmann interpretation of
the Littlewood-Richardson decomposition of a tensor product. 
 \begin{theorem}\label{the_cond} For $s, t > 0,$ $\gamma_{\lambda,\mu}$  is the conditional distribution of $\PP_{w_0} \eta(s + t)$ given
${\mathcal  G}_{s,t}$, $\PP_{w_0} \eta(s) = \lambda $ and $\PP_{w_0}\tau_s \eta(t) = \mu.$ 
  \end{theorem}
  \proof 
 When $(X_t,(\theta_t),\mathbb P_x)$ is a Markov process with shift $\theta_t$ (i.e. $X_{s+t}=X_s\circ \theta_t$), for any $\sigma(X_r, r \geq 0)$-measurable random variables $Z, Y \geq0$, one has
 $$\E(Z\circ \theta_t|\sigma(X_s, s \leq t, Y\circ \theta_t))=\E_{X_0}(Z|\sigma( Y))\circ \theta_t.$$
Let us apply this relation to  the Markov process $X=\mathcal P_{w_0} \eta$ (see \cite{bbo}). Notice that since $\PP_{w_0}(\tau_sX)=\PP_{w_0}(\tau_0X)\circ \theta_s$, it follows from the lemma that
$${\mathcal  G}_{s,t}=\sigma(X_a, \PP_{w_0}(\tau_0X)(r)\circ \theta_s, a \leq s, r \leq t).$$
Therefore,
  for any Borel nonnegative function $f: V\to \R$,
 $$\E[f(\PP_{w_0} \eta(s+t)| {\mathcal  G}_{s,t}]
=\E_{X_0}[f(X_{t})|\sigma(\mathcal P_{w_0} (\tau_0X)(r), r \leq t)]
\circ \theta_s.$$
One knows (Theorem 5.1 in  \cite{bbo}) that $X$ is the $h$--process 
of the Brownian motion killed at the boundary of $C$. In other words,
 starting from $X_0=\lambda$,  $X$ is the $h$-- process of $\lambda+\eta(t)$ 
 conditionally on $\lambda+\eta(s)\in C$, for $0 \leq s \leq t$.
  It thus follows from corollary \ref{corcondi} that
$$\E_{\lambda}[f(X_{t})|\sigma(\mathcal P_{w_0} (\tau_0X)(r), r \leq t)]=
\frac{1}{h(\lambda)}\int f(x)h(x)\, d\nu_{\lambda,\mu}(x)$$
when $\PP_{w_0} (\tau_0X)(t)=\mu$. This proves that
$$\E[f(\PP_{w_0} \eta(s+t))| {\mathcal  G}_{s,t}]=\int f(x)\, d\mu_{\lambda,\mu}(x)$$
when $\PP_{w_0}\eta(s)=\lambda$ and $\PP_{w_0}\tau_s\eta(t)=\mu$.

\subsection{A product formula}

Consider the Laplace transform of $\mu^\lambda_{DH} $ given, 
for $\lambda \in C, z \in V^*$, by
\begin{equation}\label{eqJl1}J_\lambda(z)=
k\frac{\sum_W \varepsilon(w) e^{\langle z, w\lambda \rangle}}{h(z)h(\lambda)}.\end{equation}
This is an example of a generalized Bessel function, 
following the terminology of Helgason \cite{Helg}
 in the Weyl group case  and Opdam \cite{Opdam} in the general Coxeter case.
  It was a conjecture in Gross and Richards \cite{gross-rich} that these 
  are Laplace transform of  positive measures 
(this also follows from R\"osler \cite{roesler1}).
They are positive eigenfunctions of the Laplace and of the Dunkl operators on the Weyl chamber $C$ with eigenvalue $\|\lambda\|^2$ and Dirichlet boundary conditions and $J_\lambda(0)=1$. 
Let $f_\lambda$ be the density of the probability measure $\mu^\lambda_{DH}$. One has
\begin{equation}\label{eqJl2}\int_V e^{\langle z, v \rangle} f_\lambda(v) 
\,dv=J_\lambda(z).\end{equation}
Let, for $v \in C$,
 $$f_{\lambda,\mu}(v)=
 \frac{1}{h(\mu)} \sum_{w\in W} h(wv)f_\lambda(wv-\mu).$$ 
 It follows from the next result that  $ f_{\lambda,\mu}(v) \geq 0$.
   \begin{theorem}\label{prod_J}
 (i) For $\lambda, \mu \in C$ and $z \in V^*$,
 $$J_\lambda(z)J_\mu(z)=\int_C J_v(z)f_{\lambda,\mu}(v)\, dv.$$
 (ii)  $$\gamma_{\lambda,\mu}(dx)=f_{\lambda,\mu}(x) dx.$$
 \end{theorem}
 \proof
 The first part is given by the following computation,  similar to the one in Dooley et al \cite{Dooley}, we give it for the convenience of the reader.
 It follows from (\ref{eqJl1}) and  (\ref{eqJl2}) that
$$J_\lambda(z)J_\mu(z)=  \int_V e^{\langle z,v\rangle } J_\mu(z)f_\lambda(v)\, dv=k \sum_W \varepsilon(w)\int_V \frac{ e^{\langle z, w\mu+v \rangle}}{h(\mu)h(z)}\, f_\lambda(v)\, dv.$$
Using the invariance of the measure $\mu^\lambda_{DH}$ under $W$, 
 $f_\lambda(wv)=f_\lambda(v)$ for $w \in W$. One has
\begin{eqnarray*}J_\lambda(z)J_\mu(z)
&=&k \sum_W \varepsilon(w)\int_V \frac{ e^{\langle z, w(\mu+v) \rangle}}{h(\mu)h(z)} f_\lambda(v)\, dv\\
&=&k \sum_W \varepsilon(w)\int_V \frac{ e^{\langle z, wv \rangle}}{h(\mu)h(z)} f_\lambda(v-\mu)\, dv\\
&=&\frac{1}{h(\mu)} \int_V J_v(z) h(x)f_\lambda(v-\mu)\, dv\\
&=&\frac{1}{h(\mu)} \sum_{w\in W}\int_{w^{-1}C}  J_v(z) h(v)f_\lambda(v-\mu)\, dv\\
&=&\frac{1}{h(\mu)} \sum_{w\in W}\int_{C}  J_v(z) h(wv)f_\lambda(wv-\mu)\, dv\\
&=&\int_{C}  J_z(v) f_{\lambda,\mu}(v)\, dv\end{eqnarray*}
where we have used that, up to  a set of measure zero,
$V=\cup_{w\in W}w^{-1}C$. This proves {\it (i)}.

Let us now prove {\it (ii)}, using theorem \ref{the_cond}.
 Since $\eta$ is a standard Brownian motion in $V$, 
 $\{\eta(r), r \leq s\}$ and $ \tau_s\eta$ are independent, hence, for   $z \in V^*$,
\begin{eqnarray*}\E(e^{\langle z, \eta(s+t) \rangle})|{\mathcal  G}_{s,t})&=&\E(e^{\langle z, \eta(s) \rangle}e^{\langle z, \tau_s\eta(t)\rangle}|{\mathcal  G}_{s,t})\\
&=& \E(e^{\langle z, \eta(s) \rangle}| \sigma(\mathcal P_{w_0} \eta(a), a\le s))\E(e^{\langle z, \tau_s\eta(t)\rangle}|\sigma(\mathcal P_{w_0} \tau_s\eta(b), b\le t)).\end{eqnarray*}
By  theorem \ref{p-2},
$$J_\lambda(z)= \E(e^{\langle z, \eta(s) \rangle}| \sigma(\mathcal P_{w_0} \eta(a), a\le s) $$when 
 $\mathcal P_{w_0} \eta(s)=\lambda$ and, since $ \tau_s\eta$ and $\eta$ have the same law,
$$J_\mu(z)=\E(e^{\langle z, \tau_s\eta(t)\rangle}|\sigma(\mathcal P_{w_0} \tau_s\eta(b), b\le t))$$ when $\mathcal P_{w_0} \tau_s\eta(t)=\mu $.
 Therefore
$$
\E(e^{\langle z, \eta(s+t) \rangle}|{\mathcal  G}_{s,t})=J_\lambda(z)J_\mu(z).
$$
On the other hand, by lemma \ref{lem:Px},
 ${\mathcal  G}_{s,t}$ is contained in $\sigma(\mathcal P_{w_0} \eta(r), r\le s+t)$, thus
$$ \E(e^{\langle z, \eta(s+t) \rangle}|{\mathcal  G}_{s,t})= 
\E
(\E(e^{\langle z, \eta(s+t) \rangle}|\sigma(\mathcal P_{w_0} \eta(r), r\le s+t))
|{\mathcal  G}_{s,t})$$
$$= \E(J_z(\mathcal P_{w_0} \eta( s+t))|{\mathcal  G}_{s,t}).$$
It thus follows from theorem \ref{the_cond} that 
$$J_\lambda(z)J_\mu(z)=\int J_v(z) \, d\gamma_{\lambda,\mu}(v).$$
Therefore, for all $z \in V^*, $
$$\int J_v(z) \, f_{\lambda,\mu}(v)\, dv=\int J_v(z) \, d\gamma_{\lambda,\mu}(v).$$
By injectivity of  the Fourier-Laplace transform this implies that 
$$d\gamma_{\lambda,\mu}(v)=f_{\lambda,\mu}(v)\, dv. \qquad \square$$

The positive product formula gives a positive answer to a question of 
R\"osler \cite{roesler2} for the radial Dunkl kernel. 
It shows that one can generalize the structure of Bessel-Kingman hypergroup 
to any Weyl chamber,  for the so called geometric parameter.

\section{Littelmann modules and geometric lifting.}
\subsection{}It was observed some time ago by Lusztig that the combinatorics of the canonical basis
is closely related to the geometry of the totally positive varieties. This connection was made precise by 
Berenstein and Zelevinsky in \cite{beze2}, in terms of transformations  called
 "tropicalization" and "geometric lifting".  In this section we show how some simple considerations 
on Sturm-Liouville equations lead to a natural way of lifting Littelmann paths, which take values in a Cartan algebra,  
to the corresponding Borel group. Using this lift, an application of Laplace's method explains the connection between the canonical basis
and the totally positive varieties.

This section is organized as follows.
We  first recall the notions of tropicalization and geometric lifting in the next subsection, as well as the connection between the totally positive varieties and the canonical basis. Then we make some observations on Sturm-Liouville equations and their relation to Pitman transformations and the Littelmann path model in type $A_1$. 
We extend  these observations to higher rank in the next subsections then we show, in theorem \ref{the:string}
 how they explain  the link between string parametrization of the canonical basis and the totally positive varieties.
\subsection{Tropicalization and geometric lifting}
A subtraction free rational expression is a rational function in several variables, with positive real coefficients and without minus sign, e.g.
$$t_1+2t_2 /t_3,(1-t^3)/(1-t)\ \text{or}\ 1/(t_1t_2+3t_3t_4)$$ are such expressions, but not $t_1-t_2$.
Any such expression $F(t_1,\ldots,t_n)$ can be tropicalized, which means that 
$$F_{trop}(x_1,,\ldots,x_n)=\lim_{\varepsilon\to 0_+}\varepsilon\log(F(e^{x_1/\varepsilon},\ldots,
e^{x_n/\varepsilon}))$$
exists as a piecewise linear function of the real variables $(x_1,\ldots,x_n)$, and is given by an expression in the maxplus algebra over the variables $x_1,\ldots,x_n$. More precisely, 
the tropicalization $F\to F_{trop}$
 replaces each occurence of $+$ by $\vee$ (the $\max$ 
 sign $x\vee y=\max(x,y)$), each product by a $+$, 
and each fraction  by a $-$, and each positive real number by 0. For example 
the three expressions above give
$$(t_1+2t_2/t_3)_{trop}=x_1\vee(x_2 -x_3),((1-x^3)/(1-x))_{trop}= 0\vee x\vee 2x, $$
and $$
 (1/(t_1t_2+3t_3t_4))_{trop}=-\left((x_1+x_2)\vee(x_3+x_4)\right).$$
Tropicalization is not a one to one transformation, and there exists in general many 
subtraction free rational expressions which have the same tropicalization. Given some expression $G$ in the maxplus algebra, any
 subtraction free rational expression  
 whose tropicalization is $G$ is called a geometric 
 lifting of $G$, cf \cite{beze2}.

\subsection{Double Bruhat cells and string coordinates}\label{DBCSC}
We recall some  
standard terminology, using the notations of \cite{beze2}. We
consider
a simply connected complex semisimple Lie group $G$, associated with a  
root system $R$.
Let $H$ be a maximal torus, and $B,B_-$ be corresponding
opposite Borel subgroups with unipotent
radicals $N,N_-$.
Let $\alpha_i,i\in I,$ and $\alpha_i^{\vee},i\in I,$ be the
simple positive roots and coroots,  and $s_i$ the
corresponding reflections
in the Weyl group $W$.
Let $e_i,f_i, h_i,i\in I,$ be  Chevalley
generators of the Lie algebra of $G$. 
One can  choose
representatives $\overline w\in G$ for  $w\in W$ by putting
$\overline{s_i}=\exp(-e_i)\exp(f_i)\exp(-e_i)$ and
$\overline{vw}=\overline v\,\overline w$ if
$l(v)+l(w)=l(vw)$ (see \cite{foze} (1.8), (1.9)).  The Lie algebra of  
$H$, denoted by
     $\mathfrak h$ has a Cartan decomposition $\mathfrak
    h=\mathfrak a+i\mathfrak a$ such that the roots
     $\alpha_i$ take real values on the real vector space $\mathfrak
    a$. Thus $\mathfrak a$ is generated by $\alpha_i^{\vee}, i\in I$ and  
its dual $\mathfrak a^*$ by $\alpha_i, i\in I$.

A double Bruhat cell is associated with each pair $u,v\in W$ as
$$L^{u,v}=N{\bar u}N\cap B_-\bar vB_-.$$
We will be mainly interested here in the double Bruhat cells $L^{w,e}$. 
As shown in \cite{beze2},
given a reduced decomposition $w=s_{i_1}\ldots s_{i_q}$ 
every element $g\in L^{w,e}$ has a unique decomposition
$g=x_{-i_1}(r_1)\ldots x_{-i_q}(r_q)$ with non zero complex numbers
$(r_1,\ldots, r_q)$, where $x_{-i}(s)=\varphi_i\begin{pmatrix}s&0\\
1&s^{-1}\end{pmatrix}$ (where $\varphi_i$ is the embedding of $SL_2$ into $G$
given by $e_i,f_i,h_i$). The totally positive part of the double Bruhat cell corresponds to the set of 
elements with positive real
 coordinates.
For two different reduced decompositions,
the transition map between two sets of coordinates of the form
$(r_1,\ldots, r_q)$ is given by a subtraction free rational map, which is therefore 
 subject to tropicalization.

As a simple example consider the case of type
$A_2$. Let the coordinates on the double Bruhat cell $L^{w_0,e}$ for the reduced
decompositions $w_0 =s_1s_2s_1$, and 
 $w_0=s_2s_1s_2$ be $(u_1,u_2,u_3)$ and $(t_1,t_2,t_3)$ respectively, 
 then 
\begin{equation}
\begin{pmatrix}t_2&0&0\\t_1&t_1t_3/t_2&0\\1&t_3/t_2+1/t_1&1/t_1t_3\end{pmatrix}=
\begin{pmatrix}u_1u_3&0&0\\u_3+u_2/u_1&u_2/u_1u_3&0\\1&1/u_3&1/u_2\end{pmatrix}
\end{equation}
which yields transition maps
$$
\begin{array}{rcl}\label{tropi} t_1&=&u_3+u_2/u_1\\ t_2&=&u_1u_3\\ t_3&=&u_1u_2/(u_2+u_1u_3).\end{array}
$$

On the other hand, for each reduced expression $w_0=s_{i_1}\ldots s_{i_q}$ we can consider the
parametrization of the canonical basis by means of string coordinates. For any two such reduced decompositions, the transition maps between the two sets of string coordinates are given by 
piecewise linear expressions. As shown by Berenstein and Zelevinsky,  these expressions are the tropicalizations of the transition maps between the two parametrizations of the  double Bruhat cell $L^{w_0,e}$, associated  to the Langlands dual group.
For example, in type $A_2$ (which is its own Langlands dual)
let $(x_1,x_2,x_3)$ be the Kashiwara, or string, coordinates of the canonical basis, 
using the reduced 
decomposition $w_0 =s_1s_2s_1$, and $(y_1,y_2,y_3)$ the ones corresponding
to $w_0=s_2s_1s_2$. 
The transition map between the two is given by
$$\begin{array}{rcl} y_1&=&x_3\vee(x_2-x_1)\\ y_2&=&x_1+x_3\\ y_3&=&x_1\wedge (x_2-x_3)\end{array}$$
which is the tropicalization of (\ref{tropi}).

We will show how some elementary considerations on the Sturm-Liouville
equation, and the method of variation of constants, together with the
Littelmann path model explain these connections.
\subsection{Sturm-Liouville equations}
We consider the  Sturm-Liouville equation
\begin{equation}\label{ Sturm_L}\varphi''+q\varphi=\lambda \varphi
\end{equation}
on some  interval of the real line, say $[0,T]$ to fix notations.
In general there exists no closed form for the solution to such an equation.
 However, if one solution $\varphi_0$ is known,
  which does not vanish in the interval then all the 
  solutions can be found by quadrature. Indeed 
using for example the "method of variation of constants" 
 one sees that every
other solution $\varphi$ 
of this equation in the same interval can be written in the form
$$\varphi(t)=u\varphi_0(t)+v\varphi_0(t)\int_0^t\frac{1}{\varphi_0^2(s)}ds$$
for some constants $u,v$.
If this new solution does not vanish in the
interval $I$,  we can use it to generate other solutions of the equation by
the same kind of formula.
This leads us to investigate the composition  of two maps of the form 
 $$E_{u,v}:\varphi\mapsto
 u\varphi(t)+v\varphi(t)\int_0^t\frac{1}{\varphi^2(s)}ds\label{action}$$
 acting on non vanishing continuous functions.
 It is easy to see, using integration by parts,  that whenever 
 the composition is well defined, one has
 $$E_{u,v}\circ E_{u',v'}=E_{uu',uv'+v/u'}$$ therefore these maps define a 
 partial  right
 action of   the group of unimodular 
 lower triangular matrices 
 $$\begin{pmatrix}u&0\\v&u^{-1}\end{pmatrix}$$
 on the set of continuous paths which do not vanish in $I$.  
 Of course this is equivalently a partial left action of the upper triangular
 group, but for reasons which will soon appear we choose this formulation.
 In particular if we start from $\varphi$ and construct 
 $$\psi(t)=u\varphi(t)+v\varphi(t)\int_0^t\frac{1}{\varphi^2(s)}ds$$ which does not
 vanish on $[0,T]$, then $\varphi$ can be recovered from $\psi$ by the formula
 $$\varphi(t)=u^{-1}\psi(t)-v\psi(t)\int_0^t\frac{1}{\psi^2(s)}ds.$$
 Coming back to the Sturm-Liouville equation, let $\eta,\rho$ be a fundamental
 basis of solutions at 0, namely $\eta(0)=\rho'(0)=1$, $\eta'(0)=\rho(0)=0$.
 Then in the two-dimensional space spanned by $\rho,\eta$ the transformation is
 given by $$(x,y)\mapsto (ux,uy+v/x)$$ and it is defined on $x\ne 0$. 
 Again it is easy to check, using
 this formula,
 that this defines a right action of the lower triangular group.
 
  Let us now investigate the limiting case $u=0$, which gives (assuming $v=1$
  for simplicity)
  \begin{equation}\label{TPit}{\mathcal T}\varphi(t)=\varphi(t)\int_0^t\frac{ds}{\varphi(s)^2}.\end{equation}
  This map provides a ``geometric lifting" of the one-dimensional
  Pitman transformation. Indeed 
 set $\varphi(t)=e^{a(t)}$, then using Laplace's method
  \begin{equation}\lim_{\varepsilon\to 0_+}
  \varepsilon\log\left(e^{a(t)/\varepsilon}\int_0^t
  e^{-2a(s)/\varepsilon}ds\right)=a(t)-2\inf_{0\leq s\leq t}a(s).\end{equation}
  This time the function $\varphi$ cannot be recovered from ${\mathcal T}\varphi$.
  If we
  compute the same transformation with 
$\varphi_v(t):=\varphi(t)(1+v\int_0^t\frac{1}{\varphi(s)^2}ds)$ we get
$$
\begin{array}{rcl}{\mathcal T}\varphi_v(t)&=&\varphi_v(t)
\int_0^t\frac{1}{\varphi_v(s)^2}ds\\
&=&\varphi(t)(1+v\int_0^t\frac{1}{\varphi(s)^2}ds)\left(\frac{1}{
v}-\frac{1}{
v(1+v\int_0^t\frac{1}{\varphi(s)^2}ds)}\right)\\
&=&\varphi(t)\int_0^t\frac{1}{\varphi(s)^2}ds\\
&=&{\mathcal T}\varphi(t).
\end{array}
$$
This is of course not surprising, since 
${\mathcal T}\varphi$ vanishes at 0, it thus belongs to
a one-dimensional subspace of the space of solutions to the Sturm-Liouville
equation, and ${\mathcal T}$ is not invertible.
In order to recover the function $\varphi$ from $\psi={\mathcal T}\varphi$ we thus need to 
specify some real number. A convenient choice is to impose the value of 
$$\xi=\int_0^T\frac{1}{\varphi(s)^2}ds=
\frac{\psi(T)}{\varphi(T)}.$$
With this we can compute
$$\int_t^T\frac{1}{\psi(s)^2}ds=
\frac{1}{\int_0^t\frac{1}{\varphi(s)^2}ds}-
\frac{1}{\int_0^T\frac{1}{\varphi(s)^2}ds}=
\frac{\varphi(t)}{\psi(t)}-\frac{1}{\xi}.$$
\begin{prop}Assume that $\psi={\mathcal T}\varphi$ for some nonvanishing
$\varphi$, then the set ${\mathcal T}^{-1}(\psi)$ can be parametrized by
 $\xi\in]0,+\infty[$.
For each such $\xi$ there exists a unique $\varphi_\xi\in{\mathcal T}^{-1}(\psi)$
such that $\xi=\int_0^T\frac{1}{\varphi_\xi(s)^2}ds$, given by 
$$\varphi_\xi(t)=\psi(t)\left(\frac{1}{\xi}+
\int_t^T\frac{1}{\psi(s)^2}ds\right).$$
\end{prop}
Identifying the positive halfline with the Weyl chamber for $SL_2$, we see that
sets of the form ${\mathcal T}^{-1}(\psi)$ are geometric liftings of 
the Littelmann modules for $SL_2$.
 The formula in the proposition gives a geometric
lifting of the operator ${\HH}^x$ since 
$${\HH}^xa(t)=a(t)-x\wedge2\inf_{t\leq s\leq T}a(s)=\lim_{\varepsilon\to0_+}
\varepsilon\log\left(e^{a(t)/\varepsilon}(e^{-x/\varepsilon}+\int_t^Te^{-2a(s)/\varepsilon}ds)\right).
$$

We shall now find the geometric liftings of the
Littelmann operators. For this we have, knowing an element $\varphi_{\xi_1}\in
{\mathcal T}^{-1}(\psi)$, to find the solution corresponding to $\xi_2$. 
Since 
$$\varphi_{\xi_i}(t)=\psi(t)\left(\frac{1}{\xi_i}+
\int_t^T\frac{1}{\psi(s)^2}ds\right)\qquad i=1,2$$
one has
$$\varphi_{\xi_1}=\varphi_{\xi_2}+\psi(\frac{1}{\xi_1}-\frac{1}{\xi_2})
=\varphi_{\xi_2}\left(1+(\frac{1}{\xi_1}-
\frac{1}{\int_0^T\frac{1}{\varphi_{\xi_2}(s)^2}ds}
)\int_0^t\frac{1}{\varphi_{\xi_2}(s)^2}ds\right).$$
Using Laplace
method again one can
 recover the formula for the operators ${\mathcal E}^x_{\alpha}$,
see definition \ref{littelmanntransform}.
\subsection{A $2\times 2$ matrix interpretation}

We shall now recast the above computations using a $2\times 2$ matrix
differential equation of order one,
and the Gauss decomposition of matrices. This will allow us 
in the next section to extend these
constructions to higher rank groups.

Let $N_+$ be the
nilpotent group of  upper triangular invertible $2\times 2$ matrices, let $N_-$ be the
corresponding group of lower triangular matrices, and $A$ the group of diagonal
matrices, then  an
invertible
$2\times 2$ matrix $g$ has a Gauss decomposition if it can be written as
$g=[g]_-[g]_0[g]_+$ with $[g]_-\in N_-, [g]_0\in A$ and $[g]_+\in N_+$.
  We will use also the decomposition
$g=[g]_-[g]_{0+}$ with $[g]_{0+}=[g]_0[g]_+\in B=AN_+$.
The condition for such a decomposition to exist
is exactly that the upper left coefficient of the matrix $g$ be non zero.

Let  us consider a smooth path $a:[0,T]\to{\R}$, such that $a(0)=0$,
 and let  the matrix  $b(t)$ be the solution to
\begin{equation}
\label{eqdiff}\frac{db}{dt}=\begin{pmatrix}\frac{da}{dt}&1
\\0&-\frac{da}{dt}\end{pmatrix}b;\qquad b(0)=Id.
\end{equation}
Then one has
$$b(t)=\begin{pmatrix}e^{a(t)}&e^{a(t)}\int_0^te^{-2a(s)}ds\\
0&e^{-a(t)}\end{pmatrix}.$$
Now let $g=\begin{pmatrix}u&0\\
v&u^{-1}\end{pmatrix}$ and consider the Gauss decomposition of the matrix
$$bg=\begin{pmatrix}ue^{a(t)}+ve^{a(t)}\int_0^te^{-2a(s)}ds
&u^{-1}e^{a(t)}\int_0^te^{-2a(s)}ds\\ ve^{-a(t)}& u^{-1}e^{-a(t)}\end{pmatrix}.$$
One finds that
$$[bg]_-=\begin{pmatrix}1&0\\\frac{ve^{-a(t)}}
{ue^{a(t)}+ve^{a(t)}\int_0^te^{-2a(s)}ds}&
1\end{pmatrix}$$
and 
$$[bg]_{0+}=\begin{pmatrix}ue^{a(t)}+ve^{a(t)}\int_0^te^{-2a(s)}ds &u^{-1}e^{a(t)}\int_0^te^{-2a(s)}ds\\0&
(ue^{a(t)}+ve^{a(t)}\int_0^te^{-2a(s)}ds)^{-1}\end{pmatrix}.$$
One can check the following proposition.
\begin{prop}
The upper triangular matrix $[bg]_{0+}$  satisfies the 
differential equation
$$\frac{d}{dt}[bg]_{0+}=\begin{pmatrix}\frac{d}{dt}T_{u,v}a(t)&1\\0&-\frac{d}{
dt}T_{u,v}a(t)\end{pmatrix}[bg]_{0+}$$
where $T_{u,v}a(t)=\log (E_{u,v}e^{a(t)})$.
\end{prop}
This equation is of 
the same kind as the equation (\ref{eqdiff})
 satisfied by the original matrix $b$, but with a
different initial point. The right action 
$E_{u,v}$ is thus obtained by taking the matrix solution to (\ref{eqdiff}),
multiplying it on the right by $g=\begin{pmatrix}u&0\\v&u^{-1}\end{pmatrix}$
 and looking at the 
 diagonal part of the Gauss decomposition of the
resulting matrix.
Actually in this way the partial action $T_{u,v}$
 extends to a partial action $T_g$ of the whole group of invertible 
 real $2\times 2$ matrices. One starts from the path $a$, constructs the matrix
 $b$ by the differential equation and then takes the 0-part in the Gauss
 decomposition of $bg$. This yields a path $T_ga$. The statement of the
 proposition above remains true for $[bg]_{0+}$. The importance of this
 statement is that one can iterate the procedure and see that
 $T_{g_1g_2}=T_{g_2}\circ T_{g_1}$ when defined.

Consider now the element $s=\begin{pmatrix}0&-1\\1&0\end{pmatrix}$, then 
 $$T_sa(t)=a(t)+\log\left( \int_0^te^{-2a(s)}ds\right).$$
This is the geometric lifting of the Pitman operator obtained in (\ref{TPit}).
In the next section we shall extend these considerations to groups of
higher rank.
\subsection{Paths in the Cartan algebra}\label{path-cartan}
We work now in the general framework of the beginning of section \ref{DBCSC}.

 One has the usual decomposition  
${\mathfrak g}={\mathfrak n}_-+{\mathfrak a}+{\mathfrak n}_+$.
 Correspondingly  there is a Gauss decomposition $g=[g]_{-}[g]_0[g]_+$ with
 $[g]_-\in N_-,[g]_0\in A, [g]_+\in N$, defined
on an open dense subset. We denote by $[g]_{0+}=[g]_0[g]_+$ the $B=AN_+$
 part of
the decomposition.

The following is easy to check, and provides a useful characterization of 
 the vector space generated by the $e_i$.
\begin{lemma}\label{nei}Let  $n\in {\mathfrak n}_+$, then one has 
$[h^{-1}nh]_{+}=n$  for all $h\in N_-$
if and only if $n$ belongs to the vector space generated by the $e_i$.
\end{lemma}
Let $a$ be a path in the Cartan algebra $\mathfrak a$ 
and let $b$ be a solution to the
equation
$$\frac{d}{dt} b=(\frac{d}{dt}a+n)b$$
where  $n\in \oplus_i\mathbb C e_i$.
\begin{prop} Let $g\in G$, and assume that 
$bg$ has a Gauss decomposition, then the upper part
 $[bg]_{0+}$ in the Gauss
decomposition of $bg$ satisfies the equation 
\begin{equation}
\frac{d}{dt}[bg]_{0+}=(\frac{d}{dt}T_g a+n)[bg]_{0+}\label{eqdiff2}
\end{equation}
where $T_ga(t)$ is a path in the Cartan algebra.

\end{prop}

\proof Let us write the equation
$$\frac{d}{dt}([bg]_-[bg]_{0+})=(\frac{d}{dt}a+n)[bg]_-[bg]_{0+}$$
in the form
$$[bg]_{-}^{-1}\frac{d}{dt}[bg]_{-}=[bg]_{-}^{-1}(\frac{d}{dt}a+n)[bg]_{-}
-\frac{d}{dt}[bg]_{0+}[bg]^{-1}_{0+}.$$
Since the left hand side of this equation is lower triangular, 
the right hand  side has zero upper
triangular part therefore, 
 by  
lemma \ref{nei}
$$n=\left[[bg]_{-}^{-1}(\frac{d}{dt}a+n)[bg]_{-}\right]_{+}=
\left[\frac{d}{dt}[bg]_{0+}[bg]^{-1}_{0+}\right]_{
+}$$
therefore there exists a path $T_g a$ such that equation (\ref{eqdiff2}) holds.
\qed

\medskip

We now assume that $$n=\sum_in_ie_i$$ with all $n_i>0$.
When $g=\bar s_i$ is a fundamental reflection, one gets a geometric lifting of
the Pitman
operator 
$$T_{s_i}a(t)=a(t)+\log\left(\int_0^te^{-\alpha_i(a(s))}ds
\right)\alpha^{\vee}_i$$
associated with the dual root system, i.e. 
$$\lim_{\varepsilon\to 0}\varepsilon T_{s_i}(\frac{1}{\varepsilon}a)={\mathcal
P}_{\alpha_i^{\vee}}a.$$
Thanks to the above proposition, one can prove that these geometric liftings
satisfy the braid relations, and $T_w$ provides a geometric lifting
of the 
Pitman operator $\pw$ for all $w\in W$.
  
 Analogously the Littelmann raising and lowering operators also
  have geometric liftings.
 
\subsection{Reduced double Bruhat cells}\label{redbru}
In this section we show how our considerations on Littelmann's
 path model allow us to
make the connection with the work of Berenstein and Zelevinsky \cite{beze2}.
 We consider a path $a$ on the Cartan Lie algebra, with $a(0)=0$, then belongs to
the Littelmann module $L_{\PP_{w_0}a}$.

Consider the solution $b$ to $\frac{d}{dt} b=(\frac{d}{dt}a+n)b$, $b(0)=I$.
Then $[[b]_+w_0]_{-0}\in L^{w_0,e}$, thus if
\begin{equation}\label{dec}
w_0=s_{i_1}\ldots s_{i_q}
\end{equation}
 is a reduced decomposition, then one has  
$$[[b]_+w_0]_{-0}=x_{-i_1}(r_1)\ldots
x_{-i_q}(r_q)$$
for some uniquely defined
 $r_1(a),\ldots, r_q(a)> 0$ (see \cite{beze2}).
 Let $u_k(a)=r_k(a)e^{-\alpha_{i_k}(a(T)}$.
\begin{theorem}\label{the:string}
Let $(x_1,\ldots, x_q)$ be the string parametrization of $a$ in $L_{\PP_{w_0}a}$, associated with the decomposition (\ref{dec}), then 
$$(x_1,\ldots, x_q)=\lim_{\varepsilon \to 0 } \varepsilon (\log u_1(a/\varepsilon ),\ldots, \log u_q(a/\varepsilon )).$$
\end{theorem}
\proof
When
we multiply $b$ on the right by $\bar s_{i_1}$, and take its Gauss decomposition
$$[bs_{i_1}]_-[bs_{i_1}]_0[bs_{i_1}]_+=[b]_0[b]_+s_{i_1}$$
then $$[b]_+s_{i_1}[bs_{i_1}]_+^{-1}=[b]_0^{-1}
[bs_{i_1}]_-[bs_{i_1}]_0\in
Ns_{i_1}N\cap B_-L^{s_{i_1},e}$$ and 
$$[b]_+s_{i_1}[bs_{i_1}]_+^{-1}=x_{-i_1}(r_1)$$
 for some  $r_1$. In fact, using our formula for
 Littelmann operators,
  $$r_1=e^{\alpha_1(a(T))}\int_0^Te^{-\alpha_1(a(s))}ds.$$
 Comparing with (\ref{kash}) we see that $r_1e^{-\alpha_1(a(T))}$
  gives a geometric lifting of the first  string
coordinate for the Littelmann module.
We can continue the process starting from $[bs_{i_1}]_+$, to get
$$[bs_{i_1}]_+s_{i_2}[bs_{i_1}s_{i_2}]_+^{-1}=x_{-i_2}(r_2)$$
(using the fact that $[g_1g_2]_+=[[g_1]_+g_2]_+$ for $g_1,g_2\in G$)
obtaining 
 successive
decompositions
$$[b]_+s_{i_1}\ldots s_{i_k}[bs_{i_1}\ldots s_{i_k}]_+^{-1}=x_{-i_1}(r_1)\ldots
x_{-i_k}(r_k).$$
This gives  the  coordinates of $[[b]_+w_0]_{-0}\in L^{w_0,e}$, which
are thus seen to correspond to the string coordinates by a geometric
lifting.
\qed

\section{Appendix}
This appendix is devoted to the proof of theorem \ref{thmuniq}.
\begin{lemma}\label{clnfcry} If  $B(\lambda), \lambda \in \bar C,$ 
is a closed normal family of highest weight continuous crystals 
 then for each $\lambda, \mu \in \bar C$ such that $\lambda \leq \mu$ there exists an injective map 
$\Psi_{\lambda,\mu}:B(\lambda)\to B(\mu)$ with the following properties
\begin{enumerate}[(i)]
\item
 $\Psi_{\lambda,\mu}(b_\lambda)=b_\mu,$

\item  $\Psi_{\lambda,\mu}e_\alpha^r(b)=e_\alpha^r\Psi_{\lambda,\mu}(b)$, 
for all $b \in B(\lambda),\alpha\in\Sigma,r\geq 0$,

\item  $\Psi_{\lambda,\mu}f_\alpha^r(b)=f_\alpha^r\Psi_{\lambda,\mu}(b)$ 
if $f_\alpha^r(b)\in B(\lambda)$.

\end{enumerate}
\end{lemma}
\proof Let $\nu = \mu-\lambda$. First consider the map 
$\phi_{\lambda,\mu}:B(\lambda)\to B(\lambda)\otimes B(\nu)$ given by 
$\phi_{\lambda,\mu}(b)=b\otimes b_\nu$, when $b \in B(\lambda)$.
Since $b_\nu$ is a highest weight $\varepsilon_\alpha(b_\nu)=0$. By normality, for all $b \in B(\lambda), 
\varphi_\alpha(b) \geq 0$.  Therefore
$\sigma:=\varphi_\alpha(b)-\varepsilon_\alpha(b_\nu)= \varphi_\alpha(b)\geq 0$. By definition, this implies that
$\varepsilon_\alpha(b\otimes b_\nu)=\varepsilon_\alpha(b)$, $
\varphi_\alpha(b\otimes b_\nu)=\varphi_\alpha(b)$, 
$wt(b\otimes b_\nu)=wt(b)+\nu.$ Using (\ref{sigmapos}) we see also that, for $r \geq 0$,
$
e_\alpha^r(b\otimes b_\nu)=
 e_\alpha^{r}b\otimes b_\nu$ and that, when $f_\alpha^r(b)\in B(\lambda)$, $r \leq \varphi_\alpha(b) = \sigma$ by normality, and therefore $f_\alpha^r(b\otimes b_\nu)=
f_\alpha^{r}b\otimes b_\nu$. Since the family is closed there is an isomorphim $i_{\lambda,\mu}: \FF(b_\lambda\otimes b_\nu) \to B(\mu)$. One has $i_{\lambda,\mu}(b_\lambda\otimes b_\nu)=b_{\mu}.$ One can take $\Psi_{\lambda,\mu}=i_{\lambda,\mu}\circ \phi_{\lambda,\mu}$. $\square$
\medskip

The family $\Psi_{\lambda,\mu}$ constructed above satisfies
 $\Psi_{\lambda, \lambda} =id$ and, when $\lambda \leq \mu\leq\nu$,
 $\Psi_{\mu, \nu} \circ \Psi_{\lambda,\mu}= \Psi_{\lambda,\nu}$,
so that we can consider the direct limit $B(\infty)$ of 
the family $B(\lambda), \lambda \in \bar C,$ with the injective maps
 $\Psi_{\lambda, \mu}: B(\lambda) \to  B(\mu), \lambda \leq\mu$. 
 Still following Joseph \cite{Joseph-notes},
  we define a crystal structure on $B(\infty)$. 
\begin{prop}
The direct limit $B(\infty)$  is a  highest weight upper normal continuous
 crystal with highest weight $0$.
\end{prop}
\proof By definition, the direct limit $B(\infty)$ is the quotient set
 $B/\sim$ where $B=\cup_{\lambda\in \bar C}B(\alpha)$ is the disjoint union of
  the $B(\lambda)'s$ and where $b_1 \sim b_2$ for $b_1\in B(\lambda), b_2
   \in B(\mu)$, when there exists a  $\nu \in \bar C$ such that 
   $\nu \geq\lambda, \nu\geq�\mu$ and
    $\Psi_{\lambda,\nu}(b_1)=\Psi_{\mu,\nu}(b_2)$. 
Let $\bar b$ be the image in  $B(\infty)$ of $b \in B$. 
 If $b \in B(\lambda)$, then we define
$wt(\bar b)=wt(b)-\lambda$, 
$\varepsilon_\alpha(\bar b)=\varepsilon_\alpha( b),$
 $\varphi_\alpha(\bar b)=\varepsilon_\alpha( \bar b)+\alpha^\vee(wt(\bar b))$ 
 and, when $r \geq 0$, 
$e^r_{\alpha}(\bar b)=\overline{ e^r_{\alpha}(b)}$. 
These do not depend on $\lambda$, since if
 $\mu \geq\lambda$ and
 $b'=\Psi_{\lambda,\mu}(b)$, then one has 
 $\bar b'=\bar b$ and $wt(b')=wt(b)+\mu-\lambda$. In order to define
  $ f_\alpha^r(\bar b)$ for $r \geq 0$, let us choose  $\mu \geq\lambda$ large enough to ensure that 
 $$\varphi_\alpha( b')=\varepsilon_\alpha( b')+\alpha^{\vee}(wt(b))+
\alpha^\vee(\mu-\lambda)\geq r.$$
 Then 
 $f^r_\alpha b'\neq {\bf 0}$ by normality and we  define
$f^r\bar b=\overline{ f^rb'}$. Again this does not depend on $\mu$.
Using the lemma we check that this  defines a crystal stucture on $B(\infty)$. 
Each  $\Psi_{\lambda, \mu}$, $\lambda \leq \mu$,
commutes with the $e_\alpha^r, r \geq 0 $. 
This implies that $B(\infty)$ is upper normal.
 Since each $B(\lambda)$ is a highest weight crystal, 
 $B(\infty)$ has also this property. $\square$
\medskip

We will denote $b_\infty$ the unique element of $B(\infty)$ of weight 0. 
Note  that $B(\infty)$ is not lower normal.
 For instance,  \begin{equation}\label{binfty}\varphi_\alpha(b_\infty)=0, f(b_\infty)\neq {\bf 0}, \mbox{ for all } f\in \FF.\end{equation}
For $\lambda \in \bar C$ we define the  crystal $S(\lambda)$ 
as the set with a unique element $\{s_\lambda\}$ and the maps
$wt(s_\lambda)=\lambda, \varepsilon_\alpha(s_\lambda)=
-\alpha^\vee(\lambda),\varphi_\alpha(s_\lambda)=0$ and
$e^r_\alpha(s_\lambda)=\bf 0$ when $r\neq 0.$
\begin{lemma} The map
$$\Psi_\lambda:b\in B(\lambda) \mapsto \bar b \otimes 
s_\lambda \in B(\infty)\otimes S(\lambda)$$
is a crystal embedding.
\end{lemma}
\proof Let $b \in B(\lambda)$, then
$$wt(\Psi_\lambda(b))=wt(\bar b \otimes s_\lambda)=wt(\bar b)+wt(s_\lambda)=
wt(b)-\lambda+\lambda=wt(b).$$
Let $\sigma=\varphi_\alpha(\bar b)-\varepsilon_\alpha(s_\lambda)$. 
Then $\sigma=\varphi_\alpha(b)$ since 
$\varepsilon_\alpha(s_\lambda)=-\alpha^\vee(\lambda)$ and
 $\varphi_\alpha(\bar b)=\varphi_\alpha(b)-\alpha^\vee(\lambda)$. 
 Thus $\sigma \geq 0$ by normality of $B(\lambda)$.
  By the definition of the tensor product,  
 this implies that 
$$\varepsilon_\alpha(\Psi_\lambda(b))=
\varepsilon_\alpha(\bar b \otimes s_\lambda)=\varepsilon_\alpha(\bar b)=
\varepsilon_\alpha( b),$$
 thus $\varphi_\alpha(\Psi_\lambda(b))=\varphi_\alpha( b)$.
 Furthermore, since $\sigma \geq 0$,
$$e^r_\alpha(\Psi_\lambda(b))
=e^r_\alpha(\bar b \otimes s_\lambda)=
e^{\max(r,-\sigma)}_\alpha(\bar b)\otimes 
 e^{\min(r,-\sigma)+\sigma}s_\lambda.$$
When $r \geq -\sigma$, this is equal to
  $e^r_\alpha(\bar b)\otimes  s_\lambda=\Psi_\lambda(e^r_\alpha(b))$. If
 $r < -\sigma$ then $e^r_\alpha(\Psi_\lambda(b))= e^{-\sigma}_\alpha( \bar b) \otimes e^{r+\sigma}_\alpha( s_\lambda)={\bf 0}$, since $e^{s}_\alpha( s_\lambda)={\bf 0}$ when $s\neq 0$, 
 and on the other hand, 
 $e^r_\alpha(b)={\bf 0}$ by normality.
  Thus $\Psi_\lambda(e^r_\alpha(b))={\bf 0}$.
$\square$\medskip

If $f=f_{\alpha_n}^{r_n}\cdots f_{\alpha_1}^{r_1} \in \FF$, 
we say that $f'\in F$ is extracted from $f$ if  
$f'=f_{\alpha_n}^{r_n'}\cdots f_{\alpha_1}^{r_1'}$ with 
$0 \leq r_k'\leq r_k, k=1,\cdots,n$. Recall  the definition of 
$B_\alpha=\{b_\alpha(t), t \leq 0\}$ given in Example  \ref{exba}.

\begin{lemma}\label{lem_indm} Let $f \in \FF$ and $\alpha \in \Sigma$, 
then there exists $f'$ extracted from $f$ and $t \geq 0$ such that
$$ f(b_\infty \otimes b_\alpha(0)) = f' b_\infty \otimes b_\alpha(-t).$$
Moreover if $\lambda\in \bar C$ is such that  $\alpha^\vee(\lambda)=0$ and 
$\beta^\vee(\lambda)$
 large enough for all $\beta\in \Sigma-\{\alpha\}$, then for 
  $\mu\in \bar C$, for the same $f'\in \FF$ and $t \geq 0$,
   $$ f(b_\lambda \otimes b_\mu) = f' b_\lambda \otimes f_\alpha^t b_\mu.$$
\end{lemma}
\proof
The first part follows easily from the definition of the tensor product.
 Let  $\lambda\in \bar C$ such that  $\alpha^\vee(\lambda)=0$, $\mu \in \bar C, \beta\in \Sigma -\{\alpha\}$ and $r \geq 0$. 
If, for some $s >0$, one has $ e_\beta^s(f_\alpha^rb_\mu)\neq {\bf 0}$
 then $wt(e_\beta^s(f_\alpha^rb_\mu))=\mu+s\beta-r\alpha$
  is in $\mu -\bar C$ (since $\mu$ is a highest weight). 
 This is not possible because $\beta^\vee(s\beta-r\alpha) \geq s\beta^\vee(\beta) >0$. Therefore, by normality,
$\varepsilon_\beta(f_\alpha^rb_\mu)=0.$
On the other hand, for all 
 $f=f_{\alpha_n}^{r_n}\cdots f_{\alpha_1}^{r_1} \in \FF$,
$$\varphi_\beta(fb_\lambda)=\beta^\vee(wt(fb_\lambda))
+\varepsilon_\beta(fb_\lambda) \geq
\beta^\vee(wt(fb_\lambda))=\beta^\vee(\lambda)- \sum_{k=1}^n{r_k}
\beta^{\vee}(\alpha_k).$$
Let $\sigma=\varphi_\beta(fb_\lambda)-\varepsilon_\beta(f_\alpha^rb_\mu)=
\varphi_\beta(fb_\lambda)$ and $ s \geq 0$. Then
$$\sigma = \varphi_\beta(fb_\lambda) \geq \beta^\vee(\lambda)- 
\sum_{k=1}^n{r_k} \beta^{\vee}(\alpha_k).$$
 If $\beta^\vee(\lambda)$ is large enough, then 
 $\sigma \geq \max(s,0)$ which implies, see (\ref{sigmapos}), that 
\begin{equation}\label{flb}f_\beta^{s}(fb_\lambda \otimes f_\alpha^r b_\mu)=(f_\beta^{s}fb_\lambda) \otimes f_\alpha^r b_\mu.\end{equation}
On the other hand, 
$\varphi_\alpha(b_\lambda)=\alpha^\vee(\lambda)+\varepsilon_\alpha(b_\lambda)=0$, since 
$\varepsilon_\alpha(b_\lambda)=0$ by normality. 
We also know that $\varphi_\alpha(b_\infty)=0$, see (\ref{binfty}), hence
$$\varphi_\alpha(fb_\lambda)=
\varphi_\alpha(b_\lambda)-\sum_{k=1}^n{r_k}\alpha^{\vee}(\alpha_k)=
\varphi_\alpha(b_\infty)-\sum_{k=1}^n{r_k}\alpha^{\vee}(\alpha_k)=\varphi_\alpha(fb_\infty).$$
Thus $\sigma=\varphi_\alpha(fb_\infty)$ and  does not depend on $\lambda$. 
It follows  that the following decomposition is independent of $\lambda$:
\begin{equation}\label{fla}f_\alpha^s(fb_\lambda\otimes f_\alpha^r b_\mu)=f_\alpha^{\sigma\wedge s}fb_\lambda\otimes f_\alpha^{r+s-\sigma\wedge s}b_\mu.\end{equation}
Using (\ref{flb}) and (\ref{fla}), it is now easy to prove the lemma by induction on $n$, proving first the second assertion. $\square$
\begin{prop}\label{gammaalpha} For each
 simple root $\alpha$, there is a crystal embedding  
 $\Gamma_\alpha:B(\infty)\to B(\infty)\otimes B_\alpha$ such that
  $\Gamma_\alpha(b_\infty)=b_\infty\otimes b_\alpha(0)$.
\end{prop}
\proof Let us show that the expression
\begin{equation}\label{def_Gamma}
\Gamma_\alpha(f b_\infty)=f(b_\infty\otimes b_\alpha(0)),\;\;
 f \in \FF,\end{equation}
defines the morphism $\Gamma_\alpha$.
 First we check that it is well defined.
  By definition,  $fb_\infty=\overline{f   b}_\nu$ for all $\nu\in \bar C$ 
  such that $\overline{fb}_\nu\neq {\bf 0}$.

  Let us choose $\lambda$ as in lemma \ref{lem_indm}. 
  For $\mu \in \bar C$ large enough, $\overline{f b}_{\lambda+\mu}\neq   {\bf 0}$.
   Let us write
   $$\overline{f b}_{\lambda+\mu}=f(\bar b_\lambda\otimes\bar b_\mu)=
   \overline{f' b}_\lambda \otimes\overline{ f_\alpha^t b}_\mu.$$ 
   Then  $f'$ and $t$
     depend only on $fb_{\lambda+\mu}$, which by definition depends only on
     $fb_\infty$. By lemma \ref{lem_indm},
$$f(b_\infty\otimes b_\alpha(0))= f'b_\infty \otimes b_\alpha(-t)$$
which depends only on $fb_\infty$ 
(and not on $f$ itself), showing that $\Gamma_\alpha$ 
is well defined on $\FF b_\infty$, and thus on 
$B(\infty)$, since $\FF b_\infty=B(\infty)$.
Notice that 
$f(b_\infty\otimes b_\alpha(0))\neq {\bf 0}$
 since  $f'b_\infty\neq {\bf 0}$.

Let us prove that $\Gamma_\alpha$ is injective. Suppose that  $f(b_\infty\otimes b_\alpha(0))=\tilde f(b_\infty\otimes b_\alpha(0))$  for some $f,\tilde f \in \FF$.
 Using lemma \ref{lem_indm}, 
$$f(b_\infty\otimes b_\alpha(0))= f'b_\infty \otimes b_\alpha(-t)\mbox{    and    }\tilde f(b_\infty\otimes b_\alpha(0))= \tilde f'b_\infty \otimes b_\alpha(-\tilde t).$$
If $\lambda\in \bar C$ is as in this lemma, then
$$f(b_\lambda\otimes b_\mu)= f'b_\lambda \otimes f_\alpha^t (b_\mu)=
\tilde f'b_\lambda \otimes f_\alpha^{\tilde t} b_\mu=
\tilde f(b_\lambda\otimes b_\mu),$$
therefore $fb_{\lambda+\mu}=\tilde fb_{\lambda+\mu}$,
 thus $fb_{\infty}=\tilde fb_\infty$. 
 It is clear that  $\Gamma_\alpha$ commutes with 
  $f_\alpha^r, r \geq 0$. 
  Since $\varepsilon_\alpha(b_\alpha(0))=\varphi_\alpha(b_\infty)=0$,
   $$\varepsilon_\alpha(\Gamma_\alpha(b_\infty))=
   \varepsilon_\alpha(b_\infty\otimes b_\alpha(0))=
   \varepsilon_\alpha(b_\infty),$$
hence, if $f=f_{\alpha_n}^{r_n}\cdots f_{\alpha_1}^{r_1} 
\in \FF$, $$\varepsilon_\alpha(\Gamma_\alpha(f b_\infty))
=\varepsilon_\alpha(f\Gamma_\alpha(b_\infty))=
\varepsilon_\alpha(\Gamma_\alpha(b_\infty))-\sum_{k=1}^n r_k\beta^{\vee}(\alpha_k)
= \varepsilon_\alpha(fb_\infty)).$$ Therefore $\Gamma_\alpha$ commutes with $\varepsilon_\alpha$. It also commutes with $wt$ since $wt(b_\infty)=0$. Let us now consider 
$e^r_\alpha,r \geq 0$.
Let $b\in B(\infty)$. If $e^r_\alpha(b)\neq {\bf 0}$, 
then $$\Gamma_\alpha(b)=\Gamma_\alpha(f^r_\alpha e^r_\alpha(b))=
f^r_\alpha (\Gamma_\alpha(e^r_\alpha(b))\neq {\bf 0}$$
 hence $\Gamma_\alpha(e^r_\alpha(b))=e^r_\alpha(\Gamma_\alpha(b))$.
  Suppose now that $e^r_\alpha(b)= {\bf 0}$. Since $B(\infty)$
   is upper normal,
  one has
   $\varepsilon_\alpha(b)= 0$, 
   hence $\varepsilon_\alpha(\Gamma_\alpha(b))= 0$.
    By the lemma, there is $f'\in \FF$ and $t \geq 0$ such that   $\Gamma_\alpha(b)=\Gamma_\alpha(b)=f'b_{\infty}\otimes b_\alpha(-t).$ Therefore
$$0 = \varepsilon_\alpha(\Gamma_\alpha(b))\geq \varepsilon_\alpha(f'b_\infty) \geq 0.$$
By upper normality this implies that $e^r_\alpha(f'b_\infty)={\bf 0}$, hence $$e^r_\alpha(\Gamma_\alpha(b))=  e^r_\alpha(f'b_{\infty}\otimes b_\alpha(-t))=(e^r_\alpha f'b_{\infty})\otimes b_\alpha(-t)={\bf 0}.\;\;\; \square$$

The following lemma is clear.
 \begin{lemma}\label{lem_comp}
 Let $B_1,B_2$ and $C$ be three continuous crystals and $\psi: B_1 \to B_2$ be crystal embeddings. Then  $ \tilde \psi: B_1\otimes C\to B_2\otimes C$ defined by
 $\tilde\psi(b\otimes c)= \psi(b)\otimes c$ is a crystal embedding.
  \end{lemma}
\subsection{Uniqueness. Proof of theorem \ref{thmuniq}}\label{prfuniq}
 Recall that $\Sigma$ is the set of simple roots.
  Fix a sequence $A=(\cdots, \alpha_2,\alpha_1)$ 
  of elements of $\Sigma$ such that each simple 
  root occurs infinitely many times and $\alpha_n \neq \alpha_{n+1}$ 
  for all $n \geq 1$. Let $ \hat B(A)$ be the subset of
   $\cdots B_{\alpha_2} \otimes B_{\alpha_1}$
    in which the $k$-th entry differs from $b_{\alpha_k}(0)$
   for   only finitely many $k$. One checks that the rules
      given for the multiple tensor give $\hat B(A)$ 
      the structure of a continuous crystal
       (see, e.g., Kashiwara, \cite{kashbook}, 7.2,
        Joseph \cite{Joseph-book},\cite{Joseph-notes}).
	 Let $ b_A$ be the element of $\hat B(A)$ with entries
	  $b_{\alpha_n}(0)$ for all $n \geq 1$. We denote $B(A)=\FF b_A$.  
 \begin{prop} There exists a crystal embedding $\Gamma$ 
 from $B(\infty)$ onto $ B(A)$ 
 such that $\Gamma(b_\infty)= b_A$.
 \end{prop}
 \proof Let $f \in \FF$. We can write 
 $f=f_{\alpha_k}^{r_k}\cdots f_{\alpha_1}^{r_1}$ where
  $(\cdots, \alpha_2,\alpha_1)= A$ and
   $ r _n \geq 0$ for all $n\geq 1$. By lemma \ref{lem_indm}
 $$\Gamma_{\alpha_1}(f_{\alpha_1}^{r_1}(b_\infty))=f_{\alpha_1}^{r_1}(\Gamma_{\alpha_1} b_\infty)=f_{\alpha_1}^{r_1}(b_\infty\otimes b_{\alpha_1}(0))=b_\infty\otimes b_{\alpha_1}(-r_1)$$
 therefore
$$ \Gamma_{\alpha_1}(f_{\alpha_k}^{r_k}\cdots f_{\alpha_1}^{r_1}b_\infty)=
 (f_{\alpha_k}^{r'_k}\cdots f_{\alpha_2}^{r'_2}b_\infty)\otimes b_{\alpha_1}(-r'_1)$$
for some $r'_1,\cdots, r'_k\geq0$. Similarly,
$$\Gamma_{\alpha_2}(f_{\alpha_k}^{r'_k}\cdots f_{\alpha_2}^{r'_2}b_\infty)=(f_{\alpha_k}^{r''_k}\cdots f_{\alpha_3}^{r''_3}b_\infty)\otimes b_{\alpha_2}(-r''_2)$$ for some $r''_2,r''_3,\cdots,r''_k$. 
If we apply lemma \ref{lem_comp} to $B_1=B(\infty),
 B_2=B(\infty)\otimes B_{\alpha_2},
  \psi=\Gamma_{\alpha_2},C=B_{\alpha_1}$,
   we obtain a crystal embedding 
$$\tilde \Gamma_{\alpha_2}:B(\infty)\otimes B_{\alpha_1}\to B(\infty)\otimes B_{\alpha_2}\otimes B_{\alpha_1}$$ 
such that, for $b\in B(\infty), b_1\in B_{\alpha_1}$
$$\tilde \Gamma_{\alpha_2}(b\otimes b_1)= \Gamma_{\alpha_2}b\otimes b_1.$$
Let $\Gamma_{\alpha_2,\alpha_1}=\tilde \Gamma_{\alpha_2}\circ  
\Gamma_{\alpha_1}:B(\infty)\to B(\infty)\otimes B_{\alpha_2}
\otimes B_{\alpha_1}$, then
\begin{eqnarray*}
 \Gamma_{\alpha_2,\alpha_1}(f_{\alpha_k}^{r_k}\cdots f_{\alpha_1}^{r_1}b_\infty)&=&\tilde \Gamma_{\alpha_2}( f_{\alpha_k}^{r'_k}\cdots f_{\alpha_2}^{r'_2}b_\infty\otimes b_{\alpha_1}(-r'_1))\\
 &=&\Gamma_{\alpha_2}( f_{\alpha_k}^{r'_k}\cdots f_{\alpha_2}^{r'_2}b_\infty)\otimes b_{\alpha_1}(-r'_1)\\
&=&(f_{\alpha_k}^{r''_k}\cdots f_{\alpha_3}^{r''_3}b_\infty)\otimes b_{\alpha_2}(-r''_2)\otimes b_{\alpha_1}(-r'_1).
\end{eqnarray*}
Again, with $\Gamma_{\alpha_3}$ we build  
$\Gamma_{\alpha_3,\alpha_2,\alpha_1}=\tilde 
\Gamma_{\alpha_3}\circ \Gamma_{\alpha_2,\alpha_1}$.
 Inductively we obtain strict morphisms
  $$\Gamma_{\alpha_k,\cdots,\alpha_1}:B(\infty)\to B(\infty)\otimes
   B_{\alpha_k}\otimes \cdots \otimes B_{\alpha_2}\otimes B_{\alpha_1}$$ 
   such that for some  $s_k,\cdots, s_1$
$$\Gamma_{\alpha_k,\cdots,\alpha_1}
(f_{\alpha_k}^{r_k}\cdots f_{\alpha_1}^{r_1}b_\infty)=b_\infty 
\otimes b_{\alpha_k}(-s_k)\otimes \cdots \otimes b_{\alpha_1}(-s_1).$$

Now we can define $\Gamma:B(\infty)\to B(A)$ by the formula
$$\Gamma(f_{\alpha_k}^{r_k}\cdots f_{\alpha_1}^{r_1}b_\infty)=
\cdots \otimes b_{\alpha_{k+n}}(0) \otimes\cdots \otimes b_{\alpha_{k+1}}(0)
 \otimes b_{\alpha_k}(-s_k)\otimes \cdots \otimes b_{\alpha_1}(-s_1). $$
  One checks that this is a
 crystal embedding. $\square$
 
\medskip

This shows that $B(\infty)$ is isomorphic to $B(A)$, which does not depend on the chosen closed family of crystals, and thus proves the uniqueness. 
It also shows that $B(A)$ doest not depend on $A$, as soon
 as a closed family exists.

\tableofcontents

\end{document}